\documentclass[12pt, leqno]{article}

\usepackage{amsmath,amssymb,amsthm,amscd}
\usepackage[all]{xy}
\usepackage{color}

\usepackage[top=25truemm,bottom=25truemm,left=25truemm,right=25truemm]{geometry}

\makeatletter
    
    \@addtoreset{equation}{section}
  \makeatother
  
\theoremstyle{plain}
\newtheorem{thm}{Theorem}[section]
\newtheorem{prop}[thm]{Proposition}
\newtheorem{lem}[thm]{Lemma}
\newtheorem{cor}[thm]{Corollary}
\newtheorem{conj}[thm]{Conjecture}
\theoremstyle{definition}
\newtheorem{defn}[thm]{Definition}
\newtheorem{exa}[thm]{Example}

\renewcommand{\labelenumi}{(\roman{enumi})}

\newcommand{\FRAC}[2]{\leavevmode\kern.1em\raise.5ex\hbox{\the\scriptfont0 #1}\kern-.1em/\kern-.15em\lower.25ex\hbox{\the\scriptfont0 #2}}

\newcommand{\BQ}{{\mathbf{Q}}}

\newcommand{\BZ}{{\mathbf{Z}}}

\newcommand{\dvr}{\mathcal{O}}
\newcommand{\dvrx}{\mathcal{O}_{X}}
\newcommand{\et}{\mathrm{\acute{e}t}}
\DeclareMathOperator{\Hom}{Hom}

\newcommand{\hok}{H^{1}(K,\mathbf{Q}/\mathbf{Z})}

\newcommand{\sw}{\mathrm{sw}}
\newcommand{\ab}{\mathrm{ab}}
\newcommand{\mf}{\mathcal{F}}

\newcommand{\id}{\mathrm{id}}
\newcommand{\pr}{\mathrm{pr}}

\DeclareMathOperator{\shom}{\mathcal{H}om}

\newcommand{\rsw}{\mathrm{rsw}}

\DeclareMathOperator{\cform}{char}

\newcommand{\fil}{\mathrm{fil}'}
\newcommand{\fillog}{\mathrm{fil}}

\DeclareMathOperator{\ord}{ord}

\newcommand{\gr}{\mathrm{gr}'}
\newcommand{\grlog}{\mathrm{gr}}
\newcommand{\cotx}{T^{\ast}X}
\newcommand{\zers}{T^{\ast}_{X}X}
\newcommand{\nori}{T^{\ast}_{D_{i}}X}

\newcommand{\spf}{T^\ast_{x}X}

\newcommand{\Omegax}{\Omega^{1}_{X}}
\newcommand{\dlog}{d\log}
\DeclareMathOperator{\Supp}{Supp}
\DeclareMathOperator{\Char}{Char}

\DeclareMathOperator{\dimtot}{dimtot}

\newcommand{\II}{I\hspace{-.1em}I}
\newcommand{\mII}{\mathrm{\II}}
\newcommand{\mI}{\mathrm{I}}
\newcommand{\mT}{\mathrm{T}}
\newcommand{\mW}{\mathrm{W}}
\newcommand{\mR}{\mathrm{R}}

\newcommand{\mST}{\mathrm{ST}}
\newcommand{\dt}{\mathrm{dt}}
\DeclareMathOperator{\res}{res}

\newcommand{\atk}{\mathbf{A}_{k}^{2}}

\newcommand{\mg}{\mathcal{G}}

\newcommand{\mk}{\mathcal{K}}
\newcommand{\mh}{\mathcal{H}}

\newcommand{\otimesll}{\otimes_{\Lambda}}

\DeclareMathOperator{\Spec}{Spec}
\DeclareMathOperator{\Frac}{Frac}
\DeclareMathOperator{\Proj}{Proj}

\begin{document}
\title
{Characteristic cycle of a rank one sheaf
\\and ramification theory}
\author{YURI YATAGAWA}
\date{}
\maketitle

\begin{abstract}
We compute the characteristic cycle of a rank one sheaf on a smooth surface over a perfect field
of positive characteristic.
We construct a canonical lifting on the cotangent bundle of Kato's logarithmic characteristic cycle using ramification theory and prove the equality of the characteristic cycle and the canonical lifting.
As corollaries, we obtain a computation of the singular support in terms of ramification theory and the Milnor formula for the canonical lifting.
\end{abstract}

\section*{Introduction}
As an analogy between the wild ramification of $\ell$-adic sheaves on a smooth variety
in positive characteristic
and the irregular singularity of partial differential equations on a complex manifold
suggests,
the characteristic cycle of a constructible complex of 
\'{e}tale sheaves on a smooth variety over a perfect field is defined 
to be an algebraic cycle on the cotangent bundle of the variety.
The characteristic cycle is defined by Saito \cite{sa4}
using the singular support defined by Beilinson \cite{be} and
is characterized by the Milnor formula which computes the total dimension
of the space of vanishing cycles as an intersection number of the characteristic cycle.
Further the characteristic cycle satisfies the index formula (\cite[Theorem 7.13]{sa4})
which computes the Euler characteristic as the intersection number of the characteristic cycle
with the zero section of the cotangent bundle when the variety is projective.

We consider a computation of the singular support and the characteristic cycle of an \'{e}tale constructible sheaf.
Let $X$ be a smooth variety over a perfect field $k$ of characteristic $p>0$
and $D$ a divisor on $X$ with simple normal crossings.
Let $\mf$ be a smooth sheaf of $\Lambda$-modules on $U=X-D$, where $\Lambda$
is a finite field of characteristic $\neq p$.
If the zero extension $j_{!}\mf$ of $\mf$ by the open immersion $j\colon U\rightarrow X$
satisfies the strongly non-degenerate condition (\cite[Subsection 4.2]{sa4}),
then computations of the singular support $SS(j_{!}\mf)$ and the characteristic cycle $CC(j_{!}\mf)$ 
of $j_{!}\mf$ are given by Saito \cite{sa4} using ramification theory.
In a more general setting, computations of $SS(j_{!}\mf)$ and $CC(j_{!}\mf)$ remain to be discovered.

In this manuscript, we assume that $X$ is a smooth surface and that the sheaf $\mf$ is of rank one.
We do not assume any condition on the ramification of $j_{!}\mf$.
The invariants of ramification which we use to compute the singular support $SS(j_{!}\mf)$ and
the characteristic cycle $CC(j_{!}\mf)$ are obtained explicitly by computing Witt vectors defining 
the smooth sheaf $\mf$, their differentials, and the blow-ups at closed points.

For our computations of $SS(j_{!}\mf)$ and $CC(j_{!}\mf)$,
we use Kato's logarithmic characteristic cycle $\Char^{\log}(X,U,\chi)$ (\cite{ka2})
on the logarithmic cotangent bundle $T^{*}X(\log D)$ of $X$ with logarithmic poles along $D$.
Here the symbol $\chi$ denotes the character of $\pi_{1}^{\ab}(U)$ corresponding to $\mf$.
We first construct a canonical lifting $\Char^{K}(X,U,\chi)$ on the cotangent bundle $T^{*}X$ of $X$
of Kato's logarithmic characteristic cycle $\Char^{\log}(X,U,\chi)$ using ramification theory
((\ref{canlift}) and Theorem \ref{totalpullback} (ii)).
Then we prove the equality of the characteristic cycle and the canonical lifting:

\begin{thm}[Theorem \ref{thmmain}]
\label{thmintro}
Assume that $X$ is purely of dimension $2$.
Then we have
\begin{equation}
CC(j_{!}\mf)=\Char^{K}(X,U,\chi).  
\end{equation}
\end{thm}

\noindent
Since the canonical lifting is constructed using ramification theory,
Theorem \ref{thmintro} gives a computation of the characteristic cycle $CC(j_{!}\mf)$ in terms of ramification theory.  
Further, since the singular support $SS(j_{!}\mf)$ is equal to the support of the characteristic cycle $CC(j_{!}\mf)$
(Lemma \ref{suppCC}),
Theorem \ref{thmintro} gives a computation of the singular support $SS(j_{!}\mf)$ in terms of ramification theory.

We give an outline of our proof of Theorem \ref{thmintro}.
Since the strongly non-degenerate condition is satisfied when we restrict the sheaf $j_{!}\mf$ to 
outside a closed subset of codimension $\ge 2$,
Saito's computation of the characteristic cycle in the strongly non-degenerate case
gives a computation by using ramification theory of the irreducible components of the support
of the characteristic cycle $CC(j_{!}\mf)$
whose intersection with the zero section is of codimension $\le 1$ in the variety 
and together with their multiplicities in $CC(j_{!}\mf)$.
By using this computation,
we construct a canonical lifting $\Char^{K}(X,U,\chi)$ to be equal to $CC(j_{!}\mf)$ outside a finitely many closed points
of $D$.
Then it is sufficient to prove the equality $t_{x}=u_{x}$ of the multiplicities
$t_{x}$ and $u_{x}$ of the fiber at each closed point $x$ of $D$ in $\Char^{K}(X,U,\chi)$ and $CC(j_{!}\mf)$.
By using the fact that the characteristic cycle and the canonical lifting are defined \'{e}tale locally,
we reduce the proof to the case where $X$ is projective over an algebraically closed field
by taking a smooth compactification of $X$ if necessary.
Then the characteristic cycle $CC(j_{!}\mf)$ satisfies the index formula.
The canonical lifting also satisfies the index formula (Theorem \ref{logindexformula}) obtained 
by the index formula for the logarithmic characteristic cycle of a clean sheaf (\cite{ka2}, \cite[Corollary 3.8]{saj}) 
and the resolution of non-clean points (\cite[Theorem 4.1]{ka2}).
By using the index formulas and the resolution of non-clean points, we reduce the proof to
the case where the sheaf is clean.
We establish the homotopy invariance of the characteristic cycle of a clean rank one sheaf on a smooth surface 
(Proposition \ref{prophi}), which implies the canonical lifting determines the characteristic cycle.
By using the homotopy invariance, we reduce the proof to the case where $X=\mathbf{P}_{k}^{2}$ and
further to the case where $\mf$ is has a finite set $\Sigma$ 
of closed points of $D$ containing $x$ such that
$(t_{y},u_{y})=(t_{x},u_{x})$ for every $y\in \Sigma$ and that $t_{y}=u_{y}$ for
every closed point $y\notin \Sigma$ of $D$.
Then we obtain the equality $t_{x}=u_{x}$ by using the index formulas of the characteristic cycle and the canonical lifting.

We briefly describe the construction of this paper.
In Section \ref{sklcc}, we briefly recall the definition and properties of the characteristic cycle.
We recall Brylinski-Kato's logarithmic ramification theory and Matsuda's non-logarithmic variant in Subsection \ref{sslram}.
At the end of Subsection \ref{sslram}, we give a form and a partial description of the characteristic cycle of a rank one sheaf 
on a smooth surface using ramification theory, which is already known by Saito.
In Subsection \ref{sscrc}, we study the clean and non-degenerate conditions in detail.
We prove the equality of the zero extension and the direct image by the open immersion
of a clean rank one sheaf which is ramified along every irreducible component of the boundary
in Subsection \ref{sscdoi}.

The definition of Kato's logarithmic characteristic cycle is recalled in Subsection \ref{ssklcc}.
We prove the singular support of a clean rank one sheaf is contained in
the preimage of the support of the logarithmic characteristic cycle by the canonical morphism
from the cotangent bundle to the logarithmic cotangent bundle in Subsection \ref{sslogms}.
We construct a canonical lifting on the cotangent bundle of the logarithmic characteristic cycle
of a rank one sheaf on a smooth surface in Section \ref{scanlift}.
We prove the homotopy invariance of the characteristic cycle of a clean rank one sheaf 
on a smooth surface in Section \ref{shi}.

We state the main theorem Theorem \ref{thmmain}, which gives a computation of the characteristic cycle
of a rank one sheaf on a smooth surface using ramification theory, and 
prove a few corollaries including a computation
of the singular support of the sheaf in Subsection \ref{ssmthm}.
We give a proof of the main theorem in Subsection \ref{sspfsurf}.
We first give a format of the proof of the main theorem 
and then prove the main theorem following the format.

The author would like to express her sincere gratitude to Professor Takeshi Saito for leading her to the study of the characteristic cycle, a lot of discussions, 
a lot of advice for the manuscript, and his patient encouragement.
The author thanks Professor Shigeki Matsuda for answering a question on a comment on cleanliness in his paper \cite{ma}. 
Some parts of the study are done during her stays at IHES supported by the Program for Leading Graduate 
Schools, MEXT, Japan for three months in total in autumn in 2013 and 2014 and
a stay at the Freie Universit\"{a}t Berlin supported by JSPS KAKENHI Grant Number 15J03851 for three weeks in summer in 2015. 
The author would like to thank Professor Ahmed Abbes and Professor H\'{e}l\`{e}ne Esnault for their hospitalities.
This work is partially supported by the Program for Leading Graduate 
Schools, MEXT, Japan, JSPS KAKENHI Grant Number 15J03851, and the CRC 1085 {\it Higher Invariants} at the University of Regensburg.

\tableofcontents
\section{Characteristic cycle}
\label{sklcc}

We briefly recall the characteristic cycle defined in \cite{sa4}.
The characteristic cycle is defined for a constructible complex 
on a smooth scheme over a perfect field using the singular support defined by Beilinson \cite{be}.

Let $k$ be a perfect field of characteristic $p>0$
and $X$ a smooth scheme over $k$.
Let $\Lambda$ be a finite field of characteristic prime to $p$.
Let $T^{\ast}X=\Spec S^{\bullet}\Omega^{1\vee}_{X}$ be the cotangent bundle of $X$
and $T^{\ast}_{X}X$ the zero section of $T^{*}X$.
We say that a closed subset $C$ of a vector bundle over $X$ is \textit{conical} if $C$ is stable under the multiplication by scalars.
For a morphism $h\colon W\rightarrow X$ of smooth schemes over $k$, let $dh\colon T^{\ast}X
\times_{X}W \rightarrow T^{\ast}W$ be the morphism of vector bundles over $W$ 
yielded by $h$
and let $h^{\ast}C=C\times_{X}W$ be the pull-back of a closed conical subset $C$ of $T^{\ast}X$
to $T^{*}X\times_{X}W$.

\begin{defn}[{\cite[1.2]{be}}, {\cite[Definition 3.3, Definition 3.5]{sa4}}]
Let $C$ be a closed conical subset of $T^{\ast}X$.
\begin{enumerate}
\item We say that a morphism $h\colon W\rightarrow X$ of smooth schemes over $k$ is 
$C$-\textit{transversal} at a point $w\in W$ if the subset 
$(dh^{-1}(T^{\ast}_{W}W)\cap h^{\ast}C)\times_{W}w \subset T^{*}X\times_{X}w$ 
is contained in the zero section $T^{\ast}_{X}X\times_{X}w$.

We say that a morphism $h\colon W\rightarrow X$ of smooth schemes over $k$ is 
$C$-\textit{transversal} if the subset $dh^{-1}(T^{\ast}_{W}W)\cap h^{\ast}C\subset T^*X\times_{X}W$ 
is contained in the zero section $T^{\ast}_{X}X\times_{X}W$.

\item Let $h\colon W\rightarrow X$ be a $C$-transversal morphism of smooth schemes over $k$.
We define a closed conical subset $h^{\circ}C\subset T^{\ast}W$ to be the image of
the subset $h^*C\subset T^{*}X\times_{X}W$ by 
$dh\colon T^{\ast}X\times_{X}W\rightarrow T^{\ast}W$. 

\item We say that a morphism $f\colon X\rightarrow Y$ of smooth schemes
over $k$ is $C$-\textit{transversal} at a point $x\in X$ if the subset 
$df^{-1}(C)\times_{X}x\subset T^*Y\times_{Y}x$ is contained in the zero section $T^{\ast}_{Y}Y\times_{Y}x$.

We say that a morphism $f\colon X\rightarrow Y$ of smooth schemes over $k$ is 
$C$-\textit{transversal} if the subset $df^{-1}(C)\subset T^{*}Y\times_{Y}X$ 
is contained in the zero section $T^{\ast}_{Y}Y\times_{Y}X$.

\item We say that a pair $(h,f)$ of morphisms $h\colon W\rightarrow X$ and 
$f\colon W\rightarrow Y$ of smooth schemes over $k$ is $C$-\textit{transversal} 
if $h\colon W\rightarrow X$ is $C$-transversal and if $f\colon W\rightarrow Y$ is 
$h^{\circ}C$-transversal. 
\end{enumerate}
\end{defn}

We note that if a morphism $h\colon W\rightarrow X$ of smooth schemes over $k$ is
$C$-transversal for a closed conical subset $C\subset T^{*}X$, then
the restriction $dh|_{C} \colon C\rightarrow T^{*}W$ of $dh$ to $C$ is finite by \cite[Lemma 3.1]{sa4}.

We say that a complex $\mf$ of \'{e}tale sheaves of $\Lambda$-modules on $X$
is a \textit{constructible} complex of $\Lambda$-modules if the cohomology sheaves 
$\mathcal{H}^{q}(\mf)$ of $\mf$ are constructible for every $q$ 
and are equal to $0$ except for finitely many $q$.

\begin{defn}[{\cite[1.3]{be}}]
\label{defofsingsupp}
Let $\mf$ be a constructible complex of $\Lambda$-modules on $X$.
\begin{enumerate}
\item We say that $\mf$ is \textit{micro-supported} on a closed conical subset 
$C\subset T^{\ast}X$ if for every $C$-transversal pair $(h,f)$ of morphisms 
$h\colon W\rightarrow X$ and $f\colon W\rightarrow Y$ of smooth schemes over $k$
the morphism $f\colon W\rightarrow Y$ is locally acyclic relatively to $h^{\ast}\mf$.
\item The \textit{singular support} $SS(\mf)$ of $\mf$ is the smallest closed conical subset of $T^{\ast}X$ where $\mf$ is micro-supported.
\end{enumerate}
\end{defn}

By \cite[Theorem 1.3]{be}, the singular support $SS(\mf)$ exists for every constructible complex $\mf$ of $\Lambda$-modules on $X$.
Further, if $X$ is purely of dimension $d$, then $SS(\mf)$ is purely of dimension $d$.

\begin{lem}
\label{ssdisttri}
Let $\mf'\rightarrow \mf\rightarrow \mf''\rightarrow$ be a distinguished triangle of constructible complexes of $\Lambda$-modules on $X$.
Let $C$ and $C''$ be closed conical subsets of $T^{\ast}X$
and put $C'=C\cup C''$.
If $\mf$ and $\mf''$ are micro-supported on $C$ and $C''$ respectively,
then $\mf'$ is micro-supported on $C'$.
\end{lem}

\begin{proof}
Let $(h,f)$ be a $C'$-transversal pair
of morphisms $h\colon W\rightarrow X$ and $f\colon W\rightarrow Y$ of smooth schemes over $k$.
By \cite[Theorem 1.5]{be}, we may assume that $Y$ is a smooth curve over $k$.
Let $w$ be a point on $W$ and $\bar{w}$ an algebraic geometric point lying above $w$.
We consider the morphism
\begin{equation}
\xymatrix{
h^{\ast}\mf'_{\bar{w}}\ar[d] \ar[r] &h^{\ast}\mf_{\bar{w}} \ar[r] \ar[d] & h^{\ast}\mf^{\prime \prime}_{\bar{w}} \ar[d] \ar[r] & \\
\varphi_{\bar{w}}(h^{\ast}\mf',f) \ar[r] &\varphi_{\bar{w}}(h^{\ast}\mf,f) \ar[r] & \varphi_{\bar{w}} (h^{\ast}\mf^{\prime \prime},f) \ar[r] & 
} \notag
\end{equation}
of distinguished triangles of complexes on $\bar{w}$.
Here $\varphi_{\bar{w}}$ denotes the stalk of nearby cycle complex at $\bar{w}$.
Since $(h,f)$ is $C$-transversal and $C''$-transversal, the middle and right vertical arrows are isomorphisms.
Hence the left vertical arrow is an isomorphism,
which deduces the assertion.
\end{proof}

\begin{defn}[{\cite[Definition 5.3.1]{sa4}}]
Let $C\subset T^{*}X$ be a closed conical subset.
Let $(h,f)$ be a pair of an \'{e}tale morphism 
$h\colon W\rightarrow X$ and a morphism $f\colon W\rightarrow Y$
to a smooth curve over $k$.
Let $w$ be a closed point of $W$.
We say that $w$ is at most an {\it isolated $C$-characteristic point} of $f$ if
there exists an open neighborhood $V\subset W$ of $w$
such that the pair of restrictions
$h|_{V-\{w\}}\colon V-\{w\} \rightarrow X$ and $f|_{V-\{w\}}\colon V-\{w\}\rightarrow Y$ 
of $h$ and $f$ to $V-\{w\}$ is $C$-transversal.
\end{defn}

\begin{thm}[Milnor formula, {\cite[Theorem 5.9, Theorem 5.18]{sa4}}]
\label{thmsmil}
Assume that $X$ is purely of dimension $d$.
Let $\mf$ be a constructible complex of $\Lambda$-modules on $X$.
Let $C=\bigcup_{a}C_{a}$ be a closed conical subset of $T^{*}X$ purely of dimension $d$
with irreducible components $\{C_{a}\}_{a}$ such that 
$\mf$ is micro-supported on $C$.
Then there exists a unique $\mathbf{Z}$-linear combination 
$CC(\mf)=\sum_{a}m_{a}[C_{a}]$ independent of $C$ such that
for every pair $(h,f)$ of an \'{e}tale morphism 
$h\colon W\rightarrow X$ and a morphism $f\colon W\rightarrow Y$
to a smooth curve $Y$ over $k$ and for every
at most isolated $C$-characteristic point $w\in W$ of $f$
we have 
\begin{equation}
-\dimtot \phi_{w}(h^{\ast}\mf,f)=(h^{\ast}CC (\mf), df)_{T^{\ast}W,w}.  \notag
\end{equation}
Here the left hand side denotes minus the total dimension of the space of vanishing cycles at $w$, and the right hand side denotes the intersection number supported on the fiber at $w$ of $h^{\ast}CC (\mf)$ with the section $df$ of $T^{\ast}W$ defined by the pull-back of a basis of $T^{\ast}Y$.
\end{thm}

\begin{defn}[{\cite[Definition 5.10]{sa4}}]
Let $\mf$ be a constructible complex of $\Lambda$-modules on $X$.
We call $CC(\mf)$ in Theorem \ref{thmsmil} the {\it characteristic cycle} of $\mf$.
\end{defn}

We note that $CC(\mf)$ is compatible with the pull-back by the projection
$X_{\bar{k}}\rightarrow X$ and that
the support of the characteristic cycle $CC(\mf)$ of $\mf$
is contained in the singular support $SS(\mf)$ of $\mf$.
For a cycle $A=\sum_{a}m_{a}[C_{a}]$ on a vector bundle over $X$,
the {\it support} of $A$ means the union of $C_{a}$ such that $m_{a}\neq 0$.
If $h\colon W\rightarrow X$ is an \'{e}tale morphism, then we have
\begin{equation}
h^{\ast}CC(\mf)=CC(h^{*}\mf) \notag
\end{equation}
by \cite[Lemma 5.11.2]{sa4}.

\begin{lem}
\label{suppCC}
Assume that $X$ is purely of dimension $d$.
Let $U$ be an open subscheme of $X$ and $j\colon U\rightarrow X$ the open immersion.
Let $\mf$ be a smooth sheaf of $\Lambda$-modules on $U$.
If $j$ is affine, then we have $(-1)^{d}CC(j_{!}\mf)\ge 0$ and the support of $CC(j_{!}\mf)$ is $SS(j_{!}\mf)$.
\end{lem}

\begin{proof}
Since $\mf$ is assumed to be locally constant, the complex $\mf[d]$ is a perverse sheaf 
by \cite[Examples 4.0]{bbd}.
Since $j$ is assumed to be affine, the sheaf $j_{!}\mf[d]$ is a perverse sheaf
by \cite[Corollary 4.1.10 (i)]{bbd}.
Hence we have $CC(j_{!}\mf[d])\ge 0$ and
the support of $CC(j_{!}\mf[d])$
is $SS(j_{!}\mf[d])$ by \cite[Proposition 5.14]{sa4}.
Thus the assertion holds by \cite[Lemma 5.13.1]{sa4}
and Lemma \ref{ssdisttri}.
\end{proof}

If $X$ is proper over an algebraically closed field, then
let $\chi(X,\mf)$ be the Euler characteristic $\sum_{q}(-1)^{q}\dim H^{q}(X,\mf)$
of a constructible complex $\mf$ of $\Lambda$-modules on $X$. 

\begin{thm}[Index formula, {\cite[Theorem 7.13]{sa4}}]
\label{sindex}
Let $\mf$ be a constructible complex of $\Lambda$-modules on $X$.
Assume that $X$ is projective over an algebraically closed field.
Then we have
\begin{equation}
\label{eqindex}
\chi(X,\mf)=(CC (\mf), T^{\ast}_{X}X)_{T^{\ast}X}, 
\end{equation}
where the right hand side denotes the intersection number of $CC(\mf)$ with $T^{*}_{X}X$.
\end{thm}

\begin{defn}[{\cite[Definition 7.1]{sa4}}]
\label{defpb}
Assume that $X$ is purely of dimension $d$ and
let $C$ be a closed conical subset of $T^{\ast}X$ purely of dimension $d$.
\begin{enumerate}
\item We say that a morphism $h\colon W\rightarrow X$  
from a smooth scheme $W$ over $k$ purely of dimension $c$
is \textit{properly} $C$-transversal if $h$ is $C$-transversal and if 
$h^{\ast}C$ is of dimension $c$.
\item Let $h\colon W\rightarrow X$ be a properly $C$-transversal morphism
from a smooth scheme $W$ over $k$ purely of dimension $c$.
For a linear combination $A=\sum_{a}m_{a}[C_{a}]$ of irreducible components $\{C_{a}\}_{a}$
of $C$, we define $h^{!}A$ to be
the push forward of the cycle 
$(-1)^{d-c}\cdot\sum_{a}m_{a}[h^{*}C_{a}]$ on $T^{*}X\times_{X}W$
by $dh\colon T^{*}X\times_{X}W\rightarrow T^{*}W$ in the sense of the intersection theory.
\end{enumerate}
\end{defn}

\begin{thm}[{\cite[Theorem 7.6]{sa4}}]
\label{thmpull}
Assume that $X$ is purely of dimension $d$ and let 
$C\subset T^{*}X$ be a closed conical subset purely of dimension $d$.
Let $h\colon W\rightarrow X$ be a morphism from a smooth scheme $W$ over $k$ 
purely of dimension $c$.
Let $\mf$ be a constructible complex of $\Lambda$-modules on $X$ 
micro-supported on $C$.  
If $h$ is properly $C$-transversal,
then we have 
\begin{equation}
CC (h^{\ast}\mf)=h^{!}CC (\mf). \notag
\end{equation}
\end{thm}

Let $h\colon W\rightarrow X$ be a separated morphism of schemes of finite type over $k$.
Then the functor $Rh^{!}\colon D^{b}_{c}(X,\Lambda)\rightarrow D_{c}^{b}(W,\Lambda)$ 
is defined to be the 
right adjoint functor of $Rh_{!}\colon D^{b}_{c}(W,\Lambda)\rightarrow D^{b}_{c}(X,\Lambda)$.

For constructible complexes $\mf$ and $\mg$ of $\Lambda$-modules on $X$ and $W$
respectively, let $c'_{h,\mf,\mg} \colon\mf\otimes_{\Lambda}Rh_{!}\mg
\rightarrow Rh_{!}(h^{\ast}\mf\otimes_{\Lambda}\mg)$ be the isomorphism 
of projection formula.
For constructible complexes $\mf$ and $\mh$ of $\Lambda$-modules on $X$,
we define a morphism
\begin{equation}
c_{h,\mf,\mh}\colon h^{\ast}\mf\otimesll Rh^{!}\mh\rightarrow
Rh^{!}(\mf\otimesll \mh) \notag
\end{equation}
to be the adjunction
of the composition 
\begin{equation}
Rh_{!}(h^{\ast}\mf\otimesll Rh^{!}\mh)\xrightarrow{c'^{-1}_{h,\mf,Rh^{!}\mh}}
\mf\otimes_{\Lambda}Rh_{!}Rh^{!}\mh \xrightarrow{\id\otimes\varepsilon} \mf\otimesll \mh,
\notag
\end{equation}
where $\varepsilon\colon Rh_{!}Rh^{!}\rightarrow \id$ is the adjunction morphism.

\begin{defn}[{\cite[Definition 8.5]{sa4}}]
\label{defofftransversal}
Let $h\colon W\rightarrow X$ be a separated morphism of schemes of finite type over $k$ and
$\mf$ a constructible complex of $\Lambda$-modules on $X$.
We say that $h$ is $\mf$-\textit{transversal} if the morphism
$c_{h,\mf,\Lambda}\colon h^{\ast}\mf\otimes_{\Lambda}Rh^{!}\Lambda \rightarrow Rh^{!}\mf$ 
is an isomorphism.
\end{defn}

\begin{lem}
\label{lemcltr}
Let $U$ be an open subscheme of $X$ and $j\colon U\rightarrow X$ the open immersion.
Let $\mf$ be a smooth sheaf of $\Lambda$-modules on $U$.
Let $h\colon W\rightarrow X$ be a separated morphism of smooth schemes over $k$.
Let $j'\colon h^{*}U\rightarrow W$ be the base change of $j$ by $h$.
If the canonical morphisms $j_{!}\mf\rightarrow Rj_{\ast}\mf$ and $j'_{!}h^{\ast}\mf\rightarrow Rj'_{\ast}h^{\ast}\mf$ 
are isomorphisms,
then $h$ is $j_{!}\mf$-transversal.
\end{lem}

\begin{proof}
Since $\mf$ is a smooth sheaf on $U$,
the projection $h^{*}U\rightarrow U$ induced by $h$ is $\mf$-transversal
by \cite[Lemma 8.6.2]{sa4}.
Since $W$ and $X$ are smooth over $k$,
the object $Rh^{!}\Lambda$ is isomorphic to $\Lambda(c)[2c]$ for a locally constant function $c$.
Hence the assertion holds by \cite[Proposition 8.8.1]{sa4}.
\end{proof}

\begin{prop}[{cf.\ \cite[Proposition 8.13]{sa4}}]
\label{propftr}
Let $\mf$ be a constructible complex of $\Lambda$-modules on $X$
and $C$ a closed conical subset of $T^{\ast}X$.
Then the following are equivalent:
\begin{enumerate}
\item $\mf$ is micro-supported on $C$.
\item The support of $\mathcal{F}$ in $X$ is a subset of the base $C\cap \zers \subset X$ of $C$ and 
every separated $C$-transversal morphism $h\colon W\rightarrow X$ of
smooth schemes over $k$ is $\mf$-transversal.
\end{enumerate}
Especially, if $C\supset T^{\ast}_{X}X$, 
then $\mf$ is micro-supported on $C$ if and only if every separated $C$-transversal morphism 
$h\colon W\rightarrow X$ of smooth schemes over $k$ is $\mf$-transversal.
\end{prop}

\begin{proof}
Let $h\colon W\rightarrow X$ be a morphism of smooth schemes over $k$
and $X=\bigcup_{i'}X_{i'}$ an open affine covering of $X$.
Let $W\times_{X}X_{i'}=\bigcup_{j'}W_{i'j'}$ be an open affine covering of 
$W\times_{X}X_{i'}$ for $i'$.
Let $j_{i'}\colon X_{i'}\rightarrow X$ be the open immersion and 
let $h_{i'j'}\colon W_{i'j'}\rightarrow X_{i'}$ be the morphism induced by $h$.
Then the $C$-transversality of $h$ is equivalent to the $C$-transversality
of $j_{i'}\circ h_{i'j'}$ for all $i'$ and $j'$.
Since $j_{i'}$ and $h_{i'j'}$ are separated for every $i'$ and $j'$, 
the composition $j_{i'}\circ h_{i'j'}$ is separated for every $i'$ and $j'$.
Hence the assertion holds by \cite[Proposition 8.13]{sa4}.
\end{proof}

\section{Ramification theory}
\label{sram}
\subsection{Refined Swan conductor and characteristic form}
\label{sslram}
In this subsection, we briefly recall Brylinski-Kato's logarithmic ramification theory
and Matsuda's non-logarithmic variant of Brylinski-Kato's theory.
We refer to \cite{br}, \cite{ka1}, \cite{ka2}, \cite{ma}, and \cite{ya} for more detail.

Let $K$ be a complete discrete valuation field of characteristic $p>0$.
For $n\in \mathbf{Z}_{> 0}$, we regard $H^{1}(K,\mathbf{Z}/n\mathbf{Z})$ as a subgroup of 
$H^{1}(K,\mathbf{Q}/\mathbf{Z})=\varinjlim_{m}H^{1}(K,\mathbf{Z}/m\mathbf{Z})$.
Let $W_{s}(K)$ be the Witt ring of length $s$ for an integer $s\ge 0$.
Let $V\colon W_{s}(K)\rightarrow W_{s+1}(K)$ be the Verschiebung
and $F\colon W_{s}(K)\rightarrow W_{s}(K)$ the Frobenius.
We define 
\begin{equation}
\label{deltas}
\delta_{s}\colon W_{s}(K)\rightarrow H^{1}(K,\mathbf{Q}/\mathbf{Z})
\end{equation}
to be the composition
\begin{equation}
\label{compds}
\delta_{s}\colon
W_{s}(K)\rightarrow W_{s}(K)/(F-1)W_{s}(K)\xrightarrow{\simeq}
H^{1}(K,\mathbf{Z}/p^{s}\mathbf{Z})\rightarrow H^{1}(K,\mathbf{Q}/\mathbf{Z}), 
\end{equation}
where the middle arrow
is the isomorphism in the Artin-Schreier-Witt theory.

We first recall Brylinski-Kato's logarithmic ramification theory.
Let $\dvr_{K}$ be the valuation ring of $K$ and $\mathfrak{m}_{K}$ the maximal ideal of $\dvr_{K}$.
Let $\ord_{K}\colon K\rightarrow \mathbf{Z}\cup\{\infty\}$ be the normalized valuation
such that $\ord_{K}(\pi)=1$ for a uniformizer $\pi$ of $K$.

\begin{defn}[{\cite[Proposition 1]{br}, \cite[Definition (2.2), Corollary (2.5), Definition (3.1), Theorem (3.2) (1)]{ka1}}]
Let $s\ge 0$ be an integer.
\begin{enumerate}
\item Let $a=(a_{s-1},\ldots ,a_{0})$ be an element of $W_{s}(K)$.
We define the order $\ord_{K}(a)$ of $a$ to be $\min_{0\le i\le s-1}\{p^{i}\ord_{K}(a_{i})\}$.
\item We define an increasing filtration $\{\fillog_{n}W_{s}(K)\}_{n\in \mathbf{Z}}$ 
of $W_{s}(K)$ by
\begin{equation}
\label{lwsk}
\fillog_{n}W_{s}(K)=\{a\in W_{s}(K) \;|\;
\ord_{K}(a)\ge -n\}. 
\end{equation}

\item We define an increasing filtration $\{\fillog_{n}H^{1}(K,\BQ/\BZ)\}_{n\in \mathbf{Z}_{\ge 0}}$ 
of $H^{1}(K,\BQ/\BZ)$ by
\begin{equation}
\fillog_{n}H^{1}(K,\BQ/\BZ)=H^{1}(K,\BQ/\BZ)'
+\bigcup_{s\ge 1}\delta_{s}(\fillog_{n}W_{s}(K)),\notag 
\end{equation}
where $H^{1}(K,\BQ/\BZ)'$ denotes the prime to $p$-part of $H^{1}(K,\BQ/\BZ)$.

\item Let $\chi$ be an element of $H^{1}(K,\BQ/\BZ)$.
We define the \textit{Swan conductor} $\sw(\chi)$ of $\chi$ to be
the minimal integer $n\ge 0$ such that $\chi\in \fillog_{n}H^{1}(K,\BQ/\BZ)$.
\end{enumerate}
\end{defn}

For an integer $n\ge 1$ such that  $s'=\min\{\ord_{p}(n),s\}<s$,
by (\ref{lwsk}), we have
\begin{equation}
\label{lfilwskv}
\fillog_{n}W_{s}(K)=\fillog_{n-1}W_{s}(K)+V^{s-s'-1}\fillog_{n}W_{s'+1}(K).
\end{equation}

We put $\Omega_{K}^{1}=\Omega^{1}_{K/K^{p}}$ and 
$\Omega^{1}_{\dvr_{K}}=\Omega^{1}_{\dvr_{K}/\dvr_{K}^{p}}$.
We define an increasing filtration $\{\fillog_{n}\Omega^{1}_{K}\}_{n \in \mathbf{Z}_{\ge 0}}$ 
of $\Omega^{1}_{K}$ by
\begin{equation}
\fillog_{n}\Omega^{1}_{K}
=\{(\alpha d\pi/\pi+\beta)/\pi^{n} \;|\; \alpha \in \dvr_{K}, \beta\in \Omega^{1}_{\dvr_{K}}\}
=\mathfrak{m}_{K}^{-n}\Omega^{1}_{\dvr_{K}}(\log ), \notag
\end{equation}
where $\pi$ is a uniformizer of $K$.

Let $n\ge 1$ be an integer and put $\grlog_{n}=\fillog_{n}/\fillog_{n-1}$.
We consider the morphism
\begin{equation}
-F^{s-1}d\colon W_{s}(K)\rightarrow \Omega^{1}_{K};\; (a_{s-1},\ldots,a_{0})\mapsto -\sum_{i=0}^{s-1}a_{i}^{p^{i}-1}da_{i}. \notag
\end{equation}
Then $-F^{s-1}d$ induces the morphism 
$\varphi_{s}^{(n)}\colon \grlog_{n}W_{s}(K)\rightarrow \grlog_{n}\Omega^{1}_{K}$. 
By \cite[Remark 3.2.12]{ma} or \cite[Proposition 10.7]{as4},
there exists a unique injection 
\begin{equation}
\phi^{(n)}\colon \grlog_{n}H^{1}(K,\BQ/\BZ)\rightarrow \grlog_{n}\Omega^{1}_{K}\notag
\end{equation}
such that the following diagram is commutative for every $s\in \BZ_{\ge 0}$:
\begin{equation}
\label{diagrsw}
\xymatrix{
\grlog_{n}W_{s}(K) \ar[rd] \ar[rr]^-{\varphi_{s}^{(n)}} & & \grlog_{n}\Omega^{1}_{K}\\
& \grlog_{n}H^{1}(K,\BQ/\BZ) \ar[ru]_-{\phi^{(n)}}, &
}
\end{equation}
where the left slanting arrow is induced by $\delta_{s}$.

\begin{defn}[{\cite[(3.4.2)]{ka1}, \cite[Remark 3.2.12]{ma}, see also \cite[Proposition 10.7]{as4}}]
Let $\chi$ be an element of $H^{1}(K,\BQ/\BZ)$.
Assume $\sw(\chi)=n\ge 1$.
We define the \textit{refined Swan conductor} $\rsw(\chi)$
of $\chi$ to be the image of $\chi$ by the composition
\begin{equation}
\fillog_{n}H^{1}(K,\BQ/\BZ)\rightarrow \grlog_{n}H^{1}(K,\BQ/\BZ)
\xrightarrow{\phi^{(n)}}\grlog_{n}\Omega^{1}_{K}.\notag
\end{equation}
\end{defn}

We note that $\rsw(\chi)$ is not $0$.

\begin{lem}
\label{lemrsw}
Let $\chi$ be an element of $H^{1}(K,\mathbf{Q}/\mathbf{Z})$ and
$a$ an element of $W_{s}(K)$ whose image in $H^{1}(K,\mathbf{Q}/\mathbf{Z})$ is
the $p$-part of $\chi$.
We put $\ord_{K}(a)=-n$ and assume $n \ge 1$.
Then the following are equivalent:
\begin{enumerate}
\item $\sw(\chi)=n$.
\item $\rsw(\chi)=\varphi_{s}^{(n)}(\bar{a})$.
\item $\varphi_{s}^{(n)}(\bar{a})\neq 0$ in $\grlog_{n}\Omega_{K}^{1}$.
\end{enumerate}
\end{lem}

\begin{proof}
The implication (i) $\Rightarrow$ (ii) is clear.
The implication (ii) $\Rightarrow$ (iii) follows from $\rsw(\chi)\neq 0$.
We prove (iii) $\Rightarrow$ (i).
Since the $p$-part of $\chi$ is the image of $a$ in $H^{1}(K,\mathbf{Q}/\mathbf{Z})$,
we have $\sw(\chi)\le n$.
Suppose $\sw(\chi)<n$.
Then the image of $a$ in $\grlog_{n}H^{1}(K,\mathbf{Q}/\mathbf{Z})$ is $0$.
By the commutativity of (\ref{diagrsw}), we have $\varphi_{s}^{(n)}(\bar{a})=0$,
which contradicts to the condition (iii).
Hence the assertion holds.
\end{proof}

Let $X$ be a smooth scheme over a perfect field $k$ of characteristic $p>0$.
Let $D$ be a divisor on $X$ with simple normal crossings and 
$\{D_{i}\}_{i\in I}$ the irreducible components of $D$.
Let $\dvr_{K_{i}}=\hat{\dvr}_{X,\mathfrak{p}_{i}}$ be the completion of the local ring 
$\dvr_{X,\mathfrak{p}_{i}}$ at the generic point $\mathfrak{p}_{i}$ of $D_{i}$ and
$K_{i}=\Frac \dvr_{K_{i}}$ the local field at $\mathfrak{p}_{i}$ for $i\in I$.
We put $U=X-D$. 
For an element $\chi$ of $H^{1}_{\et}(U,\BQ/\BZ)$, let
$\chi|_{K_{i}}\in H^{1}(K_{i},\mathbf{Q}/\mathbf{Z})$ be the image of $\chi$
by the canonical morphism $H^{1}_{\et}(U,\mathbf{Q}/\mathbf{Z})\rightarrow H^{1}(K_{i},\mathbf{Q}/\mathbf{Z})$ for $i\in I$. 

\begin{defn}[{\cite[(3.4.1)]{ka2}}]
\label{defswdiv}
Let $\chi$ be an element of $H^{1}_{\et}(U,\BQ/\BZ)$.
We define the \textit{Swan conductor divisor} $R_{\chi}$ of $\chi$ by 
\begin{equation}
R_{\chi}=\sum_{i\in I}\sw(\chi|_{K_{i}})D_{i}. \notag
\end{equation} 

Let $Z_{\chi}=\Supp(R_{\chi})$ be the support of $R_{\chi}$.
We write $I_{\mW,\chi}$ for
the index set of irreducible components of $Z_{\chi}$, which is a subset of the index set $I$,
and we put $I_{\mT,\chi}=I- I_{\mW,\chi}$.

For a point $x$ on $D$, we define a subset $I_{x}$ of the index set $I$ by 
$I_{x}= \{i\in I \; |\; x \in D_{i}\}$. 
For an element $\chi$ of $H^{1}_{\et}(U,\mathbf{Q}/\mathbf{Z})$ and a point $x$ on $D$,
we define subsets $I_{\mW,\chi,x}$ and $I_{T,\chi,x}$ of $I$ by
$I_{\mW,\chi,x}=I_{x}\cap I_{\mW,\chi}$ and $I_{\mT,\chi,x}=I_{x}\cap I_{\mT,\chi}$ respectively.
\end{defn}

For an element $\chi$ of $H^{1}_{\et}(U,\mathbf{Q}/\mathbf{Z})$, 
there exists a unique global section 
\begin{equation}
\rsw(\chi)\in \Gamma(Z_{\chi},\Omega^{1}_{X}(\log D)(R_{\chi})|_{Z_{\chi}}), 
\notag 
\end{equation}
such that the germ $\rsw(\chi)_{\mathfrak{p}_{i}}$ of $\rsw(\chi)$ at $\mathfrak{p}_{i}$
is $\rsw(\chi|_{K_{i}})$ for every $i\in I_{\mW,\chi}$
by \cite[(3.4.2)]{ka2}.
Here we note that if $\sw(\chi|_{K_{i}})=n_{i}\ge 1$ then 
$\Omega^{1}_{X}(\log D)(R_{\chi})_{\mathfrak{p}_{i}}=\grlog_{n_{i}}\Omega_{K_{i}}^{1}$.
We recall a construction of $\rsw(\chi)$ using sheaves of Witt vectors in Subsection \ref{sscrc}.
We call $\rsw(\chi)$ the \textit{refined Swan conductor} of $\chi$.

\begin{defn}[{\cite[(3.4.3), Definition 4.2, Lemma 4.3]{ka2}}]
\label{defofclean}
Let $\chi$ be an element of $H^{1}_{\et}(U,\BQ/\BZ)$.
\begin{enumerate}
\item We say that $(X,U,\chi)$ is \textit{clean} at $x\in X$ if 
one of the following conditions is satisfied:
\begin{enumerate}
\item $x\notin Z_{\chi}$.
\item $x\in Z_{\chi}$ and $\rsw(\chi)_{x}$ is a part of a basis of the free $\dvr_{Z_{\chi},x}$-module
$\Omega^{1}_{X}(\log D)(R_{\chi})|_{Z_{\chi},x}$.
\end{enumerate}

We say that $(X,U,\chi)$ is \textit{clean} if $(X,U,\chi)$ is clean at every point on $X$.

\item Let $x$ be a closed point of $D$ and $i$ an element of $I_{\mW,\chi,x}$.
We define $\ord(\chi;x,D_{i})$ to be the maximal integer $n\ge 0$ such that 
\begin{equation}
\rsw(\chi)|_{D_{i},x} \in \mathfrak{m}_{x}^{n} \Omega^{1}_{X}(\log D)(R_{\chi})|_{D_{i},x}, \notag
\end{equation}  
where $\mathfrak{m}_{x}$ is the maximal ideal at $x$.
\end{enumerate} 
\end{defn}

\begin{lem}
\label{lemclean}
Let $\chi$ be an element of $H^{1}_{\et}(U,\mathbf{Q}/\mathbf{Z})$. 
\begin{enumerate}
\item Let $x$ be a closed point of $Z_{\chi}$.
The following are equivalent:
\begin{enumerate}
\item $(X,U,\chi)$ is clean at $x$.
\item $\ord(\chi; x,D_{i})=0$ for some $i\in I_{\mW,\chi,x}$. 
\item $\ord(\chi; x,D_{i})=0$ for every $i\in I_{\mW,\chi,x}$.
\end{enumerate}
\item The set of points where $(X,U,\chi)$ is not clean is a closed subset of $X$ 
of codimension $\ge 2$.
\item $(X,U,\chi)$ is clean if and only if $(X,U,\chi)$ is clean at every closed point of $Z_{\chi}$.
\end{enumerate}
\end{lem}

\begin{proof}
(i) Let $i$ be an element of $I_{\mW,\chi,x}$.
Then the condition $\ord(\chi;x,D_{i})=0$ is equivalent to 
$\rsw(\chi)_{x}\neq 0$ in $\Omega^{1}_{X}(\log D)_{x}\otimes_{\dvr_{X,x}}k(x)$.
Hence both the conditions (b) and (c) are equivalent to the condition (a).

(ii) Let $Z'$ be the set of points on $X$ where $(X,U,\chi)$ is not clean.
Since $Z'$ is the set of points $x$ on $Z_{\chi}$ such that 
$\rsw(\chi)_{x}\in \mathfrak{m}_{x}\Omega^{1}_{X}(\log D)(R_{\chi})|_{Z_{\chi},x}$,
the set $Z'$ is a closed subset of $Z_{\chi}$.
Since $Z_{\chi}$ is a closed subset of $X$ and we have
$\rsw(\chi)_{\mathfrak{p}_{i}}=\rsw(\chi|_{K_{i}})\neq 0$ for every $i\in I_{\mW,\chi}$,
the set $Z'$ is a closed subset of $X$ of codimension $\ge 2$.

(iii) Since every non-empty closed subset of $Z_{\chi}$ has a closed point,
the assertion holds by (ii).
\end{proof}

Let
$\mathrm{res}_{i}\colon \Omega^{1}_{X}(\log D)|_{D_{i}}\rightarrow \dvr_{D_{i}}$
be the residue morphism for $i\in I$.
For an element $\chi$ of $H^{1}_{\et}(U,\mathbf{Q}/\mathbf{Z})$ and $i\in I_{\mW,\chi}$,
let $\xi_{i}(\chi) \colon \dvr_{X}(-R_{\chi})|_{D_{i}}\rightarrow \dvr_{D_{i}}$ be the composition
\begin{equation}
\label{xichi}
\xi_{i}(\chi)\colon \dvr_{X}(-R_{\chi})|_{D_{i}}\xrightarrow{\times \rsw(\chi)|_{D_{i}}} \Omega^{1}_{X}(\log D)|_{D_{i}}
\xrightarrow{\res_{i}} \dvr_{D_{i}}. 
\end{equation}

\begin{defn}[{\cite[Definition (7.4)]{ka1}}]
\label{defsclean}
Let $\chi$ be an element of $H^{1}_{\et}(U,\mathbf{Q}/\mathbf{Z})$.
We say that $(X,U,\chi)$ is \textit{strongly} clean at $x\in X$ if one of the following conditions is satisfied:
\begin{enumerate}
\item $x\notin Z_{\chi}$.
\item $x\in Z_{\chi}$ and the image by $\xi_{i}(\chi)$ (\ref{xichi}) in $k(x)$ is not $0$ for every $i\in I_{\mW,\chi,x}$.
\end{enumerate}
We say that $(X,U,\chi)$ is \textit{strongly} clean if $(X,U,\chi)$ is strongly clean at every point on $X$.
\end{defn}

We note that  
if $(X,U,\chi)$ is strongly clean at a point $x\in X$ then $(X,U,\chi)$ is clean at $x$
by Lemma \ref{lemclean} (i), since the condition (ii) in Definition \ref{defsclean} deduces
$\ord(\chi;x,D_{i})=0$ for every $i\in I_{\mW,\chi,x}$ if $x\in Z_{\chi}$.
\vspace{0.2cm}

We next recall Matsuda's non-logarithmic ramification theory.
Let $K$ be a complete discrete valuation field of characteristic $p>0$.
Let $\dvr_{K}$ be the valuation ring of $K$ and $\mathfrak{m}_{K}$ the maximal ideal of $\dvr_{K}$.
Let $F_{K}=\dvr_{K}/\mathfrak{m}_{K}$ be the residue field of $K$.

\begin{defn}[{cf.\ \cite[3.1, Definition 3.1.1, Definition 3.2.5]{ma}}]
Let $s\ge 0$ be an integer.
\begin{enumerate}
\item We define an 
increasing filtration $\{\fil_{m}W_{s}(K)\}_{m\in \mathbf{Z}_{\ge 1}}$ of $W_{s}(K)$ by
\begin{equation}
\label{nlfwsk}
\fil_{m}W_{s}(K)=\fillog_{m-1}W_{s}(K)+V^{s-s'}\fillog_{m}W_{s'}(K),  
\end{equation}
where $s'=\min\{\ord_{p}(m), s\}$.

\item We define an 
increasing filtration $\{\fil_{m}H^{1}(K,\BQ/\BZ)\}_{m\in \mathbf{Z}_{\ge 1}}$ of $H^{1}(K,\BQ/\BZ)$ by
\begin{equation}
\label{filph}
\fil_{m}H^{1}(K,\BQ/\BZ)=H^{1}(K,\BQ/\BZ)'
+\bigcup_{s\ge 1}\delta_{s}(\fil_{m}W_{s}(K)),
\end{equation}
where $H^{1}(K,\BQ/\BZ)'$ denotes the prime to $p$-part of $H^{1}(K,\BQ/\BZ)$.

\item Let $\chi$ be an element of $H^{1}(K,\BQ/\BZ)$.
We define the \textit{total dimension} $\dt(\chi)$
of $\chi$ to be the minimal integer $m\ge 1$ such that
$\chi\in \fil_{m}\hok$.
\end{enumerate}
\end{defn}

By (\ref{nlfwsk}), we have
\begin{equation}
\label{seqws}
\fillog_{m-1}W_{s}(K)\subset \fil_{m}W_{s}(K)\subset \fillog_{m}W_{s}(K).
\end{equation}
Further, since $\ord_{p}(1)=0$, we have
\begin{equation}
\label{flzfo}
\fillog_{0}W_{s}(K)=\fil_{1}W_{s}(K). 
\end{equation}
Hence we have
\begin{equation}
\label{seqfilh}
\fillog_{m-1}\hok\subset \fil_{m}\hok \subset \fillog_{m}\hok
\end{equation}
and 
\begin{equation}
\label{eqfilh}
\fillog_{0}\hok=\fil_{1}\hok. 
\end{equation}

For an element $\chi$ of $\fil_{m}H^{1}(K,\mathbf{Q}/\mathbf{Z})$ 
for an integer $m\in \mathbf{Z}_{\ge 1}$,
we have $\chi\in \fillog_{m-1}H^{1}(K,\mathbf{Q}/\mathbf{Z})$ or 
$\chi\in \fillog_{m}H^{1}(K,\mathbf{Q}/\mathbf{Z})
\setminus \fillog_{m-1}H^{1}(K,\mathbf{Q}/\mathbf{Z})$ by (\ref{seqfilh}).
Hence, for an element $\chi$ of $H^{1}(K,\mathbf{Q}/\mathbf{Z})$, 
we have $\dt(\chi)=\sw(\chi)+1$ or $\dt(\chi)=\sw(\chi)$.
For an element $\chi$ of $H^{1}(K,\mathbf{Q}/\mathbf{Z})$ such that $\sw(\chi)>0$,
we say that $\chi$ is of {\it type} $\mI$ if $\dt(\chi)=\sw(\chi)+1$ and
that $\chi$ is of {\it type} $\mII$ if $\dt(\chi)=\sw(\chi)$.

\begin{lem}
\label{lemgrqot}
Let $m\ge 1$ and $s\ge 0$ be integers.
We put $s'=\min\{\ord_{p}(m),s\}$.
\begin{enumerate}
\item Assume $s>s'$.
For $a\in K$, let $(\underline{a})_{s'}=(a_{s-1},\ldots,a_{0})$ be the element of 
$W_{s}(K)$ defined by $a_{i}=a$ if $i=s'$ and $a_{i}=0$ if $i\neq s'$.
Then the injection $K\rightarrow W_{s}(K)$ defined by $a\mapsto (\underline{a})_{s'}$ induces an isomorphism
\begin{equation}
\grlog_{m/p^{s'}}K\rightarrow\fillog_{m}W_{s}(K)/\fil_{m}W_{s}(K). \notag
\end{equation}
\item The morphism $V^{s-s'}\colon W_{s'}(K)\rightarrow W_{s}(K)$ induces an isomorphism
\begin{equation}
\grlog_{m}W_{s'}(K)\rightarrow\fil_{m}W_{s}(K)/\fillog_{m-1}W_{s}(K). \notag
\end{equation}
\end{enumerate}
\end{lem}

\begin{proof}
The assertion (i) holds by (\ref{lfilwskv}) and (\ref{nlfwsk}), 
and the assertion (ii) holds by (\ref{nlfwsk}).
\end{proof}
 
We define an increasing filtration $\{\fil_{m}\Omega^{1}_{K}\}_{m \in \mathbf{Z}_{\ge 1}}$ of $\Omega^{1}_{K}$ by
\begin{equation}
\fil_{m}\Omega^{1}_{K}=(1/\pi^{m})\cdot
\Omega^{1}_{\dvr_{K}}=\mathfrak{m}_{K}^{-m}\Omega^{1}_{\dvr_{K}}, \notag 
\end{equation}
where $\pi$ is a uniformizer of $K$.
We put $\gr_{m}=\fil_{m}/\fil_{m-1}$ for $m\in \mathbf{Z}_{\ge 2}$.
By \cite[3.2]{ma} and \cite[Proposition 1.17 (i)]{ya}, there exists a unique morphism
\begin{equation}
\varphi_{s}^{\prime (m)}\colon \gr_{m}W_{s}(K)\rightarrow \gr_{m}\Omega^{1}_{K}\otimes_{F_{K}}F_{K}^{1/p} \notag
\end{equation}
for $m\ge 2$
such that 
\begin{equation}
\label{vphip}
\varphi_{s}^{\prime (m)}((a_{s-1},\ldots,a_{0}))=\begin{cases}
-da_{0}+\sqrt{\overline{\pi^{2}a_{0}}}d\pi/\pi^{2} & ((p,m)=(2,2)), \\
-\sum_{i=0}^{s-1}a_{i}^{p^{i}-1}da_{i} &((p,m)\neq (2,2))
\end{cases} 
\end{equation}
if $\pi$ is a uniformizer of $K$.
By \cite[Proposition 3.2.1, Proposition 3.2.3]{ma}
and \cite[Proposition 1.17 (ii)]{ya},
for $m\ge 2$,
there exists a unique injection 
\begin{equation}
\label{phipm}
\phi^{\prime (m)}\colon \gr_{m}H^{1}(K,\BQ/\BZ)\rightarrow 
\gr_{m}\Omega^{1}_{K}\otimes_{F_{K}} F_{K}^{1/p} 
\end{equation}
such that the following diagram is commutative for every $s\in \mathbf{Z}_{\ge 0}$:
\begin{equation}
\label{diagcform}
\xymatrix{
\gr_{m}W_{s}(K) \ar[rd] \ar[rr]^-{\varphi_{s}^{\prime (m)}} & & \gr_{m}\Omega^{1}_{K}\otimes_{F_{K}}
F_{K}^{1/p}\\
& \gr_{m}H^{1}(K,\BQ/\BZ)\ar[ru]_-{\phi^{\prime (m)}}, &
}
\end{equation}
where the left slanting arrow is induced by $\delta_{s}$.

If $(p,m)\neq(2,2)$, then we may replace 
$\gr_{m}\Omega^{1}_{K}\otimes_{F_{K}}F_{K}^{1/p}$
by $\gr_{m}\Omega^{1}_{K}$ in (\ref{phipm}) by (\ref{vphip}) and the commutativity of (\ref{diagcform}).
Further, even in the case where $(p,m)=(2,2)$, 
if we restrict $\gr_{2}\hok$ to $\grlog_{1}\hok\subset \gr_{2}\hok$,
then we may replace 
$\gr_{m}\Omega^{1}_{K}\otimes_{F_{K}}F_{K}^{1/p}$
by $\gr_{m}\Omega^{1}_{K}$ in (\ref{phipm}) by (\ref{vphip}) and the commutativity of (\ref{diagcform}),
since we may replace $\gr_{m}W_{s}(K)$ by $\grlog_{1}W_{s}(K)$ in (\ref{diagcform}) if $(p,m)=(2,2)$ and
if we restrict $\gr_{2}\hok$ to $\grlog_{1}\hok$.

\begin{defn}[{\cite[Definition 3.2.5]{ma}, \cite[Definition 1.19]{ya}}]
Let $\chi$ be an element of $H^{1}(K,\BQ/\BZ)$.
Assume $\dt(\chi)=m\ge 2$.
We define the \textit{characteristic form}
$\cform(\chi)$ of $\chi$ to be the image of $\chi$ by the composition
\begin{equation}
\fil_{m}\hok\rightarrow \gr_{m}\hok \xrightarrow{\phi^{\prime (m)}} \gr_{m}\Omega^{1}_{K}\otimes_{F_{K}}
F_{K}^{1/p}. \notag
\end{equation}
\end{defn}

We note that $\cform(\chi)$ is not $0$.
If $(p,\sw(\chi),\dt(\chi))\neq(2,2,2)$, then
we may regard $\cform(\chi)$ as an element of $\gr_{m}\Omega^{1}_{K}$.
For an element $\chi$ of $\hok$ such that $\sw(\chi)\ge 1$,
we say that $\chi$ is of {\it usual type} if $(p,\sw(\chi),\dt(\chi))\neq(2,2,2)$
and that $\chi$ is of {\it exceptional type} if $(p,\sw(\chi),\dt(\chi))=(2,2,2)$.

\begin{lem}
\label{lemcform}
Let $\chi$ be an element of $H^{1}(K,\mathbf{Q}/\mathbf{Z})$ and
$a$ an element of $W_{s}(K)$ whose image in $H^{1}(K,\mathbf{Q}/\mathbf{Z})$ is
the $p$-part of $\chi$.
Let $m\ge 2$ be an integer and assume $a\in \fil_{m}W_{s}(K)$.
Then the following are equivalent:
\begin{enumerate}
\item $\dt(\chi)=m$.
\item $\cform(\chi)=\varphi'^{(m)}_{s}(\bar{a})$.
\item $\varphi'^{(m)}_{s}(\bar{a})\neq 0$ in $\gr_{m}\Omega_{K}^{1}\otimes_{F_{K}}F_{K}^{1/p}$.
\end{enumerate}
\end{lem}

\begin{proof}
The assertion holds similarly as the proof of Lemma \ref{lemrsw}
by $\cform(\chi)\neq 0$ and the commutativity of (\ref{diagcform}).
\end{proof}

\begin{lem}
\label{lemtone}
Let $\chi$ be an element of $\hok$ such that $\sw(\chi)=n\ge 1$.
We put $s'=\ord_{p}(n)$.
Let $a=(a_{s-1},\ldots,a_{0})$ be an element of $\fillog_{n}W_{s}(K)$
whose image by $\delta_{s}$ is the $p$-part of $\chi$.
If $\chi$ is of type $\mI$, then $s>s'$ and $\cform(\chi)$ depends only on $a_{s'}$.  
\end{lem}

\begin{proof}
Since $\dt(\chi)=n+1$, we have $a\notin \fil_{n}W_{s}(K)$.
Since $\fil_{n}W_{s}(K)=\fillog_{n}W_{s}(K)$ if $\ord_{p}(n)\ge s$
by (\ref{nlfwsk}), we have $s'=\ord_{p}(n)<s$.
By Lemma \ref{lemgrqot} (i), the image $\bar{a}$ of $a$ in $\fillog_{n}W_{s}(K)/\fil_{n}W_{s}(K)\subset \gr_{n+1}W_{s}(K)$ 
depends only on $a_{s'}$.
Since $\cform(\chi)$ is the image of $\bar{a}\in \gr_{n+1}W_{s}(K)$
in $\gr_{n+1}\Omega^{1}_{K}\otimes_{F_{K}}F_{K}^{1/p}$ by $\varphi'^{(n+1)}_{s}$ by Lemma \ref{lemcform}, 
the assertion holds.
\end{proof}

\begin{lem}
\label{lemwh}
Let $n>0$ be an integer and put $s'=\ord_{p}(n)$.
For an integer $s>s'$, the morphism
\begin{equation}
\label{morwh}
\fillog_{n}W_{s}(K)/\fil_{n}W_{s}(K)\rightarrow 
\fillog_{n}H^{1}(K,\mathbf{Q}/\mathbf{Z})/\fil_{n}H^{1}(K,\mathbf{Q}/\mathbf{Z})
\end{equation}
induced by $\delta_{s}$ is injective.
\end{lem}

\begin{proof}
By (\ref{compds}) and (\ref{filph}),
the kernel of (\ref{morwh}) is the image of 
$\fil_{n}W_{s}(K)+(F-1)(W_{s}(K))\cap \fillog_{n}W_{s}(K)$ in 
$\fillog_{n}W_{s}(K)/\fil_{n}W_{s}(K)$.
Since $(F-1)(W_{s}(K))\cap \fillog_{n}W_{s}(K)=(F-1)(\fillog_{[n/p]}W_{s}(K))$ 
by \cite[Lemma 1.7 (iii)]{ya} and $[n/p]<n$,
we have $\fil_{n}W_{s}(K)+(F-1)(W_{s}(K))\cap \fillog_{n}W_{s}(K)=
\fil_{n}W_{s}(K)+F(\fillog_{[n/p]}W_{s}(K))$.
Let $a=(a_{s-1},\ldots,a_{0})$ be an element of $\fillog_{[n/p]}W_{s}(K)$.
Then we have $p^{s'}\ord_{K}(a_{s'})\ge -n/p$.
Since $\ord_{p}(n)=s'$ and $\ord_{K}(a_{s'})$ is an integer,
we have $\ord_{K}(a_{s'})>- n/p^{s'+1}$ and hence $\ord_{K}(a_{s'}^{p})>-n/p^{s'}$.
Since $F(a)=(a_{s-1}^{p},\ldots,a_{0}^{p})$,
the image of $\fil_{n}W_{s}(K)+F(\fillog_{[n/p]}W_{s}(K))$ in 
$\fillog_{n}W_{s}(K)/\fil_{n}W_{s}(K)$ is $0$
by Lemma \ref{lemgrqot} (i).
Hence the assertion holds.
\end{proof}

\begin{cor}
\label{alnl}
Let $\chi$ be an element of $H^{1}(K,\mathbf{Q}/\mathbf{Z})$.
We put $\sw(\chi)=n$ and $\dt(\chi)=m$.
Let $a$ be an element of $W_{s}(K)$ whose image in $H^{1}(K,\mathbf{Q}/\mathbf{Z})$
is the $p$-part of $\chi$.
If $a\in \fillog_{n}W_{s}(K)$, then we have $a\in \fil_{m}W_{s}(K)$.
\end{cor}

\begin{proof}
Suppose $n=0$. Then $m=1$ by (\ref{eqfilh})
and hence the assertion holds by (\ref{flzfo}).

Suppose $n>0$ and that $\chi$ is of type $\mI$. 
Then $m=n+1$ and hence the assertion holds by (\ref{seqws}).

Suppose $n>0$ and that $\chi$ is of type $\mII$.
Then $m=n>0$.
If $s=s'=\min\{\ord_{p}(m),s\}$, then we have $\fil_{n}W_{s}(K)=\fillog_{n}W_{s}(K)$ by (\ref{nlfwsk}).
Hence the assertion holds if $s=s'$.
If $s>s'$, then the assertion holds by Lemma \ref{lemwh}.
\end{proof}

\begin{cor}
\label{cortw}
Let $n\ge 1$ be an integer and put $s'=\ord_{p}(n)$ and $n'=p^{-s'}n$.
Let $s>s'$ be an integer and
$a=(a_{s-1},\ldots,a_{0})$ an element of $\fillog_{n}W_{s}(K)$.
Let $\chi$ be an element of $\hok$ whose $p$-part is the image of $a$ by $\delta_{s}$. 
Then the following are equivalent:
\begin{enumerate}
\item $(\sw(\chi),\dt(\chi))=(n,n+1)$.
\item $\ord_{K}(a_{s'})=n'$.
\end{enumerate}
\end{cor}

\begin{proof}
We note that $a\notin \fil_{n}W_{s}(K)$ if and only if $\ord_{K}(a_{s'})=-n'$
by (\ref{lfilwskv}) and (\ref{nlfwsk}).
Hence it is sufficient to prove the equivalence of the condition (i) and $a\notin \fil_{n}W_{s}(K)$.
We consider the morphism (\ref{morwh}).
Since the morphism (\ref{morwh}) is injective by Lemma \ref{lemwh},
we have $a\notin \fil_{n}W_{s}(K)$ if and only if $(\sw(\chi),\dt(\chi))=(n,n+1)$ by (\ref{seqfilh}).
Hence the assertion holds.
\end{proof}

Let $X$ be a smooth scheme over a perfect field $k$ of characteristic $p>0$.
Let $D$ be a divisor on $X$ with simple normal crossings and 
$\{D_{i}\}_{i\in I}$ the irreducible components of $D$.
We put $U=X-D$.
Let $\dvr_{K_{i}}=\hat{\dvr}_{X,\mathfrak{p}_{i}}$ be the completion of the local ring 
$\dvr_{X,\mathfrak{p}_{i}}$ at the generic point $\mathfrak{p}_{i}$ of $D_{i}$ and
$K_{i}=\Frac \dvr_{K_{i}}$ the local field at $\mathfrak{p}_{i}$ for $i\in I$.
For an element $\chi$ of $H^{1}_{\et}(U,\BQ/\BZ)$, let
$\chi|_{K_{i}}\in H^{1}(K_{i},\mathbf{Q}/\mathbf{Z})$ be the image of $\chi$
by the canonical morphism $H^{1}_{\et}(U,\mathbf{Q}/\mathbf{Z})\rightarrow H^{1}(K_{i},\mathbf{Q}/\mathbf{Z})$ for $i\in I$.

\begin{defn}
\label{deftdd}
Let $\chi$ be an element of $H^{1}_{\et}(U,\mathbf{Q}/\mathbf{Z})$.
We define the \textit{total dimension divisor}
$R_{\chi}'$ of $\chi$ by
\begin{equation}
R_{\chi}'=\sum_{i\in I}\dt(\chi|_{K_{i}})D_{i}. \notag
\end{equation}
\end{defn}

For a scheme $S$ over $k$,
we consider the commutative diagram 
\begin{equation}
\xymatrix{
S^{1/p} \ar[d] \ar[r] & S \ar[r]^{F_{S}} \ar[d] & S \ar[d]\\
\Spec k \ar[r]_{F^{-1}_{k}} & \Spec k \ar[r]_{F_{k}}& \Spec k}
\notag
\end{equation}
where the left square is the base change of $S\rightarrow \Spec k$ by the inverse $F_{k}^{-1}$ of $F_{k}$.
Here the symbols $F_{S}$ and $F_{k}$ denote the absolute Frobenius morphisms of $S$ and $\Spec k$ respectively.
We define the \textit{radicial covering} $S^{1/p} \rightarrow S$ to be the composition of morphisms in the upper line.

By \cite[5.2]{ma} and \cite[Definition 1.42]{ya},
there exists a unique global section 
\begin{equation}
\cform(\chi)\in \Gamma(Z_{\chi}^{1/p},\Omega^{1}_{X}(R_{\chi}')\otimes_{\dvr_{X}}\dvr_{Z_{\chi}^{1/p}})
\notag
\end{equation}
such that 
the germ $\cform(\chi)_{\mathfrak{p}_{i}'}$ of $\cform(\chi)$ at the generic point $\mathfrak{p}_{i}'$
of $D_{i}^{1/p}$ is $\cform(\chi|_{K_{i}})$ for every $i\in I_{\mW,\chi}$.
Here we note that if $\dt(\chi|_{K_{i}})=m_{i}\ge 2$ then 
$\Omega^{1}_{X}(R_{\chi}')|_{Z_{\chi},\mathfrak{p}_{i}}=\gr_{m_{i}}\Omega^{1}_{K_{i}}$ and
$\Omega^{1}_{X}(R'_{\chi})_{\mathfrak{p}_{i}}\otimes_{\dvr_{X,\mathfrak{p}_{i}}}\dvr_{Z_{\chi}^{1/p},
\mathfrak{p}_{i}'}=\gr_{m_{i}}\Omega_{K_{i}}^{1}\otimes_{F_{K_{i}}}F_{K_{i}}^{1/p}$.
We recall a construction of $\cform(\chi)$ using sheaves of Witt vectors in Subsection \ref{sscrc}.
We call $\cform(\chi)$ the \textit{characteristic form} of $\chi$.

We note that if $\chi|_{K_{i}}$ is of usual type (e.g.\ $p\neq 2$ or $\chi|_{K_{i}}$ is of type $\mI$)
for every $i\in I_{\mW,\chi}$ 
then we may regard $\cform(\chi)$ as a global section of $\Omega^{1}_{X}(R_{\chi}')|_{Z_{\chi}}$.

\begin{defn}[cf.\ {\cite[Definition 2.17, Subsection 4.2]{sa1}}]
\label{defofnondeg}
Let $\chi$ be an element of $H^{1}_{\et}(U,\mathbf{Q}/\mathbf{Z})$.
\begin{enumerate}
\item We say that $(X,U,\chi)$ is {\it non-degenerate} at $x\in X$ if 
one of the following conditions is satisfied:
\begin{enumerate}
\item $x\notin Z_{\chi}$.
\item $x\in Z_{\chi}$ and $\cform(\chi)_{x'}$ is a part of a basis of the free
$\dvr_{Z^{1/p}_{\chi},x'}$-module $\Omega^{1}_{X}(R'_{\chi})\otimes_{\dvr_{X}}
\dvr_{Z_{\chi}^{1/p},x'}$,
where $x'$ is the unique point on $Z^{1/p}$ lying above $x$.
\end{enumerate}

We say that $(X,U,\chi)$ is {\it non-degenerate} if $(X,U,\chi)$ is non-degenerate at
every point on $X$.

\item Let $x$ be a closed point of $D$ and $i$ an element of $I_{\mW,\chi,x}$.
Let $x'$ be the unique closed point of $D_{i}^{1/p}$ lying above $x$.
Let $n'$ be the maximal integer such that
\begin{equation}
\cform(\chi)|_{D_{i}^{1/p},x'}\in \mathfrak{m}_{x'}^{n'}\Omega^{1}_{X}(R'_{\chi})_{x}\otimes_{\dvr_{X,x}}
\dvr_{D_{i}^{1/p},x'}, \notag
\end{equation}
where $\mathfrak{m}_{x'}$ is the maximal ideal at $x'$.
We define $\ord'(\chi;x,D_{i})$ by $\ord'(\chi;x,D_{i})=n'/2$.
\item We say that $(X,U,\chi)$ is {\it strongly} non-degenerate at $x\in X$ if the following conditions are satisfied:
\begin{enumerate}
\item $(X,U,\chi)$ is non-degenerate at $x$.
\item $I_{x}=I_{\mW,\chi,x}$ or $I_{x}=I_{\mT,\chi,x}$.
\end{enumerate}
We say that $(X,U,\chi)$ is {\it strongly} non-degenerate if $(X,U,\chi)$ is strongly non-degenerate at every point on $X$.
\end{enumerate}
\end{defn}

We note that if $\chi|_{K_{i}}$ is of usual type for every $i\in I_{\mW,\chi,x}$ then $(X,U,\chi)$ is 
non-degenerate at $x\in X$ if and only if one of the following conditions holds:
\begin{enumerate}
\item $x\notin Z_{\chi}$.
\item $x\in Z_{\chi}$ and $\cform(\chi)_{x}$ is a part of a basis of the free $\dvr_{Z_{\chi},x}$-module
$\Omega^{1}_{X}(R_{\chi}')|_{Z_{\chi},x}$. 
\end{enumerate}
If $\chi|_{K_{i}}$ is of usual type for $i\in I_{\mW,\chi,x}$, then
we may regard $\cform(\chi)|_{D_{i}^{1/p}}$ as a global section of $\Omega^{1}_{X}(R_{\chi}')|_{D_{i}}$.
In this case, the order $\ord'(\chi; x,D_{i})$ for a closed point $x$ of $D_{i}$ is
none other than the maximal integer $n$ such that
\begin{equation}
\cform(\chi)|_{D_{i}^{1/p},x'} \in \mathfrak{m}_{x}^{n}\Omega^{1}_{X}(R_{\chi}')|_{D_{i},x}, \notag
\end{equation}
where $x'$ is the unique closed point of $D_{i}^{1/p}$ lying above $x$ and $\mathfrak{m}_{x}$ is the maximal ideal at $x$.  

Let $\Lambda$ be a finite field of characteristic $\neq p$. 
If $\chi\in H^{1}_{\et}(U,\mathbf{Q}/\mathbf{Z})$ is the character $\pi_{1}^{ab}(U)\rightarrow \Lambda^{\times}$
associated to a smooth sheaf $\mf$ of $\Lambda$-modules of rank $1$ on $U$
and regarded as an element of $H^{1}_{\et}(U,\mathbf{Q}/\mathbf{Z})$ 
by an inclusion $\Lambda^{\times}\rightarrow \mathbf{Q}/\mathbf{Z}$, then
the definition of non-degeneration is the same as \cite[Definition 4.2]{sa1}
by \cite[Proposition 2.12 (ii)]{ya}.
The definition of strong non-degeneration is the same as that in \cite[Subsection 4.2]{sa4}
by \cite[Theorem 3.1]{ya}.

For $\chi\in H^{1}_{\et}(U,\mathbf{Q}/\mathbf{Z})$ and $x\in X$, 
we prove that $(X,U,\chi)$ is non-degenerate at $x$ if and only if $(X,U,\chi)$ is strongly non-degenerate at $x$
in Lamma \ref{lemint} (i) later.

\begin{lem}
\label{lemndg}
Let $\chi$ be an element of $H^{1}_{\et}(U,\mathbf{Q}/\mathbf{Z})$.
\begin{enumerate}
\item Let $x$ be a closed point of $Z_{\chi}$.
The following are equivalent:
\begin{enumerate}
\item $(X,U,\chi)$ is non-degenerate at $x$.
\item $\ord'(\chi;x,D_{i})=0$ for some $i\in I_{\mW,\chi,x}$. 
\item $\ord'(\chi;x,D_{i})=0$ for every $i\in I_{\mW,\chi,x}$.
\end{enumerate}
\item The set of points where $(X,U,\chi)$ is not non-degenerate is a closed subset of $X$ 
of codimension $\ge 2$.
\item $(X,U,\chi)$ is non-degenerate if and only if $(X,U,\chi)$ is non-degenerate at every closed point $x$ of $Z_{\chi}$.
\end{enumerate}
\end{lem}

\begin{proof}
The assertions hold similarly as the proof of Lemma \ref{lemclean}.
\end{proof}

We give a form and a partial description of the characteristic cycle of a rank $1$ sheaf 
on a smooth surface using ramification theory.

Let $\mf$ be a smooth sheaf of $\Lambda$-modules of rank $1$ on $U$,
where $\Lambda$ is a finite field of characteristic $\neq p$.
Let $\chi\colon \pi_{1}^{\ab}(U)\rightarrow \Lambda^{\times}$ be the character corresponding to $\mf$.
We fix an inclusion $\psi\colon \Lambda^{\times}\rightarrow \mathbf{Q}/\mathbf{Z}$ 
and regard $\chi$ as an element of $H^{1}_{\et}(U,\mathbf{Q}/\mathbf{Z})$ by $\psi$.
Then the total dimension divisor $R_{\chi}'$ is equal to the total dimension divisor $DT(\mf)$ defined in \cite[Definition 3.5.2]{sa1} by \cite[Theorem 3.1]{ya}.

For an element $i$ of $I_{\mW,\chi}$ such that $\chi|_{K_{i}}$ is of usual type,
let $L_{i,\chi}'$ be the sub line bundle of $T^{\ast}X\times_{X}D_{i}$ 
defined by the unique $\dvr_{D_{i}}$-submodule of $\Omega^{1}_{X}|_{D_{i}}$
which is locally a direct summand of $\Omega^{1}_{X}|_{D_{i}}$ of rank $1$
containing $\dvr_{X}(-R_{\chi}')|_{D_{i}}\cdot \cform(\chi)|_{D_{i}^{1/p}}$, 
where we regard $\cform(\chi)|_{D_{i}^{1/p}}$ as a global section of $\Omega^{1}_{X}(R_{\chi}')|_{D_{i}}$.
For an element $i$ of $I_{\mW,\chi}$ such that $\chi|_{K_{i}}$ is of exceptional type,
let $L_{i,\chi}''$ be the sub line bundle of 
$T^{\ast}X\times_{X}D_{i}^{1/p}$ 
defined by the unique $\dvr_{D_{i}^{1/p}}$-submodule of 
$\Omega^{1}_{X}\otimes_{\dvr_{X}}\dvr_{D_{i}^{1/p}}$ which is locally a direct summand of 
$\Omega^{1}_{X}\otimes_{\dvr_{X}}\dvr_{D_{i}^{1/p}}$ of rank $1$ containing
$(\dvr_{X}(-R_{\chi}')\otimes_{\dvr_{X}}\dvr_{D_{i}^{1/p}})\cdot \cform(\chi)|_{D_{i}^{1/p}}$
and let $[L_{i,\chi}']$ be the push-forward of $[L_{i,\chi}'']$ by the canonical morphism
$T^{\ast}X \times_{X} D^{1/p}\rightarrow T^{\ast}X \times_{X}D \rightarrow T^{\ast}X$
in the sense of intersection theory.
For an element $i$ of $I_{\mT,\chi}$, 
let $L_{i,\chi}'\subset \cotx\times_{X}D_{i}$ be the conormal bundle $\nori$ of $D_{i}$ over $X$.

Let $j\colon U\rightarrow X$ be the open immersion and $|D|$ the set of closed points of $D$.
We note that $j$ is affine.
We put $r_{i}'=\dt(\chi|_{K_{i}})$ if $\chi|_{K_{i}}$ is of usual type or $i\in I_{\mT,\chi}$
and $r_{i}'=\dt(\chi|_{K_{i}})/2=1$ if $\chi|_{K_{i}}$ is of exceptional type for $i\in I$.
If $X$ is purely of dimension $2$, then the characteristic cycle $CC(j_{!}\mf)$ is of the form  
\begin{equation}
\label{saitoccro}
CC(j_{!}\mathcal{F})=[T^{\ast}_{X}X]+\sum_{i\in I}r_{i}'[L_{i,\chi}']+\sum_{x\in |D|}u_{x}[\spf]
\end{equation}
by \cite[Theorem 7.14]{sa4}, \cite[Corollary 2.13 (i), Theorem 3.1]{ya},
and Lemma \ref{lemndg} (ii).
Here, since $j$ is an affine open immersion and 
$SS(j_{!}\mf)$ has finitely many irreducible components, 
the coefficient $u_{x}$ is a non-negative integer for every $x\in |D|$ by Lemma \ref{suppCC}
and is equal to $0$ except for finitely many $x\in |D|$.
If $x$ is a closed point of $Z_{\chi}$ and if $(X,U,\chi)$ is strongly non-degenerate at $x$,
then $u_{x}=0$ by \cite[Theorem 7.14]{sa4}.
If $x\in |D|$ and if $x\notin Z_{\chi}$, then $u_{x}=\sharp(I_{x})-1$ by \cite[Theorem 7.14]{sa4}.

\subsection{Cleanliness and non-degeneration}
\label{sscrc}

In this subsection, we study more about the clean and non-degenerate conditions.
We briefly recall the construction of the refined Swan conductor and 
the characteristic form using sheaves of Witt vectors in \cite[Subsection 1.3--1.4]{ya}.

Let $X$ be a smooth scheme over a perfect field $k$ of characteristic $p>0$.
Let $D$ be a divisor on $X$ with simple normal crossings and 
$\{D_{i}\}_{i\in I}$ the irreducible components of $D$.
We put $U=X-D$ and let $j\colon U\rightarrow X$ be the open immersion.
Let $\dvr_{K_{i}}=\hat{\dvr}_{X,\mathfrak{p}_{i}}$ be the completion of the local ring 
$\dvr_{X,\mathfrak{p}_{i}}$ at the generic point $\mathfrak{p}_{i}$ of $D_{i}$ and
$K_{i}=\Frac \dvr_{K_{i}}$ the local field at $\mathfrak{p}_{i}$ for $i\in I$.

Let $W_{s}(\dvr_{{U}_{\et}})$ be the \'{e}tale sheaf of Witt vectors of length $s\in \mathbf{Z}_{\ge 0}$ 
on $U$ and $W_{s}(\dvr_{U})$ the Zariski sheaf of Witt vectors of length $s$ on $U$.
Let $\epsilon\colon X_{\et}\rightarrow X_{\mathrm{Zar}}$ be the canonical mapping from the \'{e}tale
site of $X$ to the Zariski site of $X$.
By the exact sequence 
\begin{equation}
0\rightarrow W_{s}(\mathbf{F}_{p})\rightarrow W_{s}(\dvr_{U_{\et}})\xrightarrow{F-1}
W_{s}(\dvr_{U_{\et}})\rightarrow 0 \notag 
\end{equation}
of \'{e}tale sheaves on $U$ and $R^{1}(\epsilon\circ j)_{*}W_{s}(\dvr_{U_{\et}})=0$, 
we have an exact sequence
\begin{equation}
0\rightarrow j_{*}W_{s}(\mathbf{F}_{p})\rightarrow j_{*}W_{s}(\dvr_{U})\xrightarrow{F-1}
j_{*}W_{s}(\dvr_{U})\xrightarrow{\delta_{s}} R^{1}(\epsilon\circ j)_{*}\mathbf{Z}/p^{s}\mathbf{Z}
\rightarrow 0 \notag
\end{equation}
of Zariski sheaves on $X$.
We note that the canonical morphism
\begin{equation}
\label{hov}
H^{1}_{\et}(V\cap U,\mathbf{Z}/p^{s}\mathbf{Z})\rightarrow 
\Gamma(V,R^{1}(\epsilon\circ j)_{*}\mathbf{Z}/p^{s}\mathbf{Z}) 
\end{equation}
is an isomorphism for every open subset $V$ of $X$
by the spectral sequence $E^{q_{1},q_{2}}_{2}=H^{q_{1}}_{\mathrm{Zar}}
(V,R^{q_{2}}(\epsilon\circ j)_{*}\mathbf{Z}/p^{s}\mathbf{Z})\Rightarrow H^{q_{1}+q_{2}}_{\et}
(V\cap U,\mathbf{Z}/p^{s}\mathbf{Z})$ and the equalities $E_{2}^{1,0}=E_{2}^{2,0}=0$.

For a section $a\in \Gamma(U,W_{s}(\dvr_{U}))$, let $a|_{K_{i}}$ be the image of $a$ in
$W_{s}(K_{i})$.
For an element $\chi\in H^{1}_{\et}(U,\mathbf{Z}/p^{s}\mathbf{Z})$,
let $\chi|_{K_{i}}$ be the image of $\chi$ in $H^{1}(K_{i},\mathbf{Z}/p^{s}\mathbf{Z})$.

\begin{defn}[{\cite[Definition 1.25]{ya}}]
\label{deffilshf}
Let $R=\sum_{i\in I}n_{i}D_{i}$ be a linear combination with integral coefficients $n_{i}\in \mathbf{Z}_{\ge 0}$ for $i\in I$. 
Let $j_{i}\colon \Spec K_{i}\rightarrow X$ be the canonical morphism for $i\in I$.
\begin{enumerate}
\item We define a subsheaf $\fillog_{R}j_{*}W_{s}(\dvr_{U})$ of the sheaf $j_{*}W_{s}(\dvr_{U})$
to be the pull-back of the subsheaf $\bigoplus_{i\in I}j_{i *}\fillog_{n_{i}}W_{s}(K_{i})$ of
$\bigoplus_{i\in I}j_{i*}W_{s}(K_{i})$
by the canonical morphism $j_{*}W_{s}(\dvr_{U})\rightarrow \bigoplus_{i\in I}j_{i*}W_{s}(K_{i})$.
\item We define a subsheaf $\fillog_{R}R^{1}(\epsilon\circ j)_{*}\mathbf{Z}/p^{s}\mathbf{Z}$
of the sheaf $R^{1}(\epsilon\circ j)_{*}\mathbf{Z}/p^{s}\mathbf{Z}$ to be the image of the subsheaf
$\fillog_{R}j_{*}W_{s}(\dvr_{U})$ of $j_{*}W_{s}(\dvr_{U})$ 
by $\delta_{s}\colon j_{*}W_{s}(\dvr_{U})\rightarrow R^{1}(\epsilon
\circ j)_{*}\mathbf{Z}/p^{s}\mathbf{Z}$.
\item We define a subsheaf $\fillog_{R}j_{*}\Omega_{U}^{1}$ of the sheaf $j_{*}\Omega_{U}^{1}$
to be $\Omega^{1}_{X}(\log D)(R)$.
\end{enumerate}
\end{defn}

We note that the subsheaf $\fillog_{R}R^{1}(\epsilon\circ j)_{*}\mathbf{Z}/p^{s}\mathbf{Z}$ is equal to 
the pull-back of the subsheaf $\bigoplus_{i\in I}j_{i*}\fillog_{n_{i}}H^{1}(K_{i},\mathbf{Q}/\mathbf{Z})$
of $\bigoplus_{i\in I}j_{i*}H^{1}(K_{i},\mathbf{Q}/\mathbf{Z})$
by the canonical morphism $R^{1}(\epsilon \circ j)_{*}\mathbf{Z}/p^{s}\mathbf{Z}
\rightarrow \bigoplus_{i\in I}j_{i*}H^{1}(K_{i},\mathbf{Q}/\mathbf{Z})$
by \cite[Proposition 1.31 (i)]{ya}.
For a linear combination $R=\sum_{i\in I}n_{i}D_{i}$ with integral coefficients 
$n_{i}\in \mathbf{Z}_{\ge 0}$ for $i\in I$, 
we write $\fillog_{R}H^{1}_{\et}(U,\mathbf{Z}/p^{s}\mathbf{Z})\subset H^{1}_{\et}(U,\mathbf{Z}/p^{s}
\mathbf{Z})$ for $\Gamma(X,\fillog_{R}R^{1}(\epsilon\circ j)_{*}\mathbf{Z}/p^{s}\mathbf{Z})$
identified with a submodule of $H^{1}_{\et}(U,\mathbf{Z}/p^{s}\mathbf{Z})$ by (\ref{hov}).
Then $\fillog_{R}H^{1}_{\et}(U,\mathbf{Z}/p^{s}\mathbf{Z})$ consists of 
$\chi\in H^{1}_{\et}(U,\mathbf{Z}/p^{s}\mathbf{Z})\subset H^{1}_{\et}(U,\mathbf{Q}/\mathbf{Z})$ 
such that $\chi|_{K_{i}}\in \fillog_{n_{i}}H^{1}(K_{i},\mathbf{Q}/\mathbf{Z})$
for every $i\in I$.
Since $\fillog_{R}R^{1}(\epsilon \circ j)_{*}\mathbf{Z}/p^{s}\mathbf{Z}$ is the image of 
$\fillog_{R}j_{*}W_{s}(\dvr_{U})$, there is a global section of
$\fillog_{R}j_{*}W_{s}(\dvr_{U})$ whose image in $\Gamma(X,R^{1}(\epsilon \circ j)_{*}
\mathbf{Z}/p^{s}\mathbf{Z})\simeq H^{1}_{\et}(U,\mathbf{Z}/p^{s}\mathbf{Z})$ 
is $\chi$ for $\chi\in \fillog_{R}H^{1}_{\et}(U,\mathbf{Z}/p^{s}\mathbf{Z})$ if we shrink $X$ if necessary.

We consider the morphism
\begin{equation}
\label{fsdsh}
-F^{s-1}d\colon j_{*}W_{s}(\dvr_{U})\rightarrow j_{*}\Omega^{1}_{U}; \
(a_{s-1},\ldots,a_{0})\mapsto -\sum_{i=0}^{s-1}a_{i}^{p^{i}-1}da_{i}.
\end{equation}
For a linear combination $R=\sum_{i\in I}n_{i}D_{i}$ with integral coefficients 
$n_{i}\in \mathbf{Z}_{\ge 0}$ for $i\in I$, the morphism $-F^{s-1}d$ induces the morphism 
$\fillog_{R}j_{*}W_{s}(\dvr_{U})\rightarrow \fillog_{R}j_{*}\Omega^{1}_{U}$.
For linear combinations $R=\sum_{i\in I}n_{i}D_{i}$ and $R'=\sum_{i\in I}n_{i}'D_{i}$
with integral coefficients $n_{i},n_{i}'\in \mathbf{Z}_{\ge 0}$ such that $n_{i}\ge n_{i}'$ for $i\in I$,
we put $\grlog_{R/R'}=\fillog_{R}/\fillog_{R'}$ and
the morphism $-F^{s-1}d$ induces the morphism
\begin{equation}
\varphi^{(R/R')}\colon \grlog_{R/R'}j_{*}W_{s}(\dvr_{U})\rightarrow \grlog_{R/R'}j_{*}\Omega^{1}_{U}.
\notag
\end{equation}
By \cite[Proposition 1.31 (ii)]{ya}, if $n_{i}-1\le n_{i}'\le n_{i}$ for every $i\in I$,
then there exists a unique injection $\phi_{s}^{(R/R')}\colon \grlog_{R/R'}R^{1}(\epsilon\circ j)_{*}
\mathbf{Z}/p^{s}\mathbf{Z}\rightarrow \grlog_{R/R'}j_{*}\Omega_{U}^{1}$ such that the 
following diagram is commutative:
\begin{equation}
\xymatrix{
\grlog_{R/R'}j_{*}W_{s}(\dvr_{U}) \ar[rr]^-{\varphi_{s}^{(R/R')}} \ar[dr]
&& \grlog_{R/R'}j_{*}\Omega_{U}^{1} \\
& \grlog_{R/R'}R^{1}(\epsilon\circ j)_{*}\mathbf{Z}/p^{s}\mathbf{Z}.
\ar[ur]_-{\phi_{s}^{(R/R')}} & 
} \notag
\end{equation}
For an element $\chi\in H^{1}_{\et}(U,\mathbf{Q}/\mathbf{Z})$,
take $s\in \mathbf{Z}_{\ge 0}$ such that the $p$-part of $\chi$ is of order $p^{s}$.
Then the refined Swan conductor $\rsw(\chi)\in \Gamma(Z_{\chi},\Omega^{1}_{X}(\log D)(R_{\chi})
|_{Z_{\chi}})$ is the image of the $p$-part of $\chi$ by the composition
\begin{align}
\fillog_{R_{\chi}}H^{1}_{\et}(U,\mathbf{Z}/p^{s}\mathbf{Z})=
\Gamma &(X,\fillog_{R_{\chi}}R^{1}(\epsilon\circ j)_{*}\mathbf{Z}/p^{s}\mathbf{Z})
\rightarrow \Gamma(X,\grlog_{R_{\chi}/(R_{\chi}-Z_{\chi})}R^{1}(\epsilon\circ j)_{*}\mathbf{Z}/p^{s}\mathbf{Z}) \notag \\
&\xrightarrow{\phi_{s}^{(R_{\chi}/(R_{\chi}-Z_{\chi}))}(X)}\Gamma(X,\grlog_{R_{\chi}/(R_{\chi}-Z_{\chi})}
j_{*}\Omega^{1}_{U})=\Gamma(Z_{\chi}, \Omega^{1}_{X}(\log D)(R_{\chi})|_{Z_{\chi}}). \notag
\end{align}

\begin{defn}[{\cite[Definition 1.34]{ya}}]
Let $R=\sum_{i\in I}m_{i}D_{i}$ be a linear combination with integral coefficients $m_{i}\in \mathbf{Z}_{\ge 1}$ for $i\in I$. 
Let $j_{i}\colon \Spec K_{i}\rightarrow X$ be the canonical morphism for $i\in I$.
\begin{enumerate}
\item We define a subsheaf $\fil_{R}j_{*}W_{s}(\dvr_{U})$ of the sheaf $j_{*}W_{s}(\dvr_{U})$
to be the pull-back of the subsheaf $\bigoplus_{i\in I}j_{i *}\fil_{m_{i}}W_{s}(K_{i})$
of $\bigoplus_{i\in I}j_{i*}W_{s}(K_{i})$
by the canonical morphism $j_{*}W_{s}(\dvr_{U})\rightarrow \bigoplus_{i\in I}j_{i*}W_{s}(K_{i})$.
\item We define a subsheaf $\fil_{R}R^{1}(\epsilon\circ j)_{*}\mathbf{Z}/p^{s}\mathbf{Z}$
of the sheaf $R^{1}(\epsilon\circ j)_{*}\mathbf{Z}/p^{s}\mathbf{Z}$ to be the image of the subsheaf
$\fil_{R}j_{*}W_{s}(\dvr_{U})$ of $j_{*}W_{s}(\dvr_{U})$ 
by $\delta_{s}\colon j_{*}W_{s}(\dvr_{U})\rightarrow R^{1}(\epsilon
\circ j)_{*}\mathbf{Z}/p^{s}\mathbf{Z}$.
\item We define a subsheaf $\fil_{R}j_{*}\Omega_{U}^{1}$ of the sheaf $j_{*}\Omega_{U}^{1}$
to be $\Omega^{1}_{X}(R)$.
\end{enumerate}
\end{defn}

We note that the subsheaf $\fil_{R}R^{1}(\epsilon\circ j)_{*}\mathbf{Z}/p^{s}\mathbf{Z}$ is equal to the 
pull-back of the subsheaf $\bigoplus_{i\in I}j_{i*}\fil_{m_{i}}H^{1}(K_{i},\mathbf{Q}/\mathbf{Z})$
of $\bigoplus_{i \in I}j_{i*}H^{1}(K_{i},\mathbf{Q}/\mathbf{Z})$
by the canonical morphism $R^{1}(\epsilon \circ j)_{*}\mathbf{Z}/p^{s}\mathbf{Z}
\rightarrow \bigoplus_{i\in I}j_{i*}H^{1}(K_{i},\mathbf{Q}/\mathbf{Z})$
by \cite[Proposition 1.40 (i)]{ya}.
For a linear combination $R=\sum_{i\in I}m_{i}D_{i}$ with integral coefficients $m_{i}\in \mathbf{Z}_{\ge 1}$ for $i\in I$,
we write $\fil_{R}H^{1}_{\et}(U,\mathbf{Z}/p^{s}\mathbf{Z})\subset H^{1}_{\et}(U,\mathbf{Z}/p^{s}\mathbf{Z})$ for $\Gamma(X,\fil_{R}R^{1}(\epsilon\circ j)_{*}\mathbf{Z}/p^{s}\mathbf{Z})$
identified with a submodule of $H^{1}_{\et}(U,\mathbf{Z}/p^{s}\mathbf{Z})$ by (\ref{hov}).
Then $\fil_{R}H^{1}_{\et}(U,\mathbf{Z}/p^{s}\mathbf{Z})$ consists of 
$\chi\in H^{1}_{\et}(U,\mathbf{Z}/p^{s}\mathbf{Z})\subset H^{1}_{\et}(U,\mathbf{Q}/\mathbf{Z})$
such that $\chi|_{K_{i}}\in \fil_{m_{i}}H^{1}(K_{i},\mathbf{Q}/\mathbf{Z})$ for every $i\in I$.

For linear combinations $R=\sum_{i\in I}m_{i}D_{i}$ and $R'=\sum_{i\in I}m_{i}'D_{i}$
with integral coefficients $m_{i},m_{i}'\in \mathbf{Z}_{\ge 1}$ such that $m_{i}\ge m_{i}'$ for $i\in I$,
we put $\gr_{R/R'}=\fil_{R}/\fil_{R'}$.
If $m_{i}-1\le m_{i}'\le m_{i}$ for every $i\in I$, then we put $D^{(R/R')}=R-R'$ and 
there exists a unique morphism 
\begin{equation}
\varphi'^{(R/R')}\colon \gr_{R/R'}j_{*}W_{s}(\dvr_{U})\rightarrow \gr_{R/R'}j_{*}\Omega^{1}_{U}\otimes_{\dvr_{D^{(R/R')}}}\dvr_{{D^{(R/R')}}^{1/p}},
\notag
\end{equation}
such that if $t_{i}$ is a local equation of $D_{i}$ for $i\in I$ then locally 
\begin{equation}
\label{varphip}
\varphi'^{(R/R')}_{s}((a_{s-1},\ldots,a_{0}))=
\begin{cases}
-\sum_{i=0}^{s-1}a_{i}^{p^{i}-1}da_{i}+\sum_{i\colon (m_{i},m_{i}')=(2,1)}\sqrt{\overline{t_{i}^{2}a_{0}}}
dt_{i}/t_{i}^{2} & (p=2), \\
-\sum_{i=0}^{s-1}a_{i}^{p^{i}-1}da_{i} & (p\neq 2)
\end{cases} 
\end{equation}
by \cite[Subsection 1.4]{ya}.
Further, by \cite[Proposition 1.40 (ii)]{ya}, 
there exists a unique injection $\phi_{s}'^{(R/R')}\colon \gr_{R/R'}R^{1}(\epsilon\circ j)_{*}
\mathbf{Z}/p^{s}\mathbf{Z}\rightarrow \gr_{R/R'}j_{*}\Omega_{U}^{1}\otimes_{\dvr_{D^{(R/R')}}}
\dvr_{{D^{(R/R')}}^{1/p}}$ such that the 
following diagram is commutative:
\begin{equation}
\xymatrix{
\gr_{R/R'}j_{*}W_{s}(\dvr_{U}) \ar[rr]^-{\varphi_{s}'^{(R/R')}} \ar[dr]
&& \gr_{R/R'}j_{*}\Omega_{U}^{1}\otimes_{\dvr_{D^{(R/R')}}}\dvr_{{D^{(R/R')}}^{1/p}} \\
& \gr_{R/R'}R^{1}(\epsilon\circ j)_{*}\mathbf{Z}/p^{s}\mathbf{Z}.
\ar[ur]_-{\phi_{s}'^{(R/R')}} & 
} \notag
\end{equation}
For an element $\chi\in H^{1}_{\et}(U,\mathbf{Q}/\mathbf{Z})$,
take $s\in \mathbf{Z}_{\ge 0}$ such that the $p$-part of $\chi$ is of order $p^{s}$.
Then the characteristic form $\cform(\chi)\in \Gamma(Z_{\chi}^{1/p},\Omega^{1}_{X}(R'_{\chi})
\otimes_{\dvr_{X}}\dvr_{Z_{\chi}^{1/p}})$ is the image of the $p$-part of $\chi$ by the composition
\begin{align}
\fil_{R_{\chi}'}H^{1}_{\et}(U,&\mathbf{Z}/p^{s}\mathbf{Z})=
\Gamma (X,\fil_{R'_{\chi}}R^{1}(\epsilon\circ j)_{*}\mathbf{Z}/p^{s}\mathbf{Z})
\rightarrow \Gamma(X,\gr_{R'_{\chi}/(R'_{\chi}-Z_{\chi})}R^{1}(\epsilon\circ j)_{*}\mathbf{Z}/p^{s}\mathbf{Z}) \notag \\
&\xrightarrow{\phi_{s}^{(R'_{\chi}/(R'_{\chi}-Z_{\chi}))}(X)}\Gamma(X,\gr_{R'_{\chi}/(R'_{\chi}-Z_{\chi})}
j_{*}\Omega^{1}_{U}\otimes_{\dvr_{Z_{\chi}}}Z_{\chi}^{1/p})=\Gamma(Z_{\chi}^{1/p}, \Omega^{1}_{X}(R'_{\chi})\otimes_{\dvr_{X}}\dvr_{Z_{\chi}^{1/p}}). \notag
\end{align}

\begin{lem}
\label{alnlsch}
Let $a\in \Gamma(X,j_{*}W_{s}(\dvr_{U}))$ be a global section
and let $\chi$ be an element of $H^{1}_{\et}(U,\mathbf{Q}/\mathbf{Z})$
whose $p$-part is the image of $a$ by $\delta_{s}$.
If $a\in \Gamma(X,\fillog_{R_{\chi}}j_{*}W_{s}(\dvr_{U}))$,
then we have $a\in \Gamma(X,\fil_{R_{\chi}'}j_{*}W_{s}(\dvr_{U}))$.
\end{lem}

\begin{proof}
We put $R_{\chi}=\sum_{i\in I}n_{i}D_{i}$ and $R_{\chi}'=\sum_{i\in I}m_{i}D_{i}$.
Then we have $a|_{K_{i}}\in \fillog_{n_{i}}W_{s}(K_{i})$ for every $i\in I$.
By Corollary \ref{alnl}, we have $a|_{K_{i}}\in \fil_{m_{i}}W_{s}(K_{i})$ for every $i\in I$.
Hence the assertion holds.
\end{proof}

We study a relation between the refined Swan conductor $\rsw(\chi)$ and
the characteristic form $\cform(\chi)$ for $\chi\in H^{1}_{\et}(U,\mathbf{Q}/\mathbf{Z})$.
For an element $\chi$ of $H^{1}_{\et}(U,\mathbf{Q}/\mathbf{Z})$,
let $I_{\mI,\chi}$ and $I_{\mII,\chi}$ be the subsets of $I$ consisting of $i\in I$
such that $\chi|_{K_{i}}$ is of type $\mI$ and of type $\mII$ respectively.
We note $I_{\mW,\chi}=I_{\mI,\chi}\sqcup I_{\mII,\chi}$ 
(Definition \ref{defswdiv}).
We put $D_{\mT,\chi}=\bigcup_{i\in I_{\mT,\chi}}D_{i}$ (loc.\ cit.).
Then we have $D=Z_{\chi}\cup D_{T,\chi}$.
By Lemma \ref{alnlsch}, both the refined Swan conductor $\rsw(\chi)$ and the characteristic form 
$\cform(\chi)$ are locally defined by the same section of $\fillog_{R_{\chi}}j_{*}W_{s}(\dvr_{U})$.

Suppose $D=Z_{\chi}$.
If $I_{\mI,\chi}=\emptyset$, then we have $R_{\chi}'=R_{\chi}$.
In this case, the image of $\cform(\chi)$ by the canonical morphism
$\Omega_{X}^{1}(R_{\chi}')\otimes_{\dvr_{X}}\dvr_{Z_{\chi}^{1/p}}\rightarrow \Omega_{X}^{1}(\log D)(R_{\chi})\otimes_{\dvr_{X}}\dvr_{Z_{\chi}^{1/p}}$ is in $\Omega_{X}^{1}(\log D)(R_{\chi})|_{Z_{\chi}}$
and is the refined Swan conductor $\rsw(\chi)$.
If $I_{\mI,\chi}\neq \emptyset$, then we have the canonical morphism
$\dvr_{Z_{\chi}}\cdot \rsw(\chi)\rightarrow \Omega^{1}_{X}(R_{\chi}')|_{Z_{\chi}}$
and the image $-F^{s-1}da$ in $\Omega^{1}_{X}(R_{\chi}')|_{Z_{\chi}}$ 
of a section $a$ of $\fillog_{R_{\chi}}j_{*}W_{s}(\dvr_{U})$ whose image in $H^{1}_{\et}(U,\mathbf{Q}/\mathbf{Z})$
is locally the $p$-part of $\chi$
is locally the image of $\rsw(\chi)$ by this morphism.

If $D\neq Z_{\chi}$, then the $p$-part of $\chi$ is the image in 
$H^{1}_{\et}(U,\mathbf{Q}/\mathbf{Z})$ of an element 
$\chi'\in H^{1}_{\et}(V,\mathbf{Q}/\mathbf{Z})$ where $V=X-Z_{\chi}$,
and we have $Z_{\chi'}=Z_{\chi}$ and $R_{\chi}'=R'_{\chi'}+D_{\mT,\chi}$. 
The refined Swan conductor $\rsw(\chi)$ is the image of $\rsw(\chi')$ by
$\Omega^{1}_{X}(\log Z_{\chi})(R_{\chi})|_{Z_{\chi}}\rightarrow \Omega^{1}_{X}(\log D)(R_{\chi})|_{Z_{\chi}}$ 
and the characteristic form $\cform(\chi)$ is the image of $\cform(\chi')$
by $\Omega^{1}_{X}(R'_{\chi'})|_{Z_{\chi'}^{1/p}}\rightarrow 
\Omega^{1}_{X}(R'_{\chi})|_{Z_{\chi}^{1/p}}$.

\begin{lem}
\label{lemint}
Let $\chi$ be an element of $H^{1}_{\et}(U,\mathbf{Q}/\mathbf{Z})$.
\begin{enumerate}
\item Let $i$ be an element of $I_{\mT,\chi}$.
Then the restriction $\cform(\chi)|_{D_{i}^{1/p}\cap Z_{\chi}^{1/p}}$ is $0$.
Consequently $(X,U,\chi)$ is not non-degenerate at any point on $D_{\mT,\chi}\cap Z_{\chi}$.
\item Let $i$ be an element of $I_{\mW,\chi}$.
Then the following are equivalent:
\begin{enumerate}
\item $i\in I_{\mII,\chi}$.
\item The morphism $\xi_{i}(\chi)\colon \dvr_{X}(-R_{\chi})|_{D_{i}}\rightarrow \dvr_{D_{i}}$ (\ref{xichi}) is $0$.
\item $[L_{i,\chi}']\neq [T^{*}_{D_{i}}X]$.
\end{enumerate}
\item Let $i$ and $i'$ be two different elements of $I_{\mI,\chi}$.
Then the restriction $\cform(\chi)|_{D_{i}^{1/p}\cap D_{i'}^{1/p}}$ is $0$.
\end{enumerate}
\end{lem}

\begin{proof}
We may assume that $\chi$ is of order $p^{s}$ for some integer $s\ge 0$.

(i) Since $D\neq Z_{\chi}$, the characteristic form $\cform(\chi)$ is the image of a global section
by $\Omega^{1}_{X}(R'_{\chi}-D_{T,\chi})|_{Z_{\chi}^{1/p}}\rightarrow 
\Omega^{1}_{X}(R'_{\chi})|_{Z_{\chi}^{1/p}}$.
Hence the assertion holds. 

(ii) We may assume $D=D_{i}$.
Then the condition (b) is equivalent to that the image of $\dvr_{X}(-R_{\chi})|_{D}$ by the 
multiplication by $\rsw(\chi)|_{D}$ is contained in the image 
of $\Omega^{1}_{X}|_{D}$ in $\Omega^{1}_{X}(\log D)|_{D}$.
Namely the condition (b) is equivalent to that
$(-F^{s-1}da)|_{D}$ is locally contained in $\Omega^{1}_{X}(R_{\chi})|_{D}$
for a section $a$ of $\fillog_{R_{\chi}}j_{*}W_{s}(\dvr_{U})$ whose image in 
$H^{1}_{\et}(U,\mathbf{Q}/\mathbf{Z})$ is locally $\chi$.

Suppose that $\chi|_{K_{i}}$ is of exceptional type.
Then we have $i\in I_{\mII,\chi}$ and we note $R'_{\chi}=R_{\chi}$.
Since $\rsw(\chi)$ is the image of $\cform(\chi)$ by the canonical morphism
$\Omega^{1}_{X}(R'_{\chi})\otimes_{\dvr_{X}}\dvr_{D^{1/p}}\rightarrow \Omega^{1}_{X}(\log D)(R_{\chi})\otimes_{\dvr_{X}}\dvr_{D^{1/p}}$,
the image of $\xi_{i}(\chi)$ is $0$.

Suppose that $\chi|_{K_{i}}$ is of usual type.
Then $(-F^{s-1}da)|_{D}$ is a section of $\Omega^{1}_{X}(R_{\chi}')|_{D}$ by Lemma \ref{alnlsch}
and is locally the characteristic form $\cform(\chi)|_{D}$ for a section $a$ of $\fillog_{R_{\chi}}j_{*}W_{s}(\dvr_{U})$ 
whose image in $H^{1}_{\et}(U,\mathbf{Q}/\mathbf{Z})$ is locally $\chi$.
Since $R_{\chi}=R_{\chi}'$ is equivalent to the condition (a),
the equivalence of (a) and (b) holds by $\cform(\chi)|_{D}\neq 0$.

We prove the equivalence of (a) and (c).
If $i\in I_{\mI,\chi}$, then $\chi|_{K_{i}}$ is of usual type and $\cform(\chi)$ is the image of
$\rsw(\chi)$ by $\Omega^{1}_{X}(\log D)(R_{\chi})|_{D}\rightarrow 
\Omega^{1}_{X}(R_{\chi}+D)|_{D}$.
Since the image of the morphism $\Omega^{1}_{X}(\log D)|_{D}\rightarrow 
\Omega^{1}_{X}(D)|_{D}$ is $\mathcal{C}_{D/X}(D)$, 
where $\mathcal{C}_{D/X}$ denotes the conormal sheaf of $D$ over $X$,
we have $[L_{i,\chi}']=[T^{*}_{D_{i}}X]$ if $i\in I_{\mI,\chi}$.
Hence the implication (c) $\Rightarrow$ (a) holds.
If $i\in I_{\mII,\chi}$, then $R_{\chi}'=R_{\chi}$ and
the image of  $\cform(\chi)$ by the morphism $\Omega^{1}_{X}(R_{\chi}')\otimes_{\dvr_{X}}
\dvr_{D^{1/p}}\rightarrow 
\Omega^{1}_{X}(\log D)(R_{\chi})\otimes_{\dvr_{X}}\dvr_{D^{1/p}}$ is 
$\rsw(\chi)\neq 0$.
Since the kernel of the morphism $\Omega^{1}_{X}\otimes_{\dvr_{X}}\dvr_{D^{1/p}}\rightarrow 
\Omega^{1}_{X}(\log D)\otimes_{\dvr_{X}}\dvr_{D^{1/p}}$ is $\mathcal{C}_{D/X}\otimes_{\dvr_{X}}\dvr_{D^{1/p}}$,
the implication (a) $\Rightarrow$ (c) holds.

(iii) We may assume $D=D_{i}\cup D_{i'}$ and we put $D_{I}=D_{i}\cap D_{i'}$.
Since $\cform(\chi)|_{D_{I}}$ is the image of $\rsw(\chi)|_{D_{I}}$ by the morphism
$\Omega^{1}_{X}(\log D)(R_{\chi})|_{D_{I}}\rightarrow \Omega^{1}_{X}(R_{\chi}+D)|_{D_{I}}$
which is the $0$-mapping, the assertion holds.
\end{proof}

We note that $(X,U,\chi)$ is non-degenerate if and only if $(X,U,\chi)$ is strongly non-degenerate
for $\chi\in H^{1}_{\et}(U,\mathbf{Q}/\mathbf{Z})$ by Lemma \ref{lemint} (i).

We study the relation between $\rsw(\chi)$ and $\cform(\chi)$ at a closed point of $Z_{\chi}$.
For an element $\chi$ of $H^{1}_{\et}(U,\mathbf{Q}/\mathbf{Z})$
and a point $x$ of $D$,
we put $I_{\mI,\chi,x}=I_{x}\cap I_{\mI,\chi}$ and $I_{\mII,\chi,x}=I_{x}\cap I_{\mII,\chi}$. 
If $x$ is a closed point of $Z_{\chi}$ with a local coordinate system $(t_{1},\ldots, t_{d})$ in $X$
such that $t_{i}$ is a local equation of $D_{i}$ for $i\in I_{x}=\{1,\ldots,r\}$ and if we put 
\begin{align}
\label{rswchix}
\rsw(\chi)_{x}&=(\sum_{i=1}^{r}\alpha_{i}\dlog t_{i}+\sum_{i=r+1}^{d}\beta_{i}dt_{i})/
\prod_{i=1}^{r}t_{i}^{n_{i}}, \\
\cform(\chi)_{x'}&=(\sum_{i=1}^{r}\alpha_{i}'dt_{i}+\sum_{i=r+1}^{d}\beta_{i}'dt_{i})/
\prod_{i=1}^{r}t_{i}^{m_{i}}, \notag
\end{align}
where $R_{\chi}=\sum_{i=1}^{r}n_{i}D_{i}$, $R_{\chi}'=\sum_{i=1}^{r}m_{i}D_{i}$,
and $x'$ is the unique closed point of $Z_{\chi}^{1/p}$ lying above $x$,
then we have 
\begin{equation}
\label{relloc}
\begin{cases}
\alpha_{i}'=\alpha_{i}\prod_{i'\in I_{x}-(\{i\}\cup I_{\mII,\chi,x})}t_{i'} &
(i\in I_{\mI,\chi,x}), \\
\alpha_{i}\in t_{i}\cdot \dvr_{Z_{\chi},x}, \alpha_{i}'\in \prod_{i'\in I_{x}-I_{\mII,\chi,x}}t_{i'}
\cdot\dvr_{Z_{\chi}^{1/p},x'}
& (i\in I_{\mII,\chi,x}\cup I_{\mT,\chi,x}), \\ 
\beta_{i}'=\beta_{i}\prod_{i'\in I_{x}-I_{\mII,\chi,x}}t_{i'} &  (i=r+1,\ldots ,d).
\end{cases} 
\end{equation}

\begin{lem}
\label{lemordp}
Let $\chi$ be an element of $H^{1}_{\et}(U,\mathbf{Q}/\mathbf{Z})$ and
$x$ a closed point of $Z_{\chi}$.
Let $i$ be an element of $I_{\mI,\chi,x}$.
Let $n$ be the maximal integer $m$ such that
the image of $\xi_{i}(\chi)_{x}\colon \dvr_{X}(R_{\chi})|_{D_{i},x}\rightarrow \dvr_{D_{i},x}$ (\ref{xichi}) is
contained in $\mathfrak{m}_{x}^{m}\dvr_{D_{i},x}$, where $\mathfrak{m}_{x}$ is the maximal ideal at $x$.
Then we have $\ord'(\chi;x,D_{i})=n+\sharp(I_{\mI,\chi,x}\cup I_{\mT,\chi,x})-1$.
\end{lem}

\begin{proof}
Let the notation be as in (\ref{rswchix}).
Since $\chi|_{K_{i}}$ is of usual type, we may regard $\cform(\chi)|_{D_{i}^{1/p},x'}$ as an element of
$\Omega^{1}_{X}(R_{\chi}')|_{D_{i},x}$.
By (\ref{relloc}), we have $\cform(\chi)|_{D_{i}^{1/p},x'}=\alpha_{i}'\dlog t_{i}/\prod_{i'=1}^{r}t_{i'}^{m_{i'}}$.
Hence the order $\ord'(\chi;x,D_{i})$ is the maximal integer $m'$
such that $\alpha_{i}'\in \mathfrak{m}_{x}^{m'}\dvr_{D_{i},x}$.
Since $n$ is the maximal integer $m$ such that $\alpha_{i}\in \mathfrak{m}_{x}^{m}\dvr_{D_{i},x}$,
the assertion holds by (\ref{relloc}).
\end{proof}

\begin{lem}
\label{itsn}
Let $\chi$ be an element of $H^{1}_{\et}(U,\mathbf{Q}/\mathbf{Z})$
and $x$ a closed point of $Z_{\chi}$.
Assume that $(X,U,\chi)$ is clean at $x$
and let $d$ be the dimension of $\dvr_{X,x}$.
\begin{enumerate}
\item $\sharp(I_{x})<d$ or $I_{\mI,\chi,x}\neq \emptyset$.
\item $(X,U,\chi)$ is non-degenerate at $x$ if and only if one of the following conditions holds:
\begin{enumerate}
\item $I_{x}= I_{\mII,\chi,x}$.
\item $I_{x}=I_{\mW,\chi,x}$, $\sharp(I_{\mI,\chi,x})=1$, and 
the image of $\xi_{i}(\chi)$ in $k(x)$ is not $0$ for $i\in I_{\mI,\chi,x}$.
\end{enumerate} 
\item If $\sharp(I_{\mW,\chi,x})=d$ and if $\sharp(I_{\mI,\chi,x})=1$,
then $(X,U,\chi)$ is non-degenerate at $x$.  
\item If $\sharp(I_{x})=d$ and if $\sharp(I_{\mW,\chi,x})=1$,
then the image of $\xi_{i}(\chi)$ in $k(x)$ is not $0$ for $i\in I_{\mW,\chi,x}$.
\end{enumerate}
\end{lem}

\begin{proof}
(i) Suppose $\sharp(I_{x})=d$ and $I_{\mI,\chi,x}=\emptyset$.
Then, by (\ref{relloc}),
we have $\rsw(\chi)_{x}=0$ in $\Omega^{1}_{X}(\log D)(R_{\chi})_{x}\otimes_{\dvr_{X,x}}k(x)$,
which contradicts to the assumption that $(X,U,\chi)$ is clean at $x$.
Hence the assertion holds.

(ii) By Lemma \ref{lemint} (i) and (iii), 
we may assume $I_{x}=I_{\mW,\chi,x}$ and $\sharp(I_{\mI,\chi,x})\le 1$.

If $\sharp(I_{\mI,\chi,x})=0$, then $R_{\chi}'=R_{\chi}$ in a neighborhood of $x$ and the image of 
$\rsw(\chi)_{x}$ in $\Omega_{X}^{1}(\log D)(R_{\chi})_{x}\otimes_{\dvr_{X,x}}k(x)$ 
is the image of $\cform(\chi)_{x'}$ by
$\Omega^{1}_{X}(R'_{\chi})_{x}\otimes_{\dvr_{X,x}}k(x')\rightarrow \Omega^{1}_{X}(\log D)(R_{\chi})_{x}\otimes_{\dvr_{X,x}}k(x')$, where $x'$ is the unique point on $Z_{\chi}^{1/p}$ lying above $x$.
Since $(X,U,\chi)$ is clean at $x$, we have $\rsw(\chi)_{x}\neq 0$ in 
$\Omega_{X}^{1}(\log D)(R_{\chi})_{x}\otimes_{\dvr_{X,x}}k(x)$.
Hence we have $\cform(\chi)_{x'}\neq 0$
in $\Omega^{1}_{X}(R_{\chi})_{x}\otimes_{\dvr_{X,x}}k(x')$, which implies that
$(X,U,\chi)$ is non-degenerate at $x$.

Suppose $\sharp(I_{\mI,\chi,x})=1$ and let the notation be as in (\ref{rswchix}).
Then we have $\cform(\chi)|_{D_{i}^{1/p},x'}=\alpha_{i}dt_{i}/\prod_{i=1}^{r}t_{i}^{m_{i}}
\in \Omega_{X}^{1}(R_{\chi}')|_{D_{i},x}$ for $i\in I_{\mI,\chi,x}$ by (\ref{relloc}).
Hence $(X,U,\chi)$ is non-degenerate at $x$ if and only if the image of $\alpha_{i}$
for $i\in I_{\mI,\chi,x}$ in $k(x)$ is not $0$.
Since the last condition is equivalent to the condition (b), the assertion holds.

(iii) Let the notation be as in (\ref{rswchix}) and let $i$ be the element of $I_{\mI,\chi,x}$. 
By (\ref{relloc}), we have $\rsw(\chi)|_{D_{I_{x}},x}=\alpha_{i}\dlog t_{i}/\prod_{i=1}^{r}t_{i}^{n_{i}}$,
where $D_{I_{x}}=\bigcap_{i\in I_{x}}D_{i}$.
Hence the image of $\alpha_{i}$ in $k(x)$ is not $0$ by the cleanliness of $(X,U,\chi)$ at $x$.
Thus the assertion holds by (ii).

(iv) By (\ref{relloc}) and the cleanliness of $(X,U,\chi)$ at $x$,
the assertion holds.
\end{proof}

\begin{cor}
\label{corintdeg}
Let $\chi$ be an element of $H^{1}_{\et}(U,\mathbf{Q}/\mathbf{Z})$
and $x$ a closed point of $X$.
Assume that $\dvr_{X,x}$ is of dimension $2$ and
that $(X,U,\chi)$ is clean but not non-degenerate at $x$. 
Then $I_{\mW,\chi,x}=I_{\mI,\chi,x}$.
\end{cor}

\begin{proof}
Since $(X,U,\chi)$ is assumed to be not non-degenerate at $x$,
we have $x\in Z_{\chi}$.
Since $(X,U,\chi)$ is assumed to be clean at $x$ and 
$\dvr_{X,x}$ is assumed to be of dimension $2$, the assertion holds 
by Lemma \ref{itsn} (i) if $I_{x}\neq I_{\mW,\chi,x}$, 
by Lemma \ref{itsn} (ii) if $\sharp(I_{x})=\sharp(I_{\mW,\chi,x})=1$,
and by Lemma \ref{itsn} (i) and (iii) if $\sharp(I_{\mW,\chi,x})=2$.
\end{proof}
For the complement $V=X-E$ of a divisor $E$ on $X$  
contained in $D$ and the image $\chi\in H^{1}_{\et}(U,\mathbf{Q}/\mathbf{Z})$ of 
an element $\chi'\in H^{1}_{\et}(V,\mathbf{Q}/\mathbf{Z})$,
the refined Swan conductor $\rsw(\chi)$ and the characteristic form $\cform(\chi)$ of $\chi$ 
are the images of those $\rsw(\chi')$ and $\cform(\chi')$ of $\chi'$ by the canonical morphisms
$\Gamma(Z_{\chi'},\Omega^{1}_{X}(\log E)(R_{\chi'})|_{Z_{\chi'}})\rightarrow 
\Gamma(Z_{\chi},\Omega^{1}_{X}(\log D)(R_{\chi})|_{Z_{\chi}})$ and
$\Gamma(Z_{\chi'}^{1/p},\Omega^{1}_{X}(R_{\chi'}')\otimes_{\dvr_{X}}\dvr_{Z_{\chi'}^{1/p}})\rightarrow 
\Gamma(Z_{\chi}^{1/p},\Omega^{1}_{X}(R_{\chi}')\otimes_{\dvr_{X}}\dvr_{Z_{\chi}^{1/p}})$ respectively.
Here we note $R_{\chi}=R_{\chi'}$, $Z_{\chi}=Z_{\chi'}$, and $R_{\chi}'=R_{\chi'}'+D_{\mT,\chi}$.

\begin{lem}
\label{lemclndeg}
Let $E$ be a divisor on $X$ contained in $D$.
We put $V=X-E$.
Let $\chi'$ be an element of $H^{1}_{\et}(V,\mathbf{Q}/\mathbf{Z})$ and 
$\chi\in H^{1}_{\et}(U,\mathbf{Q}/\mathbf{Z})$ the image of $\chi'$.
Let $x$ be a point on $X$ and assume that $(X,U,\chi)$ is clean at $x$.
\begin{enumerate}
\item $(X,V,\chi')$ is clean at $x$.
Consequently, if $(X,U,\chi)$ is clean, then $(X,V,\chi')$ is clean.
\item Assume that $x$ is a closed point of $X$ and that $\sharp(I_{x})=\dim \dvr_{X,x}=2$.
Let $i$ be an element of $I_{x}$ and assume $E=D_{i}$.
Then $(X,V,\chi')$ is non-degenerate at $x$.
\end{enumerate}
\end{lem}

\begin{proof}
We note $R_{\chi}=R_{\chi'}$ and $Z_{\chi}=Z_{\chi'}$.
If $x\notin Z_{\chi'}$, then $(X,V,\chi')$ is clean and non-degenerate at $x$.
Hence we may assume $x\in Z_{\chi'}=Z_{\chi}$.
Since $(X,U,\chi)$ is assumed to be clean at $x$, we have $\rsw(\chi)_{x}\neq 0$ in $\Omega_{X}^{1}(\log D)(R_{\chi})_{x}
\otimes_{\dvr_{X,x}}k(x)$.

(i) Since $\rsw(\chi)_{x}$ is the image of $\rsw(\chi')_{x}$
by the canonical morphism
$\Omega^{1}_{X}(\log E)(R_{\chi'})|_{Z_{\chi'},x}\rightarrow
\Omega^{1}_{X}(\log D)(R_{\chi})|_{Z_{\chi},x}$,
we have $\rsw(\chi')_{x}\neq 0$ in $\Omega_{X}^{1}(\log E)(R_{\chi'})_{x}\otimes_{\dvr_{X,x}}k(x)$.
Hence the assertion holds.

(ii) By Lemma \ref{itsn} (ii), we may assume that $\chi'|_{K_{i}}$ is of type $\mI$.
Since $\rsw(\chi)_{x}$ is the image of $\rsw(\chi')_{x}$
by the canonical morphism
$\Omega^{1}_{X}(\log D_{i})(R_{\chi'})|_{Z_{\chi'},x}\rightarrow
\Omega^{1}_{X}(\log D)(R_{\chi})|_{Z_{\chi},x}$,
the image of $\xi_{i}(\chi')$ in $k(x)$ is equal to the image of $\xi_{i}(\chi)$ in $k(x)$.
Since the image of $\xi_{i}(\chi)$ in $k(x)$ is not $0$ by Lemma \ref{itsn} (iv),
the assertion holds by Lemma \ref{itsn} (ii).
\end{proof}

We give a sufficient condition for $(X,U,\chi)$ to be clean when $D=E\cup X'$ for
a smooth divisor $X'$ on $X$ in Lemma \ref{stablogtr} (ii) later.
\vspace{0.2cm}

In the rest of this subsection, we consider conditions on a Witt vector locally defining 
the $p$-part of $\chi\in H^{1}_{\et}(U,\mathbf{Q}/\mathbf{Z})$ for $(X,U,\chi)$ to be clean.
Let $\chi$ be an element of $H^{1}_{\et}(U,\mathbf{Q}/\mathbf{Z})$.
Let $x$ be a closed point of $Z_{\chi}$ and $A=\dvr_{X,x}$ the local ring at $x$.
We assume $I=I_{x}$.
Let $t_{i}\in A$ be a local equation of $D_{i}$ for $i\in I$.
We put $I=\{1,\ldots,r\}$ and $I_{\mW,\chi}=\{1,\ldots,r'\}=I_{\mW,\chi,x}$, where $1\le r'\le r\le \dim A$.
We put $n_{i}=\sw(\chi|_{K_{i}})$ and $s'_{i}=\ord_{p}(n_{i})$ for $i\in I_{\mW,\chi}$.
We assume $s'_{1}\le \cdots \le s'_{r'}$.
For an integer $s>0$ such that $s>s'_{1}$ and that the order of the $p$-part of $\chi$ is $\le p^{s}$,
we regard the $p$-part of $\chi$ as an element of $H^{1}_{\et}(U,\mathbf{Z}/p^{s}\mathbf{Z})$.

\begin{lem}
\label{stabord}
Let $i$ be an element of $I_{\mI,\chi}$.
Let $a=(a_{s-1}, \ldots ,a_{0})$ be an element of the stalk of $\fillog_{R_{\chi}}j_{*}W_{s}(\dvr_{U})$
at $x$
whose image in $H^{1}_{\et}(U,\mathbf{Z}/p^{s}\mathbf{Z})$ is locally the $p$-part of $\chi$
(then $s>s_{i}'$ by Lemma \ref{lemtone}).
We put $a_{s'_{i}}=a_{s'_{i}}'/\prod_{i'\in I_{\mW,\chi}}t_{i'}^{n_{i'}'}$, where
$a_{s'_{i}}'\in A$
and $n_{i'}'$ is the maximal integer $m'_{i'}$ such that $-m_{i'}'p^{s'_{i}}\ge -n_{i'}$ 
for $i'\in I_{\mW,\chi}$
(then $a_{s_{i}'}'\notin \mathfrak{p}_{i}$ by Lemma \ref{lemtone}).
Let $n$ be the maximal integer $m$ such that $a_{s'_{i}}'\in \mathfrak{m}^{m}_{x}A/\mathfrak{p}_{i}$.
Then we have $\ord'(\chi;x,D_{i})=np^{s'_{i}}+\sharp(I_{\mT,\chi})+\sum_{i'\in I_{\mW,\chi}\setminus\{i\}}(\dt(\chi|_{K_{i'}})-p^{s'_{i}}n_{i'}')$.
\end{lem}

\begin{proof}
By Lemma \ref{lemtone}, the germ $\cform(\chi)|_{D_{i}^{1/p},x'}$ depends only on $a_{s'_{i}}$,
where $x'$ is the unique closed point of $D_{i}^{1/p}$ lying above $x$.
Hence we may assume $a_{i'}=0$ if $i'\neq s'_{i}$.
Then we have
\begin{equation}
\rsw(\chi)|_{D_{i},x}=({a_{s'_{i}}'}^{p^{s'_{i}}-1}\prod_{i'\in I_{\mW,\chi}}t_{i'}^{n_{i'}-p^{s_{i}'}n_{i'}'})\cdot
\sum_{i''\in I_{\mW,\chi}}(n_{i''}'a_{s'_{i}}' \dlog t_{i''}-da'_{s'_{i}})/
\prod_{h\in I_{\mW,\chi}}t_{h}^{n_{h}}.
\notag
\end{equation}
Hence $n'p^{s_{i}'}+\sum_{i'\in I_{\mW,\chi}}n_{i'}-p^{s_{i}'}n_{i'}'$ is the maximal integer $m$ such that
the image of $\xi_{i}(\chi)_{x}\colon \dvr_{X}(R_{\chi})|_{D_{i},x}\rightarrow \dvr_{D_{i},x}$ is
contained in $\mathfrak{m}_{x}^{m}\dvr_{D_{i},x}$.
Since $\dt(\chi|_{K_{i'}})=n_{i'}+1$ for $i'\in I_{\mI,\chi}$, $\dt(\chi|_{K_{i'}})=n_{i'}$ for $i'\in I_{\mII,\chi}$,
and $n_{i}=p^{s_{i}'}n_{i}'$,
the assertion holds by Lemma \ref{lemordp}.
\end{proof}

\begin{lem}
\label{lemwitt}
Let $a=(a_{s-1}, \ldots ,a_{0})$ be an element of 
the stalk of $\fillog_{R_{\chi}}j_{*}W_{s}(\dvr_{U})$ at $x$
which locally defines the $p$-part of $\chi$.
We put $d=\dim A$ and $a_{h}=a_{h}'/\prod_{i'=1}^{r'}t_{i'}^{p^{-h}n_{i'}}$,
where $a_{h}'\in A$ for $0\le h\le s'_{1}$.
\begin{enumerate}
\item For $i\in I_{\mW,\chi}$, the following are equivalent:
\begin{enumerate}
\item $s'_{i}=s'_{1}$ and $a_{s'_{1}}'$ is invertible in $A$.
\item The image of $\xi_{i}(\chi)\colon \dvr_{X}(-R_{\chi})|_{D_{i}}\rightarrow \dvr_{D_{i}}$ (\ref{xichi}) in $k(x)$ is not $0$.
\end{enumerate}
\item If $a'_{s_{1}'}$ is invertible in $A$,
then $(X,U,\chi)$ is clean at $x$.
If $r=d$, then the converse is true.
\item Assume $r<d$ and $a_{h}'\in \mathfrak{m}_{x}$ for $0\le h\le s'_{1}-1$.
Then $(X,U,\chi)$ is clean at $x$ if and only if 
one of the following conditions holds:
\begin{enumerate}
\item $a_{s_{1}'}'$ is invertible in $A$.
\item $a_{s_{1}'}'\in \mathfrak{m}_{x}$ and $(t_{1},\ldots ,t_{r}, a_{0}')$ is a part of a coordinate system of $A$.
\end{enumerate}
\item Assume that $(X,U,\chi)$ is clean at $x$.
Then $(X,U,\chi)$ is non-degenerate at $x$ if and only if
one of the following conditions holds:
\begin{enumerate}
\item $I=I_{\mT,\chi}$.
\item $I=I_{\mII,\chi}$.
\item $I=I_{\mW,\chi}$ and $\sharp(I_{\mI,\chi})=1$. 
For $i\in I_{\mI,\chi}$, we have $s'_{i}=s'_{1}$ and $a_{s'_{1}}'$ is invertible in $A$.
\end{enumerate} 
\end{enumerate} 
\end{lem}

\begin{proof}
Since we have $a=(0,\ldots , 0,a_{s'_{1}},a_{s'_{1}-1},\ldots,a_{0})$ in $\grlog_{n_{1}}W_{s}(K_{1})$ by (\ref{lfilwskv}),
the germ $\rsw(\chi)|_{D_{1},x}$ depends only on $a_{h}$ for $0\le h\le s'_{1}$. 
Hence the image of $\rsw(\chi)_{x}$ in $\Omega^{1}_{X}(\log D)(R_{\chi})_{x}\otimes_{\dvr_{X,x}}k(x)$ 
depends only on $a_{h}$ for $0\le h\le s'_{1}$.
For $0\le h\le s'_{1}$, we have
\begin{equation}
\label{daj}
-a_{h}^{p^{h}-1}da_{h}={a_{h}'}^{p^{h}-1}(\sum_{i'\in I_{\mW,\chi,x}}
p^{-h}n_{i'}a_{h}'\dlog t_{i'}-da_{h}')/
\prod_{i'\in I_{\mW,\chi,x}}t_{i'}^{n_{i'}}.   
\end{equation}
Since $p^{-h}n_{i'}\in p\mathbf{Z}$ for $0\le h\le s'_{1}-1$ and $i'\in I_{\mW,\chi}$,
we have $p^{-h}n_{i'}a_{h}'^{p^{h}}=0$ for $0\le h\le s_{1}'-1$ and $i'\in I_{\mW,\chi}$.

(i) Since $p^{-h}n_{i}\in p\mathbf{Z}$ for $0\le h\le s'_{1}-1$, 
the condition (b) is equivalent to that $p^{-s'_{1}}n_{i}a_{s'_{1}}'^{p^{s'_{1}}}$ is invertible in $A$
by (\ref{daj}).
Hence the assertion holds.

(ii) If $a_{s_{1}'}'$ is invertible in $A$, then the image of $\xi_{1}(\chi)$ in $k(x)$ is not $0$ by (i).
Hence $\ord(\chi,x,D_{1})=0$. 
Thus the first assertion holds.

Conversely, suppose $r=d$ and that $(X,U,\chi)$ is clean at $x$.
Since $a_{h}'\in A$ for $0\le h\le s_{1}'$,
we have $da_{h}'=0$ in $\Omega^{1}_{X}(\log D)_{x}\otimes_{\dvr_{X,x}}k(x)$
for $0\le h\le s_{1}'$ by $r=d$.
Since $p^{-h}n_{i'}\in p\mathbf{Z}$ for $0\le h\le s'_{1}-1$ and $i'\in I_{\mW,\chi}$,
we have $a_{s_{1}'}'^{p^{s_{1}'}} \neq 0$ in $k(x)$
by the cleanliness of $(X,U,\chi)$ at $x$.
Hence the second assertion holds.

(iii) By (ii), it is sufficient to prove the equivalence between the condition (b) and
the cleanliness of $(X,U,\chi)$ at $x$ in the case  where $a_{s_{1}'}'\in \mathfrak{m}_{x}$.
Suppose $a_{s_{1}'}'\in \mathfrak{m}_{x}$.
By the assumption $a_{h}'\in \mathfrak{m}_{x}$ for $1\le h\le s'_{1}-1$ and 
$a_{s_{1}'}'\in \mathfrak{m}_{x}$, we have
$a_{h}'^{p^{h}-1}\in \mathfrak{m}_{x}$ for $1\le h\le s'_{1}$.
Hence, by (\ref{daj}), the cleanliness of $(X,U,\chi)$ at $x$ is equivalent to
$\sum_{i'\in I_{\mW,\chi}}n_{i'}a_{0}'\dlog t_{i'}-da_{0}' \neq 0$ in $\Omega^{1}_{X}(\log D)_{x}\otimes_{\dvr_{X,x}}k(x)$.
By the assumption $a_{0}'\in \mathfrak{m}_{x}$, the last condition is equivalent to 
$da_{0}'\neq 0$ in $\Omega_{X}^{1}(\log D)_{x}\otimes_{\dvr_{X,x}}k(x)$.
Since $da_{0}'$ in $\Omega_{X}^{1}(\log D)_{x}\otimes_{\dvr_{X,x}}k(x)$ is the image of $da_{0}'$ 
by the composition of the canonical morphism
$\Omega^{1}_{X,x}\otimes_{\dvr_{X,x}}k(x)\rightarrow \Omega_{D_{I},x}^{1}\otimes_{\dvr_{D_{I},x}}k(x)$
and the canonical injection
$\Omega_{D_{I},x}^{1}\otimes_{\dvr_{D_{I},x}}k(x)\rightarrow\Omega_{X}^{1}(\log D)_{x}\otimes_{\dvr_{X,x}}k(x)$,
the last condition is equivalent to $da_{0}'\neq 0$ in $\Omega^{1}_{D_{I},x}\otimes_{\dvr_{D_{I},x}}k(x)$,
which is equivalent to the condition (b) by the assumption $a_{0}'\in \mathfrak{m}_{x}$.

(iv) The assertion holds by (i) and Lemma \ref{itsn} (ii).
\end{proof}

\begin{cor}
\label{corwitt}
Let the notation be as in Lemma \ref{lemwitt}.
Assume $r=d$ and that $(X,U,\chi)$ is clean at $x$.
Let $i$ be an element of $I_{\mW,\chi,x}$.
Then the following are equivalent:
\begin{enumerate}
\item $s_{i}'=s_{1}'$.
\item The image of $\xi_{i}(\chi)$ in $k(x)$ is not $0$.
\end{enumerate}
\end{cor}

\begin{proof}
By Lemma \ref{lemwitt} (ii) and the assumption $r=d$ and that $(X,U,\chi)$ is clean at $x$,
the element $a_{s_{1}'}'$ is invertible in $A$.
Hence the assertion holds by Lemma \ref{lemwitt} (i).
\end{proof}

\subsection{Cleanliness and the direct image by the open immersion}
\label{sscdoi}
Let the notation be as in Subsection \ref{sscrc}.
We briefly recall Abbes-Saito's logarithmic ramification theory (\cite{as5}, \cite[2.2]{icm2}).
Let $R=\sum_{i\in I}n_{i}D_{i}$ be a linear combination with integral coefficients $n_{i}\ge 0$ for $i\in I$. 
Let $(X\ast_{k}X)_{i}$ denote the complement of the proper
transform of $(D_{i}\times_{k}X)\cup (X\times_{k}D_{i})$ in the blow-up of $X\times_{k}X$
along $D_{i}\times_{k}D_{i}$ for $i\in I$.
We define the $\log$-product $X\ast_{k}X$ to be the fiber product of 
$(X\ast_{k}X)_{i}$ for all $i\in I$ over $X\times_{k}X$.  
We note that if $X=\Spec A$ is affine and if $D_{i}=(t_{i}=0)$ for some $t_{i}\in A$ 
for $i\in I=\{1,\ldots ,r\}$
then we have $X*_{k}X=\Spec A\otimes_{k}A[((1\otimes t_{1})/(t_{1}\otimes 1))^{\pm 1},
\ldots, ((1\otimes t_{r})/(t_{r}\otimes 1))^{\pm 1}]$.
The diagonal $\delta\colon X\rightarrow X\times_{k}X$ is naturally lifted to a closed immersion
$\tilde{\delta}\colon X\rightarrow X\ast_{k}X$, which is called the {\it log diagonal}.
We consider the cartesian diagram
\begin{equation}
\label{logdiag}
\xymatrix{\tilde{X}\ar[r] \ar[d] & X\ast_{k}X \ar[d] \\
X \ar[r]_-{\delta} & X\times_{k}X.}
\end{equation}
Then, by \cite[Lemma 2.1]{tsu},
the fiber product $\tilde{X}$ in (\ref{logdiag}) is the union of 
$(\mathbf{G}_{m})^{\sharp(I')}$-bundles 
$\{D_{I'}\times_{k}
\mathbf{G}_{m}^{\sharp(I')}\}_{I' \subset I}$,
where $D_{I'}=\bigcap_{i\in I'}D_{i}$.
We note that if $X=\Spec A$ is affine and if $D_{i}=(t_{i}=0)$ for some $t_{i}\in A$ for $i\in I$
then the images $U_{i}$ of $(1\otimes t_{i})/(t_{i}\otimes 1)\in\Gamma(X*_{k}X,\dvr_{X*_{k}X})$
in $\Gamma(D_{I'}\times_{k}\mathbf{G}_{m}^{\sharp(I')},\dvr_{D_{I'}\times_{k}\mathbf{G}_{m}^{\sharp(I')}})$ for $i\in I'$ form a coordinate system of
$D_{I'}\times_{k}\mathbf{G}_{m}^{\sharp(I')}$.
We define $(X\ast_{k}X)^{(R)}$ to be the complement of the proper transform
of $(X\ast_{k}X)\times_{X}R$ in the blow-up of $X\ast_{k}X$ along the image
of $R$ by the log diagonal $\tilde{\delta}$. 
The open immersion $U\times_{k}U \rightarrow X\times_{k}X$ is naturally lifted to an open immersion
$\tilde{j}\colon U\times_{k}U\rightarrow X\ast_{k}X$ and further to an open immersion
$j^{(R)}\colon U\times_{k}U \rightarrow (X\ast_{k}X)^{(R)}$.
If $R=0$, then we have $(X\ast_{k}X)^{(R)}=X\ast_{k} X$.

Let $p_{i}\colon (X\ast_{k}X)^{(R)} \rightarrow X$ 
be the composition of the canonical morphism $(X*_{k}X)^{(R)}\rightarrow X\times_{k}X$
and the $i$-th projection $X\times_{k}X\rightarrow X$ for $i=1,2$.
By \cite[Lemma 4.6]{as5}, the morphism $p_{i}\colon (X\ast_{k}X)^{(R)} \rightarrow X$ is smooth for $i=1,2$.
Let $E^{(R)}$ denote the base change by $R\rightarrow X$ of $(X\ast_{k}X)^{(R)}$ regarded as
an $X$-scheme by $p_{i}\colon (X\ast_{k}X)^{(R)}\rightarrow X$, 
which is independent of the choice of $i$.
Then $E^{(R)}$ is a vector bundle over $R$ canonically isomorphic to $\mathbf{V}(\Omega_{X}^{1}
(\log D)\otimes_{\dvr_{X}}\dvr_{X}(R))\times_{X}R$ by \cite[8.13]{as5}.

Let $\mf$ be a smooth sheaf of $\Lambda$-modules of rank $1$ on $U$ 
and let $\chi\colon \pi_{1}^{\ab}(U)\rightarrow \Lambda^{\times}$ 
be the character corresponding to $\mf$.
We fix an inclusion $\Lambda^{\times}\rightarrow \mathbf{Q}/\mathbf{Z}$ and regard $\chi$ as an
element of $H^{1}_{\et}(U,\mathbf{Q}/\mathbf{Z})$.
Let $\mathcal{H}$ be the smooth sheaf 
$\mathcal{H}om(\mathrm{pr}_{2}^{\ast}\mf, \mathrm{pr}_{1}^{\ast}\mf)$
on $U\times_{k}U$, where $\mathrm{pr}_{i}\colon U\times_{k}U\rightarrow U$
is the $i$-th projection for $i=1,2$.
If $Z_{\chi}$ is not empty, then $j_{\ast}^{(R_{\chi})}\mathcal{H}$ is a smooth sheaf of $\Lambda$-modules of rank $1$ by \cite[Proposition 4.2.2.1]{as3} and
the restriction $j_{\ast}^{(R_{\chi})}\mathcal{H}|_{E^{(R_{\chi})}}$ of 
$j_{\ast}^{(R_{\chi})}\mathcal{H}$ to $E^{(R_{\chi})}$ is defined by the Artin-Schreier equation
$(F-1)(t)=-\rsw(\chi)$ by \cite[Proposition 4.2.2.2]{as3}. 

Let $I_{\mR,\chi}$ be the subset of the index set $I$ 
consisting of $i\in I$ such that $\chi|_{K_{i}}$ is ramified.
We put $I_{\mST,\chi}=I_{\mR,\chi}\cap I_{\mT,\chi}$.
We say that $\mf$ is \textit{totally wildly ramified} along $D$ if $I=I_{\mW,\chi}$.
We say that $\mf$ is \textit{strictly ramified} along $D$ if $I=I_{\mR,\chi}$ and
we say that $\mf$ is \textit{strictly tamely ramified} along $D$ if $I=I_{\mST,\chi}$.
We put $D_{\mR,\chi}=\bigcup_{i\in I_{\mR,\chi}}D_{i}$ and 
$D_{\mST,\chi}=\bigcup_{i\in I_{\mST,\chi}}D_{i}$.

\begin{lem}
\label{lemtame}
Assume that $\mf$ is tamely ramified along $D$.
\begin{enumerate}
\item $(X*_{k}X)^{(R_{\chi})}=X*_{k}X$, 
and $j^{(R_{\chi})}_{*}\mh=\tilde{j}_{*}\mh$ is a smooth sheaf
of $\Lambda$-modules of rank $1$.
\item Let $\bar{x}$ be a geometric point on $D_{\mST}$.
Let $x\in X$ be the image of $\bar{x}$ and put $I'=I_{\mST}\cap I_{x}$.
Then $(\tilde{j}_{*}\mh)|_{\bar{x}\times_{k}\mathbf{G}_{m}^{\sharp(I')}}$ is not constant.
\end{enumerate}
\end{lem}

\begin{proof}
Since $\mf$ is assumed to be tamely ramified along $D$, we have $R_{\chi}=0$.
Hence the first assertion in (i) holds.

Suppose that $\mf$ is unramified along $D$.
Then $\mf=\mf'|_{U}$ for a smooth sheaf $\mf'$ of $\Lambda$-modules of rank $1$ on $X$.
We put $\mh'=\shom(p_{2}^{*}\mf',p_{1}^{*}\mf')$.
Then $\mh'$ is a smooth sheaf of $\Lambda$-modules of rank $1$ on $X*_{k}X$ 
and we have $\mh=\tilde{j}^{*}\mh'$.
Since $\mh$ is a smooth sheaf on $U\times_{k}U\subset X*_{k}X$
and $X*_{k}X$ is smooth over $k$, the sheaf $\tilde{j}_{*}\mh$ is a smooth sheaf on $X*_{k}X$.
Hence the second assertion in (i) holds if $\mf$ is unramified along $D$.

Suppose that $\mf$ is not unramified along $D$.
Then we have $I_{\mST}\neq \emptyset$ and $D_{\mST}\neq 0$.
Since the assertions are \'{e}tale local, we may assume that $X=\Spec A$ is affine and
$D_{i}=(t_{i}=0)$ for some $t_{i}\in A$ for $i\in I=\{1,\ldots, r\}$.
Let $\bar{x}$ be a geometric point on $D_{\mST}$ 
and $x\in X$ the image of $\bar{x}$.
Since the assertions are \'{e}tale local,
we may assume $I=I_{x}$ and we
put $I'=I_{\mST}=\{1,\ldots, r'\}$, where $r'\le r$.
By taking an \'{e}tale covering of $X$ if necessary,
we may assume that $\chi$ is defined by $t^{n}=t_{1}^{n_{1}}\cdots t_{r'}^{n_{r'}}$
for an integer $n>1$ prime to $p$ and for
$(n_{i})_{i}\in (\mathbf{Z}_{> 0})^{r'}$ such that the greatest common divisor of $n$ and $(n_{i})_{i}$
is $1$. 
Since $X*_{k}X=\Spec A\otimes_{k}A[((1\otimes t_{1})/(t_{1}\otimes 1))^{\pm 1},
\ldots, ((1\otimes t_{r})/(t_{r}\otimes 1))^{\pm 1}]$ and
the smooth sheaf $\mh$ is defined by 
$t^{n}=((1\otimes t_{1})/(t_{1}\otimes 1))^{n_{1}}\cdots ((1\otimes t_{r'})/(t_{r'}\otimes 1))^{n_{r'}}$,
the second assertion in (i) holds.
Let $\iota\colon \bar{x}\times_{k}\mathbf{G}_{m}^{\sharp(I')}\rightarrow X*_{k}X$
be the morphism induced by $\tilde{X}\rightarrow X*_{k}X$.
Since the images $(U_{i})_{i\in I'}$ of $((1\otimes t_{i})/(t_{i}\otimes 1))_{i\in I'}$
by $\iota^{*}\colon \Gamma(X*_{k}X, \dvr_{X*_{k}X})\rightarrow \Gamma(X*_{k}X,\iota_{*}\dvr_{\bar{x}\times_{k}\mathbf{G}_{m}^{\sharp(I_{x})}})$
form a coordinate system of $\bar{x}\times_{k}\mathbf{G}_{m}^{\sharp(I')}$,
the assertion (ii) holds.
\end{proof}

\begin{lem}
\label{nonconsdr}
Let $\bar{x}$ be a geometric point on $D_{\mR}$ and
let $P^{(R_{\chi})}_{\bar{x}}$ be the fiber product of $(X\ast_{k}X)^{(R_{\chi})}$ and $\bar{x}$
over $X\times_{k}X$, where $\bar{x}$ is regarded as an $X\times_{k}X$-scheme
by the diagonal $X\rightarrow X\times_{k}X$.
Let $x\in X$ be the image of $\bar{x}$.
\begin{enumerate}
\item If $x\notin Z_{\chi}$, 
then $j^{(R_{\chi})}_{\ast}\mathcal{H}|_{P^{(R_{\chi})}_{\bar{x}}}$ is not constant.
\item If $x\in Z_{\chi}$ and if $(X,U,\chi)$ is clean at $x$, 
then $j^{(R_{\chi})}_{\ast}\mathcal{H}|_{P^{(R_{\chi})}_{\bar{x}}}$ is not constant.
\end{enumerate}
\end{lem}

\begin{proof}
(i) Since $\mf$ is tamely ramified along $D$ in a neighborhood of $x\notin Z_{\chi}$
and the assertion is \'{e}tale local,
we may assume $(X*_{k}X)^{(R_{\chi})}=X*_{k}X$ and $j^{(R_{\chi})}=\tilde{j}$ 
by Lemma \ref{lemtame} (i).
We put $I'=I_{\mST}\cap I_{x}$ and
we consider the commutative diagram
\begin{equation}
\xymatrix{
D_{I'}\times_{k}\mathbf{G}_{m}^{\sharp(I')} \ar[r] \ar[d] & X\ast_{k}X \ar[d] \\
D \ar[r]_{\delta} &X\times_{k}X.
}
\notag
\end{equation}
Since $P_{\bar{x}}^{(R_{\chi})}$ contains $\bar{x}\times_{k}\mathbf{G}_{m}^{\sharp(I')}$,
the assertion holds by Lemma \ref{lemtame} (ii).

(ii) For $i=1,2$, we consider the cartesian diagram
\begin{equation}
\xymatrix{
E^{(R_{\chi})} \ar[d] \ar[r] & (X\ast_{k}X)^{(R_{\chi})} \ar[d]^{p_{i}} \\
R_{\chi} \ar[r] & X.
} \notag
\end{equation}
Since $E^{(R_{\chi})}$ is independent of the choice of $i$,
the fiber product $P_{\bar{x}}^{(R_{\chi})}$ is 
$E_{\bar{x}}^{(R_{\chi})}=\bar{x}\times_{R_{\chi}}E^{(R_{\chi})}$.
Since the restriction $j^{(R_{\chi})}_{\ast}\mathcal{H}|_{E_{x}^{(R_{\chi})}}$ is defined 
by the Artin-Schreier equation $(F-1)(t)=-\rsw(\chi)$ by \cite[Proposition 4.2.2.2]{as3},
the cleanliness of $(X,U,\chi)$ implies that the restriction 
$j_{\ast}^{(R_{\chi})}\mathcal{H}|_{E^{(R_{\chi})}_{\bar{x}}}$ is not constant.
Hence the assertion holds.
\end{proof}

\begin{prop}
\label{propcliso}
Assume that $(X,U,\chi)$ is clean.
Then we have $(Rj_{\ast}\mf)|_{D_{\mR}}=0$.
Especially, if $\mf$ is strictly ramified along $D$,
then the canonical morphism $j_{!}\mf \rightarrow Rj_{\ast}\mf$ is an isomorphism.
\end{prop}

\begin{proof}
We consider the cartesian diagram
\begin{equation}
\xymatrix{
U\times_{k}U \ar[d]_-{\mathrm{pr}_{i}} \ar[r]^-{j^{(R_{\chi})}} & (X\ast_{k}X)^{(R_{\chi})} \ar[d]^-{p_{i}}\\
U \ar[r]_-{j} & X
}
\notag
\end{equation}  
for $i=1,2$ (\cite[2.3]{icm2}).
Since $\mf$ is of rank $1$, the evaluation morphism
\begin{equation}
\label{evmorph}
\mathcal{H}\otimes_{\Lambda}\mathrm{pr}_{2}^{\ast}\mf\rightarrow \mathrm{pr}_{1}^{\ast}\mf
\end{equation}
is an isomorphism.
By applying $Rj^{(R_{\chi})}_{\ast}$ to (\ref{evmorph}),
we have an isomorphism 
\begin{equation}
\label{rjiso}
Rj^{(R_{\chi})}_{\ast}(\mathcal{H}\otimes_{\Lambda}\mathrm{pr}_{2}^{\ast}\mf) 
\rightarrow Rj^{(R_{\chi})}_{\ast}\mathrm{pr}_{1}^{\ast}\mf. 
\end{equation}
We consider the canonical morphism 
\begin{equation}
\label{jrhf}
j^{(R_{\chi})}_{\ast}\mathcal{H}\otimes_{\Lambda}Rj^{(R_{\chi})}_{\ast}\mathrm{pr}_{2}^{\ast}\mf
\rightarrow Rj^{(R_{\chi})}_{\ast}(\mathcal{H}\otimes_{\Lambda} \mathrm{pr}_{2}^{\ast}\mf). 
\end{equation}
Since $j^{(R_{\chi})}_{\ast}\mathcal{H}$ and $\mh$ are locally constant
by Lemma \ref{lemtame} (i) and \cite[Proposition 4.2.2.1]{as3},
the canonical morphism (\ref{jrhf}) is an isomorphism.
Hence the composition 
$j^{(R_{\chi})}_{\ast}\mathcal{H}\otimes_{\Lambda}Rj^{(R_{\chi})}_{\ast}\mathrm{pr}_{2}^{\ast}\mf
\rightarrow Rj^{(R_{\chi})}_{\ast}\mathrm{pr}_{1}^{\ast}\mf$ of (\ref{jrhf}) and the isomorphism
(\ref{rjiso}) is an isomorphism.
By the smooth base change theorem, the base change morphism 
$p_{i}^{*}Rj_{*}\mf \rightarrow Rj_{*}^{(R_{\chi})}\pr_{i}^{*}\mf$ is an isomorphism for $i=1,2$.
Hence we have an isomorphism
\begin{equation}
\label{jhptisopo}
j^{(R_{\chi})}_{\ast}\mathcal{H}\otimes_{\Lambda}p_{2}^{\ast}Rj_{\ast}\mf
\rightarrow p_{1}^{\ast}Rj_{\ast}\mf.
\end{equation}

Let $\bar{x}$ be a geometric point on $D_{\mR}$.
Let $P_{\bar{x}}^{(R_{\chi})}$ be as in Lemma \ref{nonconsdr} and 
$\pi_{\bar{x}}\colon P_{\bar{x}}^{(R_{\chi})}\rightarrow \bar{x}$ be the projection.
By restricting (\ref{jhptisopo}) to $P^{(R_{\chi})}_{\bar{x}}$,
we have an isomorphism 
\begin{equation}
\label{mappx}
(j^{(R_{\chi})}_{\ast}\mathcal{H})|_{P^{(R_{\chi})}_{\bar{x}}}\otimes_{\Lambda}(p_{2}^{\ast}Rj_{\ast}\mf)|_{P^{(R_{\chi})}_{\bar{x}}}
\rightarrow (p_{1}^{\ast}Rj_{\ast}\mf)|_{P^{(R_{\chi})}_{\bar{x}}}.
\end{equation}
Since we have
$(p_{i}^{\ast}Rj_{\ast}\mf)|_{P^{(R_{\chi})}_{\bar{x}}}= \pi_{\bar{x}}^{\ast}(Rj_{\ast}\mf)_{\bar{x}}$ for $i=1,2$,
the complex $(p_{1}^{\ast}Rj_{\ast}\mf)|_{P^{(R_{\chi})}_{\bar{x}}}=
(p_{2}^{\ast}Rj_{\ast}\mf)|_{P^{(R_{\chi})}_{\bar{x}}}$ 
has constant cohomology sheaves.
Since $(j_{*}^{(R_{\chi})}\mh)|_{P_{\bar{x}}^{(R_{\chi})}}$ is not constant by
Lemma \ref{nonconsdr} and the morphism (\ref{mappx}) is an isomorphism, 
we have $(p_{i}^{\ast}Rj_{\ast}\mf)|_{P^{(R_{\chi})}_{\bar{x}}}=0$ for $i=1,2$.
Since $(p_{1}^{\ast}Rj_{\ast}\mf)|_{P^{(R_{\chi})}_{\bar{x}}}=\pi_{\bar{x}}^{*}(Rj_{*}\mf)_{\bar{x}}$
for $i=1,2$, we have $\pi_{\bar{x}}^{*}(Rj_{*}\mf)_{\bar{x}}=0$ and hence
$(Rj_{*}\mf)_{\bar{x}}=0$.
Thus the first assertion holds.
Since $D_{\mR}=D$ in the case where $\mf$ is strictly ramified along $D$,
the last assertion holds.
\end{proof}

\section{Logarithmic theory of characteristic cycle}
\label{slc}
\subsection{Logarithmic characteristic cycle of a rank $1$ sheaf}
\label{ssklcc}
We briefly recall Kato's logarithmic characteristic cycle defined in \cite{ka2}.

Let $X$ be a smooth scheme purely of dimension $d$ over a perfect field $k$ of characteristic $p>0$.
Let $D$ be a divisor on $X$ with simple normal crossings and $\{D_{i}\}_{i\in I}$ the irreducible components of $D$.
Let $\dvr_{K_{i}}=\hat{\dvr}_{X,\mathfrak{p}_{i}}$ be the completion of the local ring at 
the generic point $\mathfrak{p}_{i}$ of $D_{i}$
and $K_{i}=\Frac \dvr_{K_{i}}$ the local field at $\mathfrak{p}_{i}$ for $i\in I$.
We put $U=X-D$ and let $j\colon U\rightarrow X$ be the open immersion.
Let $\mf$ be a smooth sheaf of $\Lambda$-modules of rank $1$ on $U$ 
and let $\chi\colon \pi_{1}^{\ab}(U)\rightarrow \Lambda^{\times}$ 
be the character corresponding to $\mf$.
We fix an inclusion $\psi\colon \Lambda^{\times}\rightarrow \mathbf{Q}/\mathbf{Z}$ 
and regard $\chi$ as an element of $H^{1}_{\et}(U,\mathbf{Q}/\mathbf{Z})$ by $\psi$.

Let $T^{*}X(\log D)=\Spec S^{\bullet}\Omega^{1}_{X}(\log D)^{\vee}$ be the logarithmic cotangent
bundle with logarithmic poles along $D$ and
$T^{*}_{X}X(\log D)$ the zero section of $T^{*}X(\log D)$.
Let $Z_{d}(T^{\ast}X(\log D))$ be the free abelian group of $d$-cycles on $T^{\ast}X(\log D)$.
For $i\in I_{\mW,\chi}$,
let $L_{i,\chi}$ be the sub line bundle of $T^{*}X(\log D)\times_{X}D_{i}$ defined by the unique
$\dvr_{D_{i}}$-submodule of $\Omega^{1}_{X}(\log D)|_{D_{i}}$ which is 
locally a direct summand of $\Omega_{X}^{1}(\log D)|_{D_{i}}$ of rank $1$
containing $\dvr_{X}(-R_{\chi})|_{D_{i}}\cdot \rsw(\chi)|_{D_{i}}$.
We note that if $(X,U,\chi)$ is clean then $L_{i,\chi}$ is defined by 
$\dvr_{X}(-R_{\chi})|_{D_{i}}\cdot \rsw(\chi)|_{D_{i}}$.

\begin{defn}[{\cite[(3.4.4)]{ka2}}]
\label{deflogcc}
Assume that $(X,U,\chi)$ is clean. 
We define a \textit{logarithmic characteristic cycle} $\Char^{\log}(X,U,\chi)\in Z_{d}(T^{\ast}X(\log D))$ 
of $(X,U,\chi)$ by
\begin{equation}
\label{charlogform}
\Char^{\log}(X,U,\chi)=(-1)^{d}([\zers(\log D)]+\sum_{i\in I_{\mW,\chi}}\sw(\chi|_{K_{i}})[L_{i,\chi}]). 
\end{equation}
We write $SS^{\log}(X,U,\chi)$ for the support of $\Char^{\log}(X,U,\chi)$.
\end{defn}

\begin{thm}[Index formula, {\cite[Corollary 3.8]{saj}}]
\label{logcleanindexformula}
Assume that $(X,U,\chi)$ is clean.
If $X$ is proper over an algebraically closed field, then we have
\begin{equation}
\chi(X,j_{!}\mf)=(\Char^{\log}(X,U,\chi),\zers(\log D))_{T^{\ast}X(\log D)}, \notag
\end{equation}
where the right-hand side is the intersection number of $\Char^{\log}(X,U,\chi)$ with the zero section
$T^{*}_{X}X(\log D)$ of $T^{*}X(\log D)$.
\end{thm}

If $d=2$, then, without any assumption on the cleanliness of $(X,U,\chi)$, a logarithmic characteristic cycle 
$\Char^{\log}(X,U,\chi)$ is defined as follows: 
Let 
\begin{equation}
\label{seqcl}
f: X^{\prime}=X_{s}\rightarrow X_{s-1}\rightarrow \cdots \rightarrow X_{0}=X 
\end{equation}
be the composition of blow-ups $\{X_{i+1}\rightarrow X_{i}\}_{0\le i\le s-1}$ at closed points 
$\{x_{i}\in X_{i}\}_{0\le i\le s-1}$ lying above $D$ such that 
$(X^{\prime},f^{-1}(U),f^{\ast}\chi)$ is clean (\cite[Theorem 4.1]{ka2}).
We consider the algebraic correspondence
\begin{equation}
T^{*}X(\log D)\xleftarrow{\pr_{1}} T^{*}X(\log D)\times_{X}X' \xrightarrow{df^{D}} T^{*}X'(\log f^{*}D).  \notag
\end{equation}
We put $C'=SS^{\log}(X',f^{*}U,f^{*}\chi)\subset T^{*}X'(\log f^{*}D)$
and $C=\pr_{1}({df^{D}}^{-1}(C'))\subset T^{*}X(\log D)$.
Then $df^{D}$ and $\pr_{1}$ define the Gysin homomorphism 
(\cite[6.6]{ful}) and the proper push-forward 
\begin{equation}
\label{chcomp}
CH_{2}(C')\xrightarrow{{df^{D}}^{!}} 
CH_{2}({df^{D}}^{-1}(C'))
\xrightarrow{\pr_{1*}} CH_{2}(C). 
\end{equation}
Since $f$ is a composition of blow-ups at closed points, every irreducible components
of $C$ is of dimension $\le 2$.
Hence the composition (\ref{chcomp}) defines a morphism
$Z_{2}(C')\rightarrow Z_{2}(C)$ of free abelian groups.
We define the logarithmic characteristic cycle $\Char^{\log}(X,U,\chi)$ to be the image of
$\Char^{\log}(X',f^{*}U,f^{*}\chi)$ by the morphism $Z_{2}(C')\rightarrow Z_{2}(C)$.
Then the logarithmic characteristic cycle $\Char^{\log}(X,U,\chi)$ is of the form
\begin{equation}
\label{logcharcycle}
\Char^{\log}(X,U,\chi)=[\zers(\log D)]+\sum_{i\in I}\sw(\chi|_{K_{i}})[L_{i,\chi}]+
\sum_{x\in |D|}s_{x}[T^{\ast}_{x}X(\log D)],
\end{equation}
where $|D|$ is the set of closed points of $D$ and $T^{\ast}_{x}X(\log D)\subset T^{\ast}X(\log D)$ is the fiber at $x\in |D|$.
We write $SS^{\log}(X,U,\chi)$ for the support of $\Char^{\log}(X,U,\chi)$.

We note that this definition is the same as that in \cite[Remark 5.8]{ka2}
and that the definition of $\Char^{\log}(X,U,\chi)$ is independent of the choice of successive blow-ups
by \cite[Remark 5.7]{ka2}.
In particular, if $(X,U,\chi)$ is clean, then the multiplicities of fibers
at closed points on $D$ in $\Char^{\log}(X,U,\chi)$ are $0$.
We refer to \cite[Remark 5.7]{ka2} for
an algebraic computation of the multiplicities of fibers at closed points of $D$ in $\Char^{\log}(X,U,\chi)$.
We note that the triple $(X,U,\chi)$ is not always clean even if the multiplicities of fibers at closed points on $D$ are $0$.

\begin{exa}
Let $X$ be the affine plane $\atk =\Spec k[t_{1},t_{2}]$ and $D$
the irreducible divisor $(t_{1}=0)$.
Let $x\in X$ be the origin.
Let $\mf$ be a smooth sheaf of $\Lambda$-modules of rank $1$ on $U=X-D$ defined by the 
Artin-Schreier equation $t^{p}-t=(t_{2}/t_{1})^{n}$ where $n\in \mathbf{Z}_{\ge 2}$ is prime to $p$.
Then $(X,U,\chi)$ is not clean at $x$.
However we have $s_{x}=0$ by calculating a few blow-ups at non-clean points lying over $x$
following \cite[Remark 5.7]{ka2}. 
\end{exa}

By Theorem \ref{logcleanindexformula},
we obtain the following index formula.

\begin{thm}[Index formula]
\label{logindexformula}
Assume $d=2$.
If $X$ is proper over an algebraically closed field, then we have
\begin{equation}
\chi(X,j_{!}\mf)=(\Char^{\log}(X,U,\chi),\zers(\log D))_{T^{\ast}X(\log D)}. \notag
\end{equation}
\end{thm}

\subsection{Log micro support of a rank $1$ sheaf}
\label{sslogms}
For divisors $E\subset D$ on $X$ with simple normal crossings,
let $\tau_{E/D}\colon T^{*}X(\log E)\rightarrow T^{*}X(\log D)$ be the canonical morphism
of vector bundles over $X$.
For an $X$-scheme $W$, we write $\tau_{E/D, W}\colon T^{*}X(\log E)\times_{X}W\rightarrow
T^{*}X(\log D)\times_{X}W$ for the base change of $\tau_{E/D}$ by the structure morphism $W\rightarrow X$.
In particular, if $E=\emptyset$, then 
we simply write $\tau_{D}$ and $\tau_{D,W}$ for $\tau_{\emptyset/D}$ and 
$\tau_{\emptyset/D, W}$ respectively.

We keep the notation in Subsection \ref{ssklcc}.
In this subsection, we prove that $j_{!}\mf$ is micro-supported on 
$\tau_{D}^{-1}(SS^{\log}(X, U, \chi))$ in the case where $(X, U, \chi)$ is clean 
(Theorem \ref{proplogct}).

We introduce the log transversality of a morphism of smooth schemes over $k$.
The log transversality implies the cleanliness of the pull-back of a clean character 
(Proposition \ref{cleaneq} (iii)). 
Let $C_{D}=\bigcup_{I'\subset I}T^{\ast}_{D_{I'}}X\subset T^{\ast}X$
be the union of the conormal bundles of
$D_{I'}=\bigcap_{i\in I'}D_{i}$ over $X$ for all subsets $I'\subset I$.
We note that $C_{D}$ is a closed conical subset of $T^{*}X$.
By \cite[Lemma 3.4.5]{sa4}, 
a morphism $h\colon W\rightarrow X$ of smooth schemes over $k$ is 
$C_{D}$-transversal if and only if
$h^{\ast}D$ is a divisor on $W$ with simple normal crossings 
and $h^{\ast}D_{i}$ is a smooth divisor on $W$ for every $i\in I$.
For a $C_{D}$-transversal morphism $h\colon W\rightarrow X$ of smooth schemes over $k$,
let
\begin{equation}
dh^{D}\colon T^{\ast}X(\log D)\times_{X}W\rightarrow T^{\ast}W(\log h^{\ast}D) \notag
\end{equation}
be the morphism yielded by $h$.

\begin{defn}
\label{deflogtr}
Let $C$ be a closed conical subset of $T^{\ast}X(\log D)$.
\begin{enumerate}
\item We say that a morphism $h\colon W\rightarrow X$ of smooth schemes over $k$
is {\it $\log$-$D$-$C$-transversal} at a point $w\in W$ if $h$ is $C_{D}$-transversal at $w$ and 
if the subset $({dh^{D}}^{ -1}(T^{\ast}_{W}W(\log h^{\ast}D))\cap h^{*}C)\times_{W}w\subset 
T^{\ast}X(\log D)\times_{X}w$ is contained in the zero section $T^{\ast}_{X}X(\log D)\times_{X}w$.

We say that a morphism $h\colon W\rightarrow X$ of smooth schemes over $k$
is {\it $\log$-$D$-$C$-transversal} if $h$ is $C_{D}$-transversal and 
if the subset ${dh^{D}}^{ -1}(T^{\ast}_{W}W(\log h^{\ast}D))\cap h^{*}C\subset T^{\ast}X(\log D)\times_{X}W$
is contained in the zero section $T^{\ast}_{X}X(\log D)\times_{X}W$.

\item Let $h\colon W\rightarrow X$ be a $\log$-$D$-$C$-transversal morphism 
of smooth schemes over $k$. 
We define a closed conical subset
$h^{\circ}C\subset T^{\ast}W(\log h^{\ast}D)$
to be the image of the subset $h^{\ast}C\subset T^{*}X(\log D)\times_{X}W$ by 
$dh^{D}\colon T^{*}X(\log D)\times_{X}W\rightarrow T^{*}W(\log h^{*}D)$.
\end{enumerate}
\end{defn}

We note that the $\log$-$D$-$C$-transversality of $h$ is equivalent to the 
$\log$-$D$-$C$-transversality at all points on $W$.
We also note that if the morphism $h\colon W\rightarrow X$ is 
$\log$-$D$-$C$-transversal, then
the restriction $dh^{D}\colon h^{*}C\rightarrow T^{*}W(\log h^{*}D)$ is finite
by \cite[Lemma 3.1]{sa4}.

\begin{lem}
\label{pullzers}
The inverse image $\tau_{D}^{-1}(T^{*}_{X}X(\log D))$ is $C_{D}$
and we have $\tau_{D}^{!}([T^{*}_{X}X(\log D)])=\sum_{I'\subset I}[T^{*}_{D_{I'}}X]$,
where $\tau_{D}^{!}\colon CH_{d}(T^{*}_{X}X(\log D))\rightarrow CH_{d}(C_{D})$
is the Gysin homomorphism for the l.c.i.\ morphism $\tau_{D}\colon T^{*}X\rightarrow T^{*}X(\log D)$ (\cite[6.6]{ful}).
\end{lem}

\begin{proof}
Let $x$ be a closed point of $D$.
Since the assertion is local, it is sufficient to prove the assertion in a neighborhood of $x$.
We may assume $x\in D_{I}$.
Let $\mathcal{I}$ and $\mathcal{J}$ be the defining ideal sheaves of 
$\zers(\log D)\subset T^{\ast}X(\log D)$ and
$\tau_{D}^{-1}(T^{*}_{X}X(\log D))\subset T^{*}X$ respectively.
We put $I=\{1,\ldots,r\}$ and let $(t_{1}, \ldots ,t_{d})$ be a local coordinate system at $x$ such that
$t_{i}$ is a local equation of $D_{i}$ for $i\in I$.
Then $\Omega^{1}_{X}(\log D)_{x}$ and 
$\Omega^{1}_{X,x}$ 
are free $\dvr_{X,x}$-modules 
and have base $(\dlog t_{1} ,\ldots ,\dlog t_{r}, dt_{r+1}, \ldots ,dt_{d})$ 
and $(dt_{1}, \ldots, dt_{d})$ respectively.
Let the dual base of the base be denoted by
$(\partial/(\partial \log t_{1}), \ldots ,\partial/(\partial \log t_{r}), \partial/\partial t_{r+1}, 
\ldots ,\partial/\partial t_{d})$ 
and $(\partial/\partial t_{1}, 
\ldots ,\partial/\partial t_{d})$ respectively.
Since we have
\begin{equation}
\mathcal{I}_{x}=(\partial/(\partial \log t_{1}), \ldots ,\partial/(\partial \log t_{r}), \partial/\partial t_{r+1}, \ldots ,\partial/\partial t_{d}),
\notag
\end{equation}
we have
\begin{equation}
\mathcal{J}_{x}=
(t_{1}\partial/\partial t_{1},
\ldots ,t_{r}\partial/\partial t_{r}, \partial/\partial t_{r+1}, \ldots ,\partial/\partial t_{d}). \notag
\end{equation}
Hence the assertions hold.
\end{proof}

\begin{lem}
\label{lemcart}
\begin{enumerate}
\item Let $A$ be a commutative ring.
Let $M$, $M'$, $N$, and $N'$ be finitely generated projective $A$-modules.
If a commutative diagram
\begin{equation}
\label{seqfgpm}
\xymatrix{N' & N \ar[l] \\ M' \ar[u] & M \ar[l] \ar[u]}
\end{equation}
of $A$-modules is cocartesian, then the commutative diagram
\begin{equation}
\xymatrix{
S^{\bullet}N' & S^{\bullet}N \ar[l] \\
S^{\bullet}M' \ar[u] & S^{\bullet}M \ar[u] \ar[l]}
\notag
\end{equation}
of $A$-algebras is cocartesian.
\item Let $h\colon W\rightarrow X$ be a $C_{D}$-transversal morphism
of smooth schemes over $k$.
Then the commutative diagram
\begin{equation}
\label{cartdh}
\xymatrix{
T^{*}X\times_{X}W \ar[d]_-{dh} \ar[r]^-{\tau_{D,W}} & T^{*}X(\log D)\times_{X}W \ar[d]^-{dh^{D}} \\
T^{*}W \ar[r]_-{\tau_{h^{*}D}} & T^{*}W(\log h^{*}D).
} 
\end{equation}
is cartesian.
\end{enumerate}
\end{lem}

\begin{proof}
(i) We prove that the commutative diagram
\begin{equation}
\xymatrix{
\Hom_{A{\text -alg}}(S^{\bullet}N', B)\ar[d] \ar[r] & \Hom_{A{\text -alg}}(S^{\bullet}N, B) \ar[d] \\
\Hom_{A{\text -alg}}(S^{\bullet}M', B) \ar[r] & \Hom_{A{\text -alg}}(S^{\bullet}M,B)
} \notag
\end{equation}
is cartesian for every $A$-algebra $B$,
which deduces the assertion.
By the universality of symmetric algebra,
it is sufficient to prove that the commutative diagram
\begin{equation}
\xymatrix{
\Hom_{A{\text -mod}}(N', B)\ar[d] \ar[r] & \Hom_{A{\text -mod}}(N, B) \ar[d] \\
\Hom_{A{\text -mod}}(M', B) \ar[r] & \Hom_{A{\text -mod}}(M,B)
} \notag
\end{equation}
is cartesian for every $A$-algebra $B$ regarded as an $A$-module
and this holds by the assumption that the diagram (\ref{seqfgpm}) is cocartesian.

(ii) We consider the commutative diagram
\begin{equation}
\xymatrix{
0\ar[r] & h^{*}\Omega_{X}^{1} \ar[r] \ar[d]_-{dh} & h^{*}\Omega_{X}^{1}(\log D) \ar[r] \ar[d]^-{dh^{D}} 
& \bigoplus_{i\in I}h^{*}\dvr_{D_{i}} \ar[r] \ar[d] & 0 \\
0 \ar[r] & \Omega^{1}_{W} \ar[r] & \Omega^{1}_{W}(\log h^{*}D) \ar[r] & 
\bigoplus_{i\in I}\dvr_{h^{*}D_{i}} \ar[r] & 0.} \notag
\end{equation}
Since $h$ is assumed to be $C_{D}$-transversal, the horizontal lines are exact and
the right vertical arrow is an isomorphism by \cite[Lemma 3.4.5]{sa4}.
Hence the left square is cartesian.
Since the dual of the left square is cocartesian,
the assertion holds by (i).
\end{proof}

\begin{prop}
\label{logtrctr}
Let $h\colon W\rightarrow X$ be a morphism of smooth schemes over $k$ and
let $C\subset T^{\ast}X(\log D)$ be a closed conical subset.
Then the following are equivalent:
\begin{enumerate}
\item $h$ is $\log$-$D$-$C$-transversal. 
\item $h$ is $C_{D}\cup \tau_{D}^{-1}(C)$-transversal.
\end{enumerate}
\end{prop}

\begin{proof}
We may assume that $h$ is $C_{D}$-transversal and it suffices to prove that
$h$ is $\log$-$D$-$C$-transversal if and only if $h$ is $\tau_{D}^{-1}(C)$-transversal.
Since the image of $T^{*}_{W}W$ by $\tau_{h^{*}D}$ is $T^{*}_{W}W(\log D)$,
the commutative diagram
\begin{equation}
\xymatrix{
dh^{-1}(T^{*}_{W}W)\ar[r]^-{\tau_{D,W}} \ar[d]_-{dh} & dh^{D-1}(T^{*}_{W}W(\log h^{*}D))
\ar[d]^-{dh^{D}} \\
T^{*}_{W}W \ar[r]_-{\tau_{h^{*}D}} & T^{*}_{W}W(\log h^{*}D)
} \notag
\end{equation}
is cartesian by Lemma \ref{lemcart} (ii).
Since the lower horizontal arrow is an isomorphism, the upper horizontal arrow is an isomorphism.
We consider the isomorphism
\begin{equation}
dh^{-1}(T^{*}_{W}W)\cap h^{*}\tau_{D}^{-1}(C) \rightarrow 
{dh^{D}}^{-1}(T^{*}_{W}W(\log h^{*}D))\cap h^{*}C \notag
\end{equation}
induced by the upper horizontal arrow.
Since $T^{*}_{X}X\times_{X}W$ is isomorphic to $T^{*}_{X}X(\log D)\times_{X}W$ by $\tau_{D,W}$,
the $\log$-$D$-$C$-transversality of $h$ is equivalent to
the $\tau_{D}^{-1}(C)$-transversality of $h$.
Hence the assertion holds.
\end{proof}

\begin{lem}
\label{logtransl}
Let $h\colon W\rightarrow X$ be a morphism of smooth schemes over $k$ and
assume that $h$ is $C_{D}$-transversal.
Let $Z$ be a closed subscheme of $X$
and $L\subset T^{\ast}X(\log D)\times_{X}Z$ a sub line bundle.
Then the following are equivalent:
\begin{enumerate}
\item $h$ is $\log$-$D$-$L$-transversal.
\item $dh^{D}(h^{*}L)$ is a sub line bundle of $T^{\ast}W(\log h^{\ast}D)\times_{W} h^{\ast}Z$.
\end{enumerate}
\end{lem}

\begin{proof}
Since $L$ is a sub line bundle of $T^{\ast}X(\log D)\times_{X}Z$, 
both the conditions (i) and (ii) are equivalent to
${dh^{D}}^{-1}(T^{\ast}_{W}W(\log h^{\ast}D))\cap h^{\ast}L
\subset T^{\ast}_{X}X(\log D)\times_{X}W$.
\end{proof}

For an element $\chi$ of $H^{1}_{\et}(U,\mathbf{Q}/\mathbf{Z})$ and
a $C_{D}$-transversal morphism $h\colon W\rightarrow X$ of smooth schemes over $k$,
let $dh^{D}\colon (h^{*}(\Omega^{1}_{X}(\log D)(R_{\chi}))|_{h^{*}Z_{\chi}}
\rightarrow \Omega^{1}_{W}(\log h^{*}D)(h^{*}R_{\chi})|_{h^{*}Z_{\chi}}$
be the morphism yielded by $h$ by abuse of notation.
If $(X,U,\chi)$ is clean, then let $C_{\chi}$ be the union of $L_{i,\chi}$ for $i\in I$.
Namely $C_{\chi}$ is the
sub line bundle of $T^{*}X(\log D)\times_{X}Z_{\chi}$ 
defined by the invertible subsheaf
$\dvr_{X}(-R_{\chi})|_{Z_{\chi}}\cdot \rsw(\chi)$
of $\Omega^{1}_{X}(\log D)|_{Z_{\chi}}$.

\begin{prop}
\label{cleaneq}
Let $\chi$ be an element of $H^{1}_{\et}(U,\BQ/\BZ)$ and
assume that $(X,U,\chi)$ is clean.  
Let $h\colon W\rightarrow X$ be a $C_{D}$-transversal 
morphism of smooth schemes over $k$.
\begin{enumerate}
\item Let $w$ be a point on $h^{*}Z_{\chi}$.
Then the following are equivalent:
\begin{enumerate}
\item $h$ is $\log$-$D$-$C_{\chi}$-transversal at $w$.
\item $dh^{D}(h^{*}\rsw(\chi))_{w}\notin 
\mathfrak{m}_{w}\Omega^{1}_{W}(\log h^{\ast}D)(h^{*}R_{\chi})|_{h^{\ast}Z_{\chi},w}$.
\end{enumerate}
\item The following are equivalent:
\begin{enumerate}
\item $h$ is $\log$-$D$-$C_{\chi}$-transversal at every generic point of $h^{*}Z_{\chi}$.
\item $R_{h^{*}\chi}=h^{*}R_{\chi}$ and
$\rsw(h^{\ast}\chi)=dh^{D}(h^{*}\rsw(\chi))$.
\end{enumerate}
\item The following are equivalent:
\begin{enumerate}
\item $h$ is $\log$-$D$-$C_{\chi}$-transversal at every point on $h^{*}Z_{\chi}$.
\item $R_{h^{*}\chi}=h^{*}R_{\chi}$, $\rsw(h^{\ast}\chi)=dh^{D}(h^{*}\rsw(\chi))$,
and the pull-back $(W,h^{\ast}U,h^{\ast}\chi)$ of $(X,U,\chi)$ by $h$ is clean.
\item $dh^{D}(h^{*}C_{\chi})$ is the sub line bundle $C_{h^{*}\chi}$ of $T^{\ast}W(\log h^{\ast}D)\times_{W} h^{\ast}Z_{\chi}$.
\item $h$ is $\log$-$D$-$SS^{\log}(X,U,\chi)$-transversal. 
\end{enumerate} 
If one of the equivalent conditions holds
and if $X$ and $W$ are purely of dimension $d$ and $c$ respectively, 
then we have
\begin{equation}
\label{hcss}
h^{\circ}SS^{\log}(X,U,\chi)=
SS^{\log}(W,h^{\ast}U,h^{\ast}\chi) 
\end{equation}
and 
\begin{equation}
\label{heclog}
h^{!}\Char^{\log}(X,U,\chi)=\Char^{\log}(W,h^{*}U,h^{*}\chi). 
\end{equation}
Here $h^{!}\Char^{\log}(X,U,\chi)$ is the push-forward of $(-1)^{d-c}h^{*}\Char^{\log}(X,U,\chi)$
by $dh^{D}$ in the sense of intersection theory.
\end{enumerate}
\end{prop}

\begin{proof}
(i) Since the condition (b) is equivalent to that $dh^{D}(h^{*}C_{\chi})\times_{W}w$ is a sub line bundle of
$T^{*}W(\log h^{*}D)\times_{W}w$, the assertion holds by Lemma \ref{logtransl}. 

(ii) We prove (a) $\Rightarrow$ (b).
Let $\mathfrak{q}_{i}$ be a generic point of $h^{*}D_{i}$ 
and $K_{i}'=\Frac\hat{\dvr}_{W,\mathfrak{q}_{i}}$ the local field at $\mathfrak{q}_{i}$
for $i\in I_{\mW,\chi}$.
Let $a$ be a section of $\fillog_{R_{\chi}}j_{*}W_{s}(\dvr_{U})$ whose image in
$H^{1}_{\et}(U,\mathbf{Q}/\mathbf{Z})$ is locally the $p$-part of $\chi$.
We put $a'=h^{*}a$.
We note that $a'$ is a section of $\fillog_{h^{*}R_{\chi}}j'_{*}W_{s}(\dvr_{h^{*}U})$
whose image in $H^{1}_{\et}(h^{*}U,\mathbf{Q}/\mathbf{Z})$ is locally the $p$-part of $h^{*}\chi$,
where $j'\colon h^{*}U\rightarrow W$ is the base change of $j$ by $h$.
By (i), we have $dh^{D}(h^{*}\rsw(\chi))_{\mathfrak{q_{i}}}\notin \mathfrak{m}_{\mathfrak{q_{i}}}\Omega^{1}_{W}(\log h^{*}D)(h^{*}R_{\chi})|_{h^{*}Z_{\chi},\mathfrak{q}_{i}}$ for every $i\in I_{\mW,\chi}$.
Since $dh^{D}(h^{*}\rsw(\chi))_{\mathfrak{q_{i}}}$ is the image of $a'$ for $i\in I_{\mW,\chi}$, 
the condition (b) holds
by Lemma \ref{lemrsw}.

Conversely, the implication (b) $\Rightarrow$ (a) holds by (i),
since $\rsw(h^{*}\chi)_{\mathfrak{q}_{i}}\neq 0$ in
$\Omega_{W}^{1}(\log h^{*}D)(h^{*}R_{\chi})|_{h^{*}Z_{\chi},\mathfrak{q}_{i}}$ for every $i\in I$.

(iii) We may assume $R_{h^{*}\chi}=h^{*}R_{\chi}$ and $\rsw(h^{*}\chi)=dh^{D}(h^{*}\rsw(\chi))$ 
to prove the equivalence of (a) and (b) by (ii).
Then the conditions (a) and (b) are equivalent by (i).
The equivalence of (a) and (c) holds by Lemma \ref{logtransl}.
Since $h$ is $\log$-$D$-$T^{*}_{X}X(\log D)$-transversal if (and only if) $h$ is $C_{D}$-transversal
by Lemma \ref{pullzers} and Proposition \ref{logtrctr}, the condition (d) is equivalent to that
$h$ is $\log$-$D$-$C_{\chi}$-transversal.
Since $C_{\chi}$ is a sub line bundle of $T^{*}X(\log D)\times_{X}Z_{\chi}$,
the last condition is equivalent to the condition (a).
Hence the equivalence of (a) and (d) holds.

Suppose that one of the equivalent conditions (a)--(d) holds.
Then the equalities (\ref{hcss}) and (\ref{heclog}) hold by (b).
\end{proof}

For a proper morphism $f\colon X'\rightarrow X$ of smooth schemes over $k$
and a closed conical subset $C'\subset T^{*}X'$, 
we define a closed conical subset 
$f_{\circ}C' \subset T^{\ast}X$ to be the image by the projection $T^{\ast}X\times_{X}X'\rightarrow T^{\ast}X$ of the inverse image of $C'$ by
$df\colon T^{\ast}X\times_{X}X'\rightarrow T^{\ast}X'$
(\cite[1.2]{be}, \cite[Definition 3.7]{sa4}).

For a $C_{D}$-transversal proper morphism $f\colon X'\rightarrow X$ of smooth schemes over $k$
and a closed conical subset 
$C'\subset T^{*}X'(\log f^{*}D)$, we define a closed conical subset $f_{\circ}C'\subset T^{*}X(\log D)$
to be the image by the projection $T^{*}X(\log D)\times_{X}X'\rightarrow T^{*}X(\log D)$ of
the inverse image of $C'$ by $df^{D}\colon T^{*}X(\log D)\times_{X}X'\rightarrow T^{*}X'(\log f^{*}D)$. 

\begin{lem}
\label{stablogtr}
Let $E$ be a divisor on $X$ with simple normal crossings and
$X'$ a smooth divisor on $X$ not contained in $E$.
Assume $D=X'\cup E$.
We put $V=X-E$ and let $i\colon X'\rightarrow X$ be the closed immersion.
\begin{enumerate}
\item Let $C\subset T^{*}X(\log E)$ be a closed conical subset.
Assume that $i$ is $\log$-$E$-$C$-transversal.
Then we have 
\begin{equation}
\label{eqtauinv}
\tau_{E/D}^{-1}(\tau_{E/D}(C))=C\cup i_{\circ}i^{\circ}C.
\end{equation}
\item Let $\chi'$ be an element of $H^{1}_{\et}(V,\BQ/\BZ)$ and 
$\chi\in H^{1}_{\et}(U,\mathbf{Q}/\mathbf{Z})$ the image of $\chi'$.
Then the following are equivalent:
\begin{enumerate}
\item $(X,U,\chi)$ is clean.
\item $(X,V,\chi')$ is clean and $i$ is $\log$-$E$-$SS^{\log}(X,V,\chi')$-transversal. 
\end{enumerate}
If one of the equivalent conditions holds, 
then we have
\begin{equation}
\label{ssvssu}
i^{\circ}SS^{\log}(X,V,\chi')=SS^{\log}(X^{\prime},i^{*}V, i^{\ast}\chi')
\end{equation}
and 
\begin{equation}
\label{tedinv}
\tau_{E/D}^{-1}(SS^{\log}(X,U,\chi))=SS^{\log}(X,V,\chi')\cup i_{\circ}SS^{\log}(X',i^{*}V,i^{*}\chi').
\end{equation}
\end{enumerate}
\end{lem}

\begin{proof}
(i) The assertion holds over $X-X'$,
since $\tau_{E/D}$ is an isomorphism outside $X'$.
We prove the assertion on $X'$.
We consider the commutative diagram
\begin{equation}
\xymatrix{
T^{*}X(\log E)\times_{X}X' \ar[rr]^-{\tau_{E/D, X'}} \ar[dr]_-{di^{E}} & & T^{*}X(\log D)\times_{X}X' \\
& T^{*}X'(\log i^{*}E), \ar[ur] }
\notag
\end{equation}
where the right slanting arrow is the canonical injection.
By the injectivity of the right slanting arrow, the left hand side of (\ref{eqtauinv}) is 
${di^{E}}^{-1}(di^{E}(i^{*}C))$ over $X'$.
Since ${di^{E}}^{-1}(di^{E}(i^{*}C))$ is $i_{\circ}i^{\circ}C$ over $X'$, the assertion holds.

(ii) By Lemma \ref{lemclndeg} (i), we may assume that $(X,V,\chi')$ is clean.
By Proposition \ref{cleaneq} (iii), 
the condition (b) is equivalent to that $di^{E}(i^{*}C_{\chi'})$ is a sub line bundle of $T^{*}X'(\log i^{*}E)
\times_{X'}i^{*}Z_{\chi'}$.
Since $Z_{\chi}=Z_{\chi'}$ and $\tau_{E/D,X'}$ is the composition 
of $di^{E}$ and the canonical injection $T^{*}X'(\log i^{*}E)\rightarrow T^{*}X(\log D)\times_{X}X'$,
the last condition is equivalent to that
$\tau_{E/D,X'}(i^{*}C_{\chi'})$ is a sub line bundle of $T^{*}X(\log D)\times_{X}i^{*}Z_{\chi}$.
Since $R_{\chi}=R_{\chi'}$ and $\rsw(\chi)$ is the image of $\rsw(\chi')$ in 
$\Omega^{1}_{X}(\log D)(R_{\chi})|_{Z_{\chi}}$,
the last condition is equivalent to that 
$\dvr_{X}(-R_{\chi})|_{i^{*}Z_{\chi}}\cdot \rsw(\chi)|_{i^{*}Z_{\chi}}$ defines a sub line bundle of 
$T^{*}X(\log D)\times_{X}i^{*}Z_{\chi}$.
Since the last condition is equivalent to the cleanliness of $(X,U,\chi)$ at every point on $X'$
and since $\rsw(\chi)=\rsw(\chi')$ outside $X'$, the assertion holds by the assumption
that $(X,V,\chi')$ is clean.

Suppose that one of the equivalent conditions (a) and (b) holds.
Then the equality (\ref{ssvssu}) holds by (b) and Proposition \ref{cleaneq} (iii).
Since $(X,U,\chi)$ and $(X,V,\chi')$ are clean, $R_{\chi}=R_{\chi'}$, 
and $\rsw(\chi)$ is the image of $\rsw(\chi')$ in $\Omega_{X}^{1}(\log D)(R_{\chi})|_{Z_{\chi}}$,
we have $SS^{\log}(X,U,\chi)=\tau_{E/D}(SS^{\log}(X,V,\chi'))$.
Hence the equality (\ref{tedinv}) holds by (\ref{ssvssu}) and (i) with $C=SS^{\log}(X,V,\chi')$.
\end{proof}

\begin{thm}
\label{proplogct}
If $(X,U,\chi)$ is clean,
then $j_{!}\mf$ is micro-supported on $\tau_{D}^{-1}(SS^{\log}(X,U,\chi))$.
\end{thm}

\begin{proof}
It is sufficient to prove $SS(j_{!}\mf)\subset \tau_{D}^{-1}(SS^{\log}(X,U,\chi))$. 
We first reduce the proof to the case where $\chi$ is of order a power of $p$.
Since $j$ is affine, the singular support $SS(j_{!}\mf)$ of $j_{!}\mf$ is stable
under the replacement of $\chi$ by the $p$-part of $\chi$ 
by \cite[Theorem 0.1]{sy} and Lemma \ref{suppCC}.
Since $SS^{\log}(X,U,\chi)$ is stable under the same replacement, we may assume
that $\chi$ is of order a power of $p$.
Then $\mf$ is unramified along $D_{\mT,\chi}$,
where $D_{\mT,\chi}=\bigcup_{i\in I_{\mT,\chi}}D_{i}$ and $I_{\mT,\chi}$ is as in Definition \ref{defswdiv}. 

Let $h\colon W\rightarrow X$ be a separated $\tau_{D}^{-1}(SS^{\log}(X,U,\chi))$-transversal morphism
of smooth schemes over $k$.
Since $C_{D}\subset \tau_{D}^{-1}(SS^{\log}(X,U,\chi))$ by Lemma \ref{pullzers},
the morphism $h$ is $\log$-$D$-$SS^{\log}(X,U,\chi)$-transversal
by Proposition \ref{logtrctr}.
Hence we have $R_{h^{\ast}\chi}=h^{\ast}R_{\chi}$ and $(W,h^{\ast}U,h^{\ast}\chi)$ is clean
by Proposition \ref{cleaneq} (iii).
We prove the assertion by the induction on $\sharp (I_{\mT,\chi})$.

If $\sharp(I_{\mT,\chi})=0$, then $\mf$ is totally wildly ramified along $D$
and hence the canonical morphism $j_{!}\mf \rightarrow Rj_{*}\mf$ is an isomorphism
by Proposition \ref{propcliso}.
Since $\mf$ is totally wildly ramified along $D$ and
$R_{h^{\ast}\chi}=h^{\ast}R_{\chi}$, the sheaf $h^{\ast}\mf$ is totally wildly ramified along $h^{*}D$.
Since $(W,h^{\ast}U,h^{\ast}\chi)$ is clean and $h^{\ast}\mf$ is totally wildly ramified along $h^{\ast}D$,
the canonical morphism $j'_{!}h^{*}\mf \rightarrow Rj'_{*}h^{*}\mf$ is an isomorphism 
by Proposition \ref{propcliso}.
Here $j'\colon h^{*}U\rightarrow W$ is the base change of $j$ by $h$.
Hence the morphism $h$ is $j_{!}\mf$-transversal 
by Lemma \ref{lemcltr}.
Since $T^{*}_{X}X\subset \tau_{D}^{-1}(SS^{\log}(X,U,\chi))$, 
the sheaf $j_{!}\mf$ is micro-supported on $\tau_{D}^{-1}(SS^{\log}(X,U,\chi))$ by Proposition 
\ref{propftr}.

Suppose $\sharp(I_{\mT,\chi})\ge 1$.
Let $r$ be an element of $I_{\mT,\chi}$.
We put $E=\bigcup_{i\in I\setminus \{r\}}D_{i}$,
$V=X-E$, and $V^{\prime}=X-D_{r}$.
We consider the cartesian diagram
\begin{equation}
\label{hvdiag}
\xymatrix{
U\ar[d]_{j_{2}^{\prime}} \ar[r]^-{j'_{1}} & V \ar[d]^{j_{2}} \\
V^{\prime} \ar[r]_{j_{1}} & X,
}  \notag
\end{equation}
where every morphism is the open immersion.
Let $\mf^{\prime}$ be a smooth sheaf of $\Lambda$-modules of rank $1$ on $V$ such that 
$\mf^{\prime}|_{U}=\mf$
and let $\chi^{\prime}\colon \pi^{\ab}_{1}(V)\rightarrow \Lambda^{\times}$ be 
the character corresponding to $\mf^{\prime}$ regarded as an element of 
$H^{1}_{\et}(V,\mathbf{Q}/\mathbf{Z})$
by the inclusion $\psi\colon \Lambda^{\times}\rightarrow \mathbf{Q}/\mathbf{Z}$.
Then $\chi\in H^{1}_{\et}(U,\mathbf{Q}/\mathbf{Z})$ is the image of $\chi'$.
Since $(X,V,\chi')$ is clean by Lemma \ref{lemclndeg} (i),
the sheaf $j_{2!}\mf^{\prime}$ is micro-supported on $\tau_{E}^{-1}(SS^{\log}(X,V,\chi^{\prime}))$ 
by the induction hypothesis. 

We consider the cartesian diagram
\begin{equation}
\xymatrix{
D_{r}\cap V\ar[d] \ar[r]^-{j_{2}^{\prime\prime}} &D_{r}
\ar[d]_{i_{1}} & h^{\ast}D_{r} \ar[l] \ar[d] \\
V \ar[r]_{j_{2}} & X & W. \ar[l]^-{h}
} \notag
\end{equation}
Since $(X,U,\chi)$ is clean, the closed immersion $i_{1}$ is $\log$-$E$-$SS^{\log}(X,V,\chi^{\prime})$-transversal 
and we have $i^{\circ}_{1}SS^{\log}(X,V,\chi^{\prime})=SS^{\log}(D_{r},i_{1}^{*}V, i_{1}^{\ast}\chi')$ by Lemma \ref{stablogtr} (ii).
Since $i_{1}$ is $\log$-$E$-$SS^{\log}(X,V,\chi')$-transversal,
we have $R_{i_{1}^{\ast}\chi^{\prime}}=i_{1}^{*}R_{\chi^{\prime}}$ and
$(D_{r},i_{1}^{*}V, i_{1}^{\ast}\chi')$ is clean by Proposition \ref{cleaneq} (iii).
Hence the sheaf
$j_{2!}^{\prime\prime}i_{1}^{\ast}\mf^{\prime}\simeq i_{1}^{*}j_{2!}\mf'$ is 
micro-supported on $\tau^{-1}_{i_{1}^{*}E}(SS^{\log}(D_{r},i_{1}^{*}V, i_{1}^{\ast}\chi'))$
by the induction hypothesis.
Since $i_{1}$ is a closed immersion, the sheaf $i_{1*}i_{1}^{\ast}j_{2!}\mf^{\prime}$ is
micro-supported on $i_{1\circ}\tau_{i_{1}^{*}E}^{-1}(SS^{\log}(D_{r},i_{1}^{*}V, i_{1}^{\ast}\chi'))$
by \cite[Lemma 2.2 (ii)]{be}.
Hence the sheaf $j_{1!}j^{*}_{1}j_{2!}\mf'\simeq j_{!}\mf$ is micro-supported on 
$C'=\tau_{E}^{-1}(SS^{\log}(X,V,\chi'))\cup i_{1\circ}
\tau_{i_{1}^{*}E}^{-1}(SS^{\log}(D_{r},i_{1}^{*}V, i_{1}^{\ast}\chi'))$ 
by Lemma \ref{ssdisttri}.

We prove $\tau_{D}^{-1}(SS^{\log}(X,U,\chi))=C'$.
Since $(X,U,\chi)$ is assumed to be clean, we have
$\tau_{D}^{-1}(SS^{\log}(X,U,\chi))=\tau_{E}^{-1}(SS^{\log}(X,V,\chi'))\cup 
\tau_{E}^{-1}(i_{1\circ}SS^{\log}(D_{r},i_{1}^{*}V,i_{1}^{*}\chi'))$ by Lemma \ref{stablogtr} (ii).
We consider the commutative diagram
\begin{equation}
\xymatrix{
T^{*}X \ar[d]_-{\tau_{E}} & T^{*}X\times_{X}D_{r} \ar[l]_-{\pr_{1}} \ar[r]^-{di_{1}} 
\ar[d]^-{\tau_{E,D_{r}}}
& T^{*}D_{r} \ar[d]^-{\tau_{i_{1}^{*}E}} 
\\
T^{*}X(\log E) & T^{*}X(\log E)\times_{X}D_{r} \ar[l]^-{\pr_{1}} \ar[r]_-{di_{1}^{E}} 
& T^{*}D_{r}(\log i_{1}^{*}E).
}\notag
\end{equation}
Since the left square is cartesian and the right square is commutative, we have
$\tau_{E}^{-1}(i_{1\circ}SS^{\log}(D_{r},i_{1}^{*}V,i_{1}^{*}\chi'))
=i_{1\circ}\tau_{i_{1}^{*}E}^{-1}(SS^{\log}(D_{r},i_{1}^{*}V, i_{1}^{\ast}\chi'))$.
Hence the assertion holds.
\end{proof}

\begin{conj}
\label{conjcc}
Let $\tau_{D}^{!}: CH_{2}(SS^{\log}(X,U,\chi))\rightarrow CH_{2}(\tau_{D}^{-1}(SS^{\log}(X,U,\chi)))$ 
is the Gysin homomorphism for the l.c.i.\ morphism $\tau_{D} \colon T^{\ast}X\rightarrow T^{\ast}X(\log D)$ 
(\cite[6.6]{ful}). 
If $(X,U,\chi)$ is clean,
then we have 
\begin{equation}
CC(j_{!}\mf)=\tau_{D}^{!}(\Char^{\log}(X,U,\chi)) \notag
\end{equation} 
in $CH_{d}(\tau_{D}^{-1}(SS^{\log}(X,U,\chi)))$.
\end{conj}

We prove Conjecture \ref{conjcc} in the case where $X$ is purely of dimension $2$ in Corollary \ref{corconj} later.

\begin{lem}
\label{lemcharksc}
Assume that $(X,U,\chi)$ is strongly clean.
\begin{enumerate}
\item $\tau_{D}^{-1}(SS^{\log}(X,U,\chi))=\bigcup_{I'\subset I}T^{*}_{D_{I'}}X$. 
\item Let $\tau_{D}^{!}: CH_{2}(SS^{\log}(X,U,\chi))\rightarrow CH_{2}(\tau_{D}^{-1}(SS^{\log}(X,U,\chi)))$ 
is the Gysin homomorphism for the l.c.i.\ morphism $\tau_{D} \colon T^{\ast}X\rightarrow T^{\ast}X(\log D)$ 
(\cite[6.6]{ful}). 
Then $\tau_{D}^{!}(\Char^{\log}(X,U,\chi))$ is defined as a $d$-cycle
and we have 
\begin{equation}
\tau_{D}^{!}(\Char^{\log}(X,U,\chi))=
(-1)^{d}\sum_{I'\subset I}(1+\sum_{i\in I'}\sw(\chi|_{K_{i}}))[T^{\ast}_{D_{I'}}X]. \notag
\end{equation}
\end{enumerate}
\end{lem}

\begin{proof}
Let $x$ be a closed point of $X$.
Since the assertions are local, it is sufficient to prove the assertions in a neighborhood of $x$.
We may assume $I=I_{x}$.

If $x\notin Z_{\chi}$, then the assertions hold by Lemma \ref{pullzers}.

Suppose $x\in Z_{\chi}$.
We use the notation as in (\ref{rswchix}) and the proof of Lemma \ref{pullzers}.
Let $\mathcal{I}$ and $\mathcal{J}$ be the defining ideal sheaves of $L_{1,\chi}\subset \cotx(\log D)\times_{X}D_{1}$ and
$\tau_{D}^{-1}(L_{1,\chi})\subset \cotx\times_{X}D_{1}$ respectively.
Since $\alpha_{1}$ is invertible in $\dvr_{D_{1},x}$, we have
\begin{align}
\mathcal{I}_{x}=(\alpha_{2}\partial/(\partial \log  t_{1})&-\alpha_{1}\partial /(\partial\log t_{2}), \ldots ,
\alpha_{r}\partial/(\partial \log t_{1})-\alpha_{1}\partial /(\partial\log t_{r}),\notag \\
&\beta_{r+1}\partial/(\partial\log t_{1})-\alpha_{1}\partial/\partial t_{r+1},\ldots ,\beta_{d}\partial/(\partial\log t_{1})-\alpha_{1}\partial/\partial t_{d}).\notag
\end{align}
Hence we have 
\begin{equation}
\mathcal{J}_{x}=(t_{2}\partial/\partial t_{2},\ldots ,t_{r}\partial/\partial t_{r},\partial /\partial t_{r+1},
\ldots ,\partial /\partial t_{d}), \notag
\end{equation}
which implies $\tau_{D}^{-1}(L_{1,\chi})=\bigcup_{1\in I'\subset I}T^{*}_{D_{I'}}X$ and
$\tau_{D}^{!}([L_{1,\chi}])=\sum_{1\in I'\subset I}[T^{*}_{D_{I'}}X]$.
Similarly as the case where $i=1$, we have $\tau_{D}^{-1}(L_{i,\chi})=\bigcup_{i\in I'\subset I}T^{*}_{D_{I'}}X$ and
$\tau_{D}^{!}([L_{i,\chi}])=\sum_{i\in I'\subset I}[T^{*}_{D_{I'}}X]$ for $i\in I$.
Hence the assertions hold by Lemma \ref{pullzers}.
\end{proof}

\section{Canonical lifting of logarithmic characteristic cycle}
\label{scanlift}

Let $X$ be a smooth scheme purely of dimension $d$ over a perfect field $k$ of characteristic $p>0$.
Let $D$ be a divisor on $X$ with simple normal crossings and $\{D_{i}\}_{i\in I}$ the irreducible components of $D$.
Let $\dvr_{K_{i}}=\hat{\dvr}_{X,\mathfrak{p}_{i}}$ be the completion of the local ring at 
the generic point $\mathfrak{p}_{i}$ of $D_{i}$
and $K_{i}=\Frac \dvr_{K_{i}}$ the local field at $\mathfrak{p}_{i}$ for $i\in I$.
We put $U=X-D$ and let $j\colon U\rightarrow X$ be the open immersion.
We note that $j$ is an affine open immersion.
Let $\mf$ be a smooth sheaf of $\Lambda$-modules of rank $1$ on $U$ 
and let $\chi\colon \pi_{1}^{\ab}(U)\rightarrow \Lambda^{\times}$ 
be the character corresponding to $\mf$.
We fix an inclusion $\psi\colon \Lambda^{\times}\rightarrow \mathbf{Q}/\mathbf{Z}$ 
and regard $\chi$ as an element of $H^{1}_{\et}(U,\mathbf{Q}/\mathbf{Z})$ by $\psi$.

For the subsets $I_{\mI,\chi}$ and $I_{\mII,\chi}$ of $I$
(the remark after Lemma \ref{alnlsch}), we put 
$Z_{\mI,\chi}=\bigcup_{i\in I_{\mI,\chi}}D_{i}$ and $Z_{\mII,\chi}=\bigcup_{i\in I_{\mII,\chi}}D_{i}$.
Let $L_{i,\chi}'$, $L_{i,\chi}''$, $[L_{i,\chi}']$, and $r_{i}'$ be as in the end of Subsection \ref{sslram} for $i\in I$.
For a closed point $x$ of $D$, we put $I_{x}=\{1,\ldots, r\}$ and $I_{\mW,\chi,x}=\{1,\ldots,r'\}$
(Definition \ref{defswdiv}), where $r'\le r\le d$.
Let $(t_{1},\ldots,t_{d})$ be a local coordinate system at $x$ such that $t_{i}$
is a local equation of $D_{i}$ for $i\in I_{x}$.

Assume $d=2$.
We construct a canonical lifting $\Char^{K}(X,U,\chi)\in Z_{2}(T^{*}X)$ 
of Kato's logarithmic characteristic cycle using ramification theory
(Theorem \ref{totalpullback}).
For a closed point $x$ of $D$,
we put
\begin{align}
\label{tx}
t_{x}= \sharp (I_{x})-1+s_{x} +\sum_{i\in I_{\mW,\chi,x}}
\sw(\chi&|_{K_{i}})(\ord^{\prime}(\chi;x,D_{i})
-\ord(\chi;x,D_{i}))  \\
&+\sum_{i\in I_{\mII,\chi,x}}(\ord(\chi;x,D_{i})+1-\sharp(I_{x})), \notag
\end{align}
where $\ord(\chi;x,D_{i})$ is as in Definition \ref{defofclean} (ii),
$\ord'(\chi;x,D_{i})$ is as in Definition \ref{defofnondeg} (ii),
and $s_{x}$ is as in (\ref{logcharcycle}).
We define a $2$-cycle $\Char^{K}(X,U,\chi)$ on $T^{*}X$ by
\begin{equation} 
\label{canlift}
\Char^{K}(X,U,\chi)=[\zers] + \sum_{i\in I}r_{i}'[L_{i,\chi}^{\prime}]+\sum_{x \in |D|} t_{x}[\spf],
\end{equation}
where $|D|$ is the set of closed points of $D$.
Since the Swan conductor, the total dimension, and the characteristic form are defined \'{e}tale
locally, we have
\begin{equation}
h^{\ast}\Char^{K}(X,U,\chi)=\Char^{K}(W,h^{*}U,h^{*}\chi) \notag
\end{equation}
for an \'{e}tale morphism $h\colon W\rightarrow X$.
The $2$-cycle $\Char^{K}(X,U,\chi)$ is compatible with the pull-back by the projection
$X_{\bar{k}}\rightarrow X$.
Since the Swan conductor, the total dimension, and the characteristic form are stable under the replacement
of $\chi$ by the $p$-part of $\chi$,
the $2$-cycle $\Char^{K}(X,U,\chi)$ is stable under the replacement of $\chi$ by the $p$-part of $\chi$.
We write $SS^{K}(X,U,\chi)$ for the support of $\Char^{K}(X,U,\chi)$.

\begin{prop}
\label{propndeg}
Assume $d=2$.
Let $x$ be a closed point of $D$ 
and assume that $(X,U,\chi)$ is clean at $x$.
\begin{enumerate}
\item If $(X,U,\chi)$ is non-degenerate at $x$,
then we have $CC(j_{!}\mf)=\Char^{K}(X,U,\chi)$ in a neighborhood of $x$.
\item If $\sharp(I_{\mT,\chi,x})=\sharp(I_{\mW,\chi,x})=1$, then 
we have $CC(j_{!}\mf)=\Char^{K}(X,U,\chi)$ in a neighborhood of $x$.
\end{enumerate}
\end{prop}

\begin{proof}
Since the first two terms in the right hand sides in (\ref{saitoccro}) and (\ref{canlift}) are equal,
it is sufficient to prove that $u_{x}$ in (\ref{saitoccro}) and $t_{x}$ in (\ref{canlift}) are equal.

(i) Since $(X,U,\chi)$ is assumed to be non-degenerate at $x$,
we have $I_{x}=I_{\mT,\chi,x}$ or $I_{x}=I_{\mW,\chi,x}$ by Lemma \ref{lemint} (i).
Since $(X,U,\chi)$ is assumed to be clean and non-degenerate at $x$,
we have $s_{x}=0$ and $\ord'(\chi; x,D_{i})=\ord (\chi; x,D_{i})=0$ for every $i\in I_{\mW,\chi,x}$
in (\ref{tx}).
Hence, if $I_{x}=I_{\mT,\chi, x}$, then we have $t_{x}=\sharp(I_{x})-1$ by (\ref{tx}).
If $I_{x}=I_{\mW,\chi,x}$, then we have $t_{x}=0$ by Lemma \ref{itsn} (i) and (ii) and (\ref{tx}), since $d=2$.
Hence we have $t_{x}=u_{x}$ by \cite[Theorem 7.14]{sa4}.

(ii) Since $CC(j_{!}\mf)$ is stable under the replacement
of $\chi$ by the $p$-part of $\chi$ by \cite[Theorem 0.1]{sy}
and so is $\Char^{K}(X,U,\chi)$, we may assume
that $\chi$ is of order $p^{s}$ for an integer $s\ge 0$.
Since the assertion is local, we may assume $I=I_{x}$.
We put $I_{\mW,\chi}=\{1\}$ and $I_{\mT,\chi}=\{2\}$.
Then $\mf$ is unramified along $D_{2}$.

We put $V=X-D_{1}$ and $V'=X-D_{2}$. 
We consider the cartesian diagram
\begin{equation}
\xymatrix{
U \ar[r]^-{j_{1}'} \ar[d]_-{j_{2}'} & V \ar[d]_-{j_{2}} & V\cap D_{2} \ar[l]_-{i_{1}'} \ar[d]^-{j_{2}''} \\
V' \ar[r]_-{j_{1}} & X & D_{2} \ar[l]^-{i_{1}}
} \notag
\end{equation}
where $i_{1}$ and $i_{1}'$ are closed immersions and the others are open immersions.
Let $\mf'$ be a smooth sheaf of $\Lambda$-modules of rank $1$ on $V$ such that $\mf'|_{U}=\mf$
and let $\chi'\colon \pi_{1}^{\ab}(V)\rightarrow \Lambda^{\times}$ be the character
corresponding to $\mf'$.
We regard $\chi'$ as an element of $H^{1}_{\et}(V,\mathbf{Q}/\mathbf{Z})$ by $\psi$.
Then we note that $R_{\chi}=R_{\chi'}$ and $R_{\chi}'=R_{\chi'}'$.
Since $(X,U,\chi)$ is assumed to be clean at $x$ and $\sharp(I_{x})=2$, 
we have $I_{\mW,\chi}=I_{\mI,\chi}$ by Lemma \ref{itsn} (i).
Since $R_{\chi}=R_{\chi'}$ and $R_{\chi}'=R_{\chi'}'$, we have $I_{\mW,\chi'}=I_{\mI,\chi'}$.
By \cite[Lemma 5.13.1]{sa4}, we have 
\begin{equation} 
\label{eqccdt}
CC(j_{!}\mf)=CC(j_{1!}j_{1}^{*}j_{2!}\mf')=
CC(j_{2!}\mf')-CC(i_{1*}i_{1}^{*}j_{2!}\mf').
\end{equation}

Since $(X,U,\chi)$ is assumed to be clean at $x$, the triple $(X,V,\chi')$ is clean at $x$ 
by Lemma \ref{lemclndeg} (i) and $(X,V,\chi')$ is non-degenerate at $x$ by Lemma \ref{lemclndeg} (ii).
By (i), we have
\begin{equation}
\label{ccjtf}
CC(j_{2!}\mf')=\Char^{K}(X,V,\chi')=[T^{*}_{X}X]+\dt(\chi'|_{K_{1}})[L_{1,\chi'}'] 
\end{equation}
in a neighborhood of $x$.

Since $(X,U,\chi)$ is assumed to be clean at $x$, the closed immersion $i_{1}$ is 
$\log$-$D_{1}$-$SS^{\log}(X,V,\chi')$-transversal in a neighborhood of $x$ by Lemma \ref{stablogtr} (ii).
Since $I_{\mW,\chi'}=I_{\mI,\chi'}=\{1\}$, we have $L_{1,\chi'}'=T^{*}_{D_{1}}X$ by Lemma \ref{lemint} (ii).
Hence we have
\begin{equation}
\label{ssjtf}
SS^{K}(X,V,\chi')=T^{*}_{X}X\cup T^{*}_{D_{1}}X=C_{D_{1}}
\end{equation}
by (\ref{ccjtf}).
By Proposition \ref{logtrctr},
the closed immersion $i_{1}$ is $SS^{K}(X,V,\chi')$-transversal. 
Since $j$ is an affine open immersion, we have
$SS(j_{2!}\mf')=SS^{K}(X,V,\chi')$ in a neighborhood of $x$ by Lemma \ref{suppCC} and (\ref{ccjtf}).
Since $i_{1}^{*}SS^{K}(X,V,\chi')$ is purely of dimension $1$ in a neighborhood of $x$ by (\ref{ssjtf}), 
the closed immersion $i_{1}$ is 
properly $SS(j_{2!}\mf')$-transversal in a neighborhood of $x$.
Hence we have 
\begin{equation}
\label{eqccps}
CC(i_{1}^{*}j_{2!}\mf')=i_{1}^{!}CC(j_{2!}\mf')=-([T^{*}_{D_{2}}D_{2}]+\dt(\chi'|_{K_{1}})[T^{*}_{x}D_{2}])
\end{equation}
in a neighborhood of $x$ by Theorem \ref{thmpull}.
Since $i_{1}$ is a closed immersion of smooth schemes over $k$,
we have
\begin{equation}
\label{cciijtf}
CC(i_{1*}i_{1}^{*}j_{2!}\mf')=-([T^{*}_{D_{2}}X]+\dt(\chi'|_{K_{1}})[T^{*}_{x}X]) 
\end{equation}
in a neighborhood of $x$ by (\ref{eqccps}) and \cite[Lemma 5.13.2]{sa4}.
Hence we have
\begin{equation}
CC(j_{!}\mf)=[T^{*}_{X}X]+\dt(\chi'|_{K_{1}})[T^{*}_{D_{1}}X]+[T^{*}_{D_{2}}X]+\dt(\chi'|_{K_{1}})[T^{*}_{x}X] \notag
\end{equation}
by (\ref{eqccdt}), (\ref{ccjtf}), and (\ref{cciijtf}).

Since $\dt(\chi|_{K_{1}})=\dt(\chi'|_{K_{1}})$ and $I_{\mW,\chi}=I_{\mI,\chi}=\{1\}$,
it is sufficient to prove $t_{x}=\dt(\chi|_{K_{1}})$ by Lemma \ref{lemint} (ii).
Since $(X,U,\chi)$ is assumed to be clean, we have $s_{x}=0$ and $\ord(\chi; x,D_{1})=0$ in (\ref{tx}).
Since $I_{\mW,\chi}=I_{\mI,\chi}=\{1\}$ and $\sharp(I)=2$,
we have $\ord'(\chi;x,D_{1})=1$ by Lemma \ref{lemordp} and Lemma \ref{itsn} (iv).
Hence we have $t_{x}=1+\sw(\chi|_{K_{1}})=\dt(\chi|_{K_{1}})$ by (\ref{tx}).
\end{proof}

\begin{thm}
\label{totalpullback}
Assume $d=2$.
\begin{enumerate}
\item $SS^{K}(X,U,\chi)\subset \tau_{D}^{-1}(SS^{\log}(X,U,\chi))$.
\item Let $\tau_{D}^{!}: CH_{2}(SS^{\log}(X,U,\chi))\rightarrow CH_{2}(\tau_{D}^{-1}(SS^{\log}(X,U,\chi)))$ 
be the Gysin homomorphism for the l.c.i.\ morphism $\tau_{D} \colon T^{\ast}X\rightarrow T^{\ast}X(\log D)$ 
(\cite[6.6]{ful}).
Then we have 
\begin{equation}
\Char^{K}(X,U,\chi)=\tau_{D}^{!}(\Char^{\log}(X,U,\chi)) \notag
\end{equation}
in $CH_{2}(\tau_{D}^{-1}(SS^{\log}(X,U,\chi)))$.
\end{enumerate}
\end{thm}
 
We prepare some lemmas to prove of Theorem \ref{totalpullback}.

\begin{lem}
\label{pullofli}
Assume $d=2$.
Let $i$ be an element of $I_{\mW,\chi}$. 
\begin{enumerate}
\item If $i \in I_{\mI,\chi}$, 
then $\tau_{D}^{-1}(L_{i,\chi})$ is the union of $L_{i,\chi}^{\prime}$ and 
$\spf$ for $x\in Z_{\mII,\chi}\cap D_{i}$ such that 
$\ord'(\chi;x,D_{i})-\ord(\chi;x,D_{i})+1\neq 0$ and for closed points 
$x$ of $D_{i}\setminus(Z_{\mII,\chi}\cap D_{i})$ 
such that $\ord'(\chi;x,D_{i})-\ord(\chi;x,D_{i})\neq 0$.
Further we have 
\begin{equation}
\tau_{D}^{!}([L_{i,\chi}])=[L_{i,\chi}']+
\sum_{x\in |D_{i}|}(\ord'(\chi;x,D_{i})-\ord(\chi;x,D_{i}))
[T^{*}_{x}X]+\sum_{x\in Z_{\mII,\chi}\cap D_{i}}[T^{*}_{x}X], \notag
\end{equation}
where $\tau_{D}^{!}: CH_{2}(L_{i,\chi})\rightarrow CH_{2}(\tau_{D}^{-1}(L_{i,\chi}))$ 
is the Gysin homomorphism for the l.c.i.\ morphism 
$\tau_{D} \colon T^{\ast}X\rightarrow T^{\ast}X(\log D)$ (\cite[6.6]{ful}).
\item If $i\in I_{\mII,\chi}$, then $\tau_{D}^{-1}(L_{i,\chi})=\cotx\times_{X}D_{i}$.
\end{enumerate}
\end{lem}

\begin{proof}
We may assume $i=1$.
Let $x$ be a closed point of $D_{1}$.
Let $\mathcal{I}$ and $\mathcal{J}$ be the defining ideal sheaves of $L_{1,\chi}\subset \cotx(\log D)\times_{X}D_{1}$ and
$\tau_{D}^{-1}(L_{1,\chi})\subset \cotx\times_{X}D_{1}$ respectively.

We use the notation as in (\ref{rswchix}) and the proof of Lemma \ref{pullzers}.
Then we have 
\begin{equation}
\mathcal{I}_{x}=
\begin{cases}
(t_{2}^{-\ord(\chi;x,D_{1})}(\beta_{2}\partial/(\partial \log t_{1})-\alpha_{1}\partial/\partial t_{2})) 
& (\sharp (I_{x})=1), \\
(t_{2}^{-\ord(\chi;x,D_{1})}(\alpha_{2}\partial/(\partial \log t_{1})
-\alpha_{1}\partial/(\partial \log t_{2}))) & (\sharp (I_{x})=2).
\end{cases} \notag
\end{equation}
Therefore we have
\begin{equation}
\mathcal{J}_{x}=
\begin{cases}
(t_{2}^{-\ord(\chi;x,D_{1})}\alpha_{1}\partial/\partial t_{2}) & (\sharp (I_{x})=1), \\
(t_{2}^{-\ord(\chi;x,D_{1})}\alpha_{1}t_{2}\partial/\partial t_{2}) & (\sharp (I_{x})=2).
\end{cases} \notag
\end{equation}

(i) Let $n$ be the normalized valuation of $\alpha_{1}$ in $\dvr_{D_{1},x}$.
Then we have $\ord'(\chi;x,D_{1})=n+\sharp(I_{\mI,\chi,x}\cup I_{\mT,\chi,x})-1$ by Lemma \ref{lemordp}.
Hence we have 
\begin{equation}
\label{defidi}
\mathcal{J}_{x}=
\begin{cases}
(t_{2}^{\ord'(\chi;x,D_{1})-\ord(\chi;x,D_{1})+1}\partial/\partial t_{2})
& (x \in Z_{\mII,\chi}\cap D_{1}), \\
(t_{2}^{\ord'(\chi;x,D_{1})-\ord(\chi;x,D_{1})}\partial/\partial t_{2})
& (\text{otherwise}).
\end{cases} 
\end{equation}
Since $L_{1,\chi}'=T^{\ast}_{D_{1}}X$ by Lemma \ref{lemint} (ii) and the assertions are local, the assertions hold.

(ii) Since $\chi|_{K_{1}}$ is of type $\mII$, we have $\alpha_{1}=0$ in $\dvr_{D_{1},x}$ by (\ref{relloc}).
Hence $\mathcal{J}_{x}=(0)$.
Since the assertion is local, the assertion holds.
\end{proof}

For $i\in I$, let $p_{i}\colon \tau_{D}^{-1}(L_{i,\chi})\rightarrow D_{i}$ be the canonical projection.
We note that $\tau_{D}^{-1}(L_{i,\chi})=T^{*}X\times_{X}D_{i}$ for $i\in I_{\mII,\chi}$ by Lemma \ref{pullofli} (ii).

\begin{lem}
\label{lemoftypeiip}
Assume $d=2$.
Let $i$ be an element of $I_{\mII,\chi}$.
Let $\tau_{D}^{!}\colon CH_{2}(T^{*}X(\log D)\times_{X}D_{i})
\rightarrow CH_{2}(T^{*}X\times_{X}D_{i})$ be 
the Gysin homomorphism for the l.c.i.\ morphism $\tau_{D} \colon T^{\ast}X\rightarrow T^{\ast}X(\log D)$ (\cite[6.6]{ful}).
Then we have
\begin{equation}
\tau_{D}^{!}([L_{i,\chi}])=p^{\ast}_{i}(c_{1}(\cotx(\log D)\times_{X}D_{i})\cap [D_{i}] - c_{1}(\dvr_{X}(-R_{\chi})|_{D_{i}})\cap[D_{i}]-\sum_{x\in |D_{i}|}\ord(\chi;x,D_{i})[\{x\}])\notag
\end{equation} 
in $CH_{2}(\tau_{D}^{-1}(L_{i,\chi}))=CH_{2}(\cotx\times_{X}D_{i})$.
\end{lem}

\begin{proof}
Let $q_{i}\colon T^{\ast}X(\log D)\times_{X}D_{i}\rightarrow D_{i}$ be the canonical projection.
By the excess intersection formula and the Whitney sum formula,
we have
\begin{equation}
[L_{i,\chi}]= q_{i}^{\ast}(c_{1}(T^{\ast}X(\log D)\times_{X}D_{i})\cap [D_{i}]-c_{1}(L_{i,\chi})\cap [D_{i}]) \notag
\end{equation}
in $CH_{2}(T^{*}X(\log D)\times_{X}D_{i})$.
Hence we have
\begin{equation}
\tau_{D}^{!}([L_{i,\chi}])= p_{i}^{\ast}(c_{1}(T^{\ast}X(\log D)\times_{X}D_{i})\cap [D_{i}]-c_{1}(L_{i,\chi})\cap [D_{i}]) \notag
\end{equation} 
in $CH_{2}(T^{*}X\times_{X}D_{i})$.
Since $L_{i,\chi}$ is the sub line bundle of $T^{\ast}X(\log D)\times_{X}D_{i}$ defined by 
the image of the injection   
\begin{equation}
\dvr_{X}(-R_{\chi})\otimes_{\dvr_{X}} \dvr_{D_{i}}(\sum_{x\in |D_{i}|}\ord(\chi;x,D_{i})x|_{D_{i}})
\xrightarrow{\times \rsw(\chi)|_{D_{i}}} \Omega^{1}_{X}(\log D)|_{D_{i}} \notag
\end{equation}
defined by the multiplication by $\rsw(\chi)|_{D_{i}}$,
the assertion holds.
\end{proof}

Let $|D|'$ and $|D_{i}|'$ for $i\in I$ be the sets of closed points $x$ of $D$ and $D_{i}$ such that $\sharp(I_{x})=d=2$
respectively.

\begin{lem}
\label{lemoftypeii}
Assume $d=2$.
Let $i$ be an element of $I_{\mII,\chi}$.
\begin{enumerate}
\item If $\chi|_{K_{i}}$ is of usual type, then we have
\begin{equation}
[L_{i,\chi}^{\prime}]=p_{i}^{\ast}(c_{1}(\cotx\times_{X}D_{i}) \cap [D_{i}] - c_{1}(\dvr_{X}(-R_{\chi}^{\prime})|_{D_{i}})\cap [D_{i}] - \sum_{x\in |D_{i}|} \ord'(\chi;x,D_{i})[\{x\}]) \notag
\end{equation}
in $CH_{2}(\tau_{D}^{-1}(L_{i,\chi}))=CH_{2}(\cotx\times_{X}D_{i})$.
\item If $\chi|_{K_{i}}$ is of exceptional type, then we have
\begin{equation}
[L_{i,\chi}^{\prime}]=2p_{i}^{\ast}(c_{1}(\cotx\times_{X}D_{i}) \cap [D_{i}] - c_{1}(\dvr_{X}(-R_{\chi}^{\prime})|_{D_{i}})\cap [D_{i}] - \sum_{x\in |D_{i}|} \ord'(\chi;x,D_{i})[\{x\}]) \notag 
\end{equation}
in $CH_{2}(\tau_{D}^{-1}(L_{i,\chi}))=CH_{2}(\cotx\times_{X}D_{i})$.
\item We have 
\begin{equation}
[\nori]=p^{\ast}_{i}(c_{1}(\dvr_{X}(-R_{\chi})|_{D_{i}})\cap [D_{i}]+\sum_{x\in |D_{i}|}\ord(\chi;x,D_{i})[\{x\}]-\sum_{x\in |D_{i}|'}
[\{x\}]) \notag
\end{equation}
in $CH_{2}(\tau_{D}^{-1}(L_{i,\chi}))=CH_{2}(\cotx\times_{X}D_{i})$. 
\end{enumerate}
\end{lem}

\begin{proof}
(i) By the excess intersection formula and the Whitney sum formula,
we have 
\begin{equation}
[L_{i,\chi}^{\prime}]=p_{i}^{\ast}(c_{1}(T^{\ast}X\times_{X}D_{i})\cap [D_{i}]
-c_{1}(L_{i,\chi}^{\prime})\cap [D_{i}]) \notag 
\end{equation}
in $CH_{2}(T^{*}X\times_{X}D_{i})$.
Since the sub line bundle $L_{i,\chi}^{\prime}$ of $\cotx \times_{X}D_{i}$ 
is defined by the image of the injection 
\begin{equation}
\dvrx(-R_{\chi}^{\prime})\otimes_{\dvr_{X}}
\dvr_{D_{i}}(\sum_{x\in |D_{i}|}\ord'(\chi;x,D_{i})x|_{D_{i}})
\xrightarrow{\times\cform(\chi)|_{D_{i}^{1/p}}} \Omegax|_{D_{i}} \notag
\end{equation}
defined by the multiplication by $\cform(\chi)|_{D_{i}^{1/p}}$,
where $\cform(\chi)|_{D_{i}^{1/p}}$ is regarded as a global section of $\Omega_{X}^{1}(R_{\chi}')|_{D_{i}}$,
the assertion holds.

(ii) We note that if $\chi|_{K_{i}}$ is of exceptional type then we have $p=2$.
Let $p_{i}^{\prime}\colon T^{\ast}X\times_{X}D_{i}^{1/2}\rightarrow D_{i}^{1/2}$ be the canonical projection.
Similarly as in the proof of (i), we have
\begin{equation}
[L_{i,\chi}^{\prime\prime}]={p_{i}^{\prime}}^{\ast}(c_{1}(T^{\ast}X\times_{X}D_{i}^{1/2})\cap [D_{i}^{1/2}]-
c_{1}(L_{i,\chi}^{\prime\prime})\cap [D_{i}^{1/2}]). \notag
\end{equation} 
Since $L_{i,\chi}^{\prime\prime}$ is the sub line bundle of $\cotx\times_{X}D^{1/2}$ defined by the image of the injection
\begin{equation}
\dvrx(-R_{\chi}^{\prime})\otimes_{\dvr_{X}}
\dvr_{D_{i}^{1/2}}(\sum_{x^{\prime}\in |D_{i}^{1/2}|}2\ord'(\chi;x,D_{i})x^{\prime}|_{D_{i}^{1/2}})
\xrightarrow{\times\cform(\chi)|_{D_{i}^{1/2}}} \Omegax\otimes_{\dvr_{X}}\dvr_{D_{i}^{1/2}} \notag
\end{equation}
defined by the multiplication by $\cform(\chi)|_{D_{i}^{1/2}}$,
we have 
\begin{align}
[L_{i,\chi}^{\prime\prime}]&=p_{i}^{\prime \ast}(c_{1}(\cotx\times_{X}D_{i}^{1/2})\cap [D_{i}^{1/2}]\notag \\
&\qquad\qquad -c_{1}(\dvrx(-R_{\chi}^{\prime})\otimes_{\dvr_{X}}
\dvr_{D_{i}^{1/2}})\cap [D_{i}^{1/2}]
-\sum_{x'\in |D_{i}^{1/2}|}2\ord'(\chi;x,D_{i})[\{x'\}]), \notag
\end{align}
where $x$ is the image of $x'$ in $D_{i}$ for $x'\in |D_{i}^{1/2}|$.
We consider the cartesian diagram
\begin{equation}
\xymatrix{
\cotx\times_{X}D_{i}^{1/2}\ar[r]^{F'_{i}} \ar[d]_{p_{i}^{\prime}} & \cotx\times_{X}D_{i} \ar[d]^{p_{i}}\\
D_{i}^{1/2}\ar[r]_{F_{i}} & D_{i}.
} \notag
\end{equation}
Since $[L_{i,\chi}']=F_{i*}'[L_{i,\chi}'']$ and $F_{i*}'p_{i}'^{*}[A]=p_{i}^{*}F_{i*}[A]$ in $CH_{2}(T^{*}X\times_{X}D_{i})$ for $[A]\in CH_{0}(D_{i}^{1/2})$ by the base change formula,
the assertion holds.

(iii) Since the sequence $0\rightarrow T^{*}_{D_{i}}X\rightarrow T^{*}X\times_{X}D_{i}
\rightarrow T^{*}D_{i}\rightarrow 0$ is exact,
we have
\begin{equation}
[\nori]=p_{1}^{*}(c_{1}(T^{*}D_{i})\cap [D_{i}]) \notag
\end{equation}
in $CH_{2}(T^{\ast}X\times_{X}D_{i})$ by the excess intersection formula.
We consider the sequence
\begin{equation}
0\rightarrow \dvr_{X}(-R_{\chi})\otimes_{\dvr_{X}} 
\dvr_{D_{i}}(\sum_{x\in |D_{i}|}\ord(\chi;x,D_{i})x|_{D_{i}})
\xrightarrow{\times\rsw(\chi)|_{D_{i}}}\Omega^{1}_{X}(\log D)|_{D_{i}}\xrightarrow{\res_{i}}\dvr_{D_{i}}\rightarrow 0. \notag
\end{equation}
Since $d=2$ and $i\in I_{\mII,\chi}$, the sequence is exact by (\ref{relloc}).
Since the image of the second entry by the multiplication by $\rsw(\chi)|_{D_{i}}$
is isomorphic to $\Omega_{D_{i}}^{1}(\log\bigcup |D_{i}|')$,
we have
\begin{align}
c_{1}(\dvr_{X}(-R_{\chi})|_{D_{i}})\cap [&D_{i}]+
\sum_{x\in |D_{i}|}\ord(\chi;x,D_{i})[\{x\}]
\notag \\
&=c_{1}(T^{*}D_{i}(\log \bigcup |D_{i}|'))\cap [D_{i}] \notag \\ 
&=c_{1}(T^{*}D_{i})\cap [D_{i}]+c_{1}(\dvr_{D_{i}}(\bigcup |D_{i}|'))\cap [D_{i}] \notag
\end{align}
in $CH_{0}(D_{i})$.
Hence the assertion holds.
\end{proof}

\begin{proof}[Proof of Theorem \ref{totalpullback}]
(i) We note that $SS^{\log}(X,U,\chi)$ is the union of $T^{*}_{X}X(\log D)$, $L_{i,\chi}$ for $i\in I_{\mW,\chi}$, 
and $T^{*}_{x}X(\log D)$ for $x\in |D|$ such that $s_{x}\neq 0$ by (\ref{logcharcycle}).
Since $L_{i,\chi}'=T^{*}_{D_{i}}X$ for $i\in I_{\mT,\chi}$
and $\tau_{D}^{-1}(T^{*}_{X}X(\log D))=C_{D}$ by Lemma \ref{pullzers}, 
it is sufficient to prove 
$T^{*}_{x}X\subset \tau_{D}^{-1}(SS^{\log}(X,U,\chi))$ for $x\in |D|-(|D|'\cup |Z_{\mII,\chi}|)$
such that $s_{x}=0$ and $t_{x}\neq 0$ by Lemma \ref{pullofli}.
Let $x$ be an element of $|D|-(|D|'\cup |Z_{\mII,\chi}|)$.
If $x\in D_{\mT,\chi}$, then $s_{x}=0$ and $t_{x}=0$ by (\ref{tx}), 
since $(X,U,\chi)$ is clean at $x$ and $\sharp(I_{x})=1$.
Hence we may assume $x\in Z_{\mI,\chi}$ and $I_{x}=\{1\}$.
Suppose that $s_{x}=0$.
Then we have $t_{x}=\sw(\chi|_{K_{1}})(\ord'(\chi;x,D_{1})
-\ord(\chi;x,D_{1}))$ by (\ref{tx}).
Since $\sw(\chi|_{K_{1}})>0$, 
we have $T^{*}_{x}X\subset \tau_{D}^{-1}(SS^{\log}(X,U,\chi))$
if $t_{x}\neq 0$ by Lemma \ref{pullofli} (i).
Hence the assertion holds.

(ii) Let $i$ be an element of $I_{\mI,\chi}$.
By Lemma \ref{pullofli} (i), we have
\begin{equation}
\tau_{D}^{!}([L_{i,\chi}])=[L_{i,\chi}']+\sum_{x\in |D_{i}|}(\ord'(\chi;x,D_{i})-\ord(\chi;x,D_{i}))[\spf] 
+p_{i}^{\ast}(c_{1}(\dvr_{X}(Z_{\mII,\chi})|_{D_{i}})\cap [D_{i}]) \notag
\end{equation} 
in $CH_{2}(\tau_{D}^{-1}(L_{i,\chi}))$.

Let $i$ be an element of $I_{\mII,\chi}$. 
Then we have $\tau_{D}^{-1}(L_{i,\chi})=T^{*}X\times_{X}D_{i}$ by Lemma \ref{pullofli} (ii).
By Lemma \ref{lemoftypeiip} and Lemma \ref{lemoftypeii} (i) and (ii), we have
\begin{align}
\tau_{D}^{!}([L_{i,\chi}])&=[L_{i,\chi}^{\prime}]+p_{i}^{\ast}(c_{1}(\dvr_{X}(D)|_{D_{i}})\cap[D_{i}]
\notag\\
&\qquad -c_{1}(\dvr_{X}(D-Z_{\mII,\chi})|_{D_{i}})\cap [D_{i}]+\sum_{x\in |D_{i}|}(\ord'(\chi;x,D_{i})-\ord(\chi;x,D_{i}))[\{x\}]) \notag \\
&=[L_{i,\chi}^{\prime}]+p_{i}^{\ast}(c_{1}(\dvr_{X}(Z_{\mII,\chi})|_{D_{i}})\cap[D_{i}]
+\sum_{x\in |D_{i}|}(\ord'(\chi;x,D_{i})-\ord(\chi;x,D_{i}))[\{x\}]) \notag
\end{align}
in $CH_{2}(\tau_{D}^{-1}(L_{i,\chi}))$ if $\chi|_{K_{i}}$ is of usual type and
\begin{align}
2\tau_{D}^{!}([L_{i,\chi}])&=[L_{i,\chi}^{\prime}]+2p_{i}^{\ast}(c_{1}(\dvr_{X}(D)|_{D_{i}})\cap[D_{i}]
\notag\\
&\qquad -c_{1}(\dvr_{X}(D-Z_{\mII,\chi})|_{D_{i}})\cap [D_{i}]+\sum_{x\in |D_{i}|}(\ord'(\chi;x,D_{i})-\ord(\chi;x,D_{i}))[\{x\}]) \notag \\
&=[L_{i,\chi}^{\prime}]+2p_{i}^{\ast}(c_{1}(\dvr_{X}(Z_{\mII,\chi})|_{D_{i}})\cap[D_{i}]
+\sum_{x\in |D_{i}|}(\ord'(\chi;x,D_{i})-\ord(\chi;x,D_{i}))[\{x\}]) \notag
\end{align}
in $CH_{2}(\tau_{D}^{-1}(L_{i,\chi}))$ if $\chi|_{K_{i}}$ is of exceptional type.

By Lemma \ref{lemint} (ii), Lemma \ref{pullzers}, and
the above computation of $\tau_{D}^{!}([L_{i,\chi}])$ for $i\in I_{\mW,\chi}$, we have
\begin{align}
\label{cmck}
\tau_{D}^{!}(&\Char^{\log}(X,U,\chi))-\Char^{K}(X,U,\chi)
\\
&=\sum_{i\in I_{\mII,\chi}}[T^{\ast}_{D_{i}}X]+\sum_{x\in |D|'}[T^{*}_{x}X]+\sum_{x\in |D|}s_{x}[T^{\ast}_{x}X]
+\sum_{i\in I_{\mW,\chi}}\sw(\chi|_{K_{i}})p_{i}^{\ast}(c_{1}(\dvr_{X}(Z_{\mII,\chi})|_{D_{i}})\cap [D_{i}])
\notag \\
&\qquad \qquad \qquad +\sum_{x\in |D|}\sum_{i\in I_{\mW,\chi,x}}\sw(\chi|_{K_{i}})
(\ord'(\chi ;x,D_{i})-\ord(\chi ;x,D_{i}))[T^{*}_{x}X]
-\sum_{x\in |D|}t_{x}[T^{*}_{x}X]. \notag
\end{align}
Since we have
\begin{equation}
\sum_{i\in I_{\mW,\chi}}\sw(\chi|_{K_{i}})p_{i}^{\ast}(c_{1}(\dvr_{X}(Z_{\mII,\chi})|_{D_{i}})\cap [D_{i}])=\sum_{i\in I_{\mII,\chi}}p_{i}^{\ast}(c_{1}(\dvr_{X}(R_{\chi})|_{D_{i}})\cap[D_{i}])\notag
\end{equation} 
in $CH_{2}(\tau_{D}^{-1}(SS^{\log}(X,U,\chi)))$, the difference (\ref{cmck}) is $0$ 
by (\ref{tx}) and Lemma \ref{lemoftypeii} (iii).
\end{proof}

\begin{cor}[Index formula]
\label{ndegindex}
Assume $d=2$.
If $X$ is proper over an algebraically closed field, then we have
\begin{equation}
\chi(X,j_{!}\mf)=(\mathrm{Char}^{K}(X,U,\chi),T^{\ast}_{X}X)_{T^{\ast}X}. \notag
\end{equation}
\end{cor}

\begin{proof}
The assertion follows from Theorem \ref{logindexformula} and Theorem \ref{totalpullback} (ii).
\end{proof}

\section{Homotopy invariance of characteristic cycle}
\label{shi}

Let $X$ be a smooth scheme purely of dimension $d$ 
over a perfect field $k$ of characteristic $p>0$.
Let $D$ be a divisor on $X$ with simple normal crossings 
and $\{D_{i}\}_{i\in I}$ the irreducible components of $D$.
Let $K_{i}$ be the local field at the generic point of $D_{i}$ for $i\in I$
and let $|D|$ be the set of closed points of $D$.
We put $U=X-D$ and let $j\colon U\rightarrow X$ be the open immersion.
Let $\mf_{0}$ and $\mf_{1}$ be smooth sheaves of $\Lambda$-modules of rank $1$ on $U$
and let $\chi_{i}\colon \pi_{1}^{\ab}(U)\rightarrow \Lambda^{\times}$ be the character
corresponding to $\mf_{i}$ for $i=0,1$.
We fix an inclusion $\psi\colon \Lambda^{\times}\rightarrow \mathbf{Q}/\mathbf{Z}$ and regard
$\chi_{i}$ as an element of $H^{1}_{\et}(U,\mathbf{Q}/\mathbf{Z})$ by $\psi$ for $i=0,1$.

\begin{prop} 
\label{prophi}
Assume $d=2$ and  
that $(X,U,\chi_{i})$ for $i=0,1$ are clean.
If $\Char^{K}(X,U,\chi_{0})=\Char^{K}(X,U,\chi_{1})$,
then we have 
\begin{equation}
\label{eqcc}
CC(j_{!}\mf_{0})=CC(j_{!}\mf_{1}). 
\end{equation}
\end{prop}

We note that 
we have $R_{\chi_{0}}=R_{\chi_{1}}$ and $R_{\chi_{0}}'=R_{\chi_{1}}'$
by Lemma \ref{lemint} (ii) and (\ref{saitoccro})
if $\Char^{K}(X,U,\chi_{0})=\Char^{K}(X,U,\chi_{1})$.

The following corollary holds immediately by using Lemma \ref{lemint} (ii), 
(\ref{tx}), (\ref{canlift}), and Proposition \ref{prophi}:
\begin{cor}
\label{corckcc}
Assume $d=2$ and that $(X,U,\chi_{i})$ for $i=0,1$ are clean.
If $I_{\mW,\chi_{0}}=I_{\mI,\chi_{0}}=I_{\mW,\chi_{1}}=I_{\mI,\chi_{1}}$, $R_{\chi_{0}}=R_{\chi_{1}}$,
and if $\ord'(\chi_{0};x,D_{i'})=\ord'(\chi_{1};x,D_{i'})$ for every closed point $x$ of $Z_{\chi_{0}}=Z_{\chi_{1}}$
and $i'\in I_{\mW,\chi_{0},x}=I_{\mW,\chi_{1},x}$,
then we have
\begin{equation}
\Char^{K}(X,U,\chi_{0})=\Char^{K}(X,U,\chi_{1}) \notag
\end{equation}
and
\begin{equation}
CC(j_{!}\mf_{0})=CC(j_{!}\mf_{1}). \notag
\end{equation}
\end{cor}
\vspace{0.2cm}

We consider $\tilde{X}=X\times_{k}\mathbf{A}_{k}^{1}$.  
We put $\tilde{D}=D\times_{k}\mathbf{A}_{k}^{1}\subset \tilde{X}$
and $\tilde{D}_{i}=D_{i}\times_{k}\mathbf{A}_{k}^{1}$ for $i\in I$.
Then $\tilde{D}$ is a divisor on $\tilde{X}$ with simple normal crossings whose irreducible components
are $\{\tilde{D}_{i}\}_{i\in I}$.
Let $\tilde{U}=U\times_{k}\mathbf{A}_{k}^{1}\subset \tilde{X}$ be the complement of $\tilde{D}$ in $\tilde{X}$ and
$\tilde{j}\colon \tilde{U}\rightarrow \tilde{X}$ the open immersion.
Let $X_{i}$ be the closed subscheme $X\times_{k}\{i\}\subset \tilde{X}$ for $i=0,1$
and let $h_{i}\colon X\rightarrow \tilde{X}$ be the composition of the inverse of the projection 
$\pr_{1}\colon X_{i}\rightarrow X$ and the closed immersion $X_{i}\rightarrow \tilde{X}$ for $i=0,1$.
Let $\tilde{p}\colon \tilde{X}\rightarrow X$ be the projection.

\begin{lem}
\label{lemcako}
Let $\tilde{V}$ be an open subscheme of $\tilde{X}$ containing $X_{0}\cup X_{1}$.
Let $\mg$ be a constructible complex of $\Lambda$-modules on $\tilde{V}$
and let $C\subset T^{*}X$ be a closed conical subset purely of dimension $d$.
We regard $T^{*}X\times_{k}\mathbf{A}_{k}^{1}$ as a subbundle of $T^{*}\tilde{X}$
by the injection $d\tilde{p}\colon T^{*}X\times_{k}\mathbf{A}_{k}^{1}\rightarrow T^{*}\tilde{X}$.
Assume that $SS(\mg|_{\tilde{V}})\subset T^{*}\tilde{V}$ is contained in $C\times_{k}\mathbf{A}^{1}_{k}$.
\begin{enumerate}
\item $h_{i}$ is properly $SS(\mg|_{\tilde{V}})$-transversal for $i=0,1$.
\item $CC(h_{0}^{*}\mg)=CC(h_{1}^{*}\mg)$.
\end{enumerate}
\end{lem}

\begin{proof}
We note that $SS(\mg|_{\tilde{V}})$ is purely of dimension $d+1$ by \cite[Theorem 1.3]{be}.
Since the projection $\tilde{p}\colon \tilde{X}\rightarrow X$ is smooth, 
the projection $\tilde{p}$ is $C$-transversal and $\tilde{p}^{\circ}C=C\times_{k}\mathbf{A}_{k}^{1}$ 
by \cite[Lemma 3.4.1]{sa4}.
Since the morphism $h_{i}$ is a section of $\tilde{p}$ for $i=0,1$
and the identity morphism $\id_{X}\colon X\rightarrow X$ is $C$-transversal,
the morphism $h_{i}$ is $\tilde{p}^{\circ}C$-transversal for $i=0,1$ by \cite[Lemma 3.4.3]{sa4}.
Let $\{C_{a}\}_{a}$ be the irreducible components of $C$.
Then $\{C_{a}\times_{k}\mathbf{A}_{k}^{1}\}_{a}$ are the irreducible components of 
$\tilde{p}^{\circ}C=C\times_{k}\mathbf{A}_{k}^{1}$.
Since $SS(\mg|_{\tilde{V}})\subset \tilde{p}^{\circ}C=C\times_{k}\mathbf{A}_{k}^{1}$ 
and $SS(\mg|_{\tilde{V}})$ is purely of dimension $d+1$,
the pull-back $h_{i}^{*}SS(\mg|_{\tilde{V}})$ is of dimension $d$ for $i=0,1$.
Hence the assertion (i) holds.

Since $SS(\mg|_{\tilde{V}})$ is a union of irreducible components of $\tilde{p}^{\circ}C=C\times_{k}\mathbf{A}_{k}^{1}$
and $h_{i}^{*}(C_{a}\times_{k}\mathbf{A}_{k}^{1})=C_{a}$ for $a$ and $i=0,1$,
the assertion (ii) holds by (i) and Theorem \ref{thmpull}.
\end{proof}

Let $s\ge 0$ be an integer such that the orders of the $p$-parts of 
$\chi_{0}$ and $\chi_{1}$ are $\le p^{s}$.
After shrinking $X$ if necessary, we take global sections
$a=(a_{s-1},\ldots,a_{0})$ and $a'=(a_{s-1}',\ldots, a_{0}')$ 
of $\fillog_{R_{\chi_{0}}}j_{*}W_{s}(\dvr_{U})$ and $\fillog_{R_{\chi_{1}}}j_{*}W_{s}(\dvr_{U})$
whose images in $H^{1}_{\et}(U,\mathbf{Q}/\mathbf{Z})$ are the $p$-parts of 
$\chi_{0}$ and $\chi_{1}$ respectively 
(see the remark after Definition \ref{deffilshf})
and we assume that $X=\Spec A$ is affine.
We put $\tilde{X}=X\times_{k}\mathbf{A}_{k}^{1}=\Spec A[T]$.
We define a global section $b$ of $\tilde{j}_{*}W_{s}(\dvr_{\tilde{U}})$ by
\begin{equation}
\label{tildfeq}
b=(a_{s-1}(1-T)^{p}, \ldots ,a_{i}(1-T)^{p^{s-i}}, \ldots ,a_{0}(1-T)^{p^{s}})
+(a_{s-1}'T^{p}, \ldots, a_{i}'T^{p^{s-i}}, \ldots ,a_{0}'T^{p^{s}}). 
\end{equation}
Let $\mg$ be the smooth sheaf of $\Lambda$-modules of rank $1$ on $\tilde{U}$ defined by the
Artin-Schreier-Witt equation $(F-1)(t)=b$.
Then we have $h_{i}^{*}\mg=\mf_{i}$ for $i=0,1$.
Let $\varphi\colon \pi_{1}^{\ab}(\tilde{U})\rightarrow \Lambda^{\times}$ be the character 
corresponding to $\mg$ and regard $\varphi$ 
as an element of $H^{1}_{\et}(\tilde{U},\mathbf{Q}/\mathbf{Z})$ by $\psi$.
Let $\tilde{K}_{i}$ be the local field at the generic point of $\tilde{D}_{i}$ for $i\in I$. 

\begin{lem}
\label{lemhswdt}
Let $i$ be an element of $I$
and assume $\sw(\chi_{0}|_{K_{i}})=\sw(\chi_{1}|_{K_{i}})$.
Then we have $\sw(\varphi|_{\tilde{K}_{i}})=\sw(\chi_{0}|_{K_{i}})$.
\end{lem}

\begin{proof}
By \cite[Theorem (9.1)]{ka1}, we have $\sw(\varphi|_{\tilde{K}_{i}})\ge \sw(\chi_{0}|_{K_{i}})$.
Since $b|_{\tilde{K}_{i}}\in \fillog_{n_{i}}W_{s}(\tilde{K}_{i})$, where $n_{i}=\sw(\chi_{0}|_{K_{i}})$,
we have $\sw(\varphi|_{\tilde{K}_{i}})\leq \sw(\chi_{0}|_{K_{i}})$.
Hence the assertion holds.
\end{proof}

\begin{lem}
\label{lemhrswcf}
Assume $\sw(\chi_{0}|_{K_{i'}})=\sw(\chi_{1}|_{K_{i'}})$ for every $i'\in I$
(then $R_{\varphi}=R_{\chi_{i}}\times_{k}\mathbf{A}_{k}^{1}$ for $i=0,1$ by Lemma \ref{lemhswdt}).
Then $\rsw(\varphi)$ is the image of
$-(F^{s-1}da)(1-T)^{p^{s}}-(F^{s-1}da')T^{p^{s}}$ in $\Gamma(Z_{\varphi},\Omega^{1}_{\tilde{X}}(\log \tilde{D})
(R_{\varphi})|_{Z_{\varphi}})$.
Consequently, we have 
$\rsw(\varphi)=\rsw(\chi_{0})(1-T)^{p^{s}}+\rsw(\chi_{1})T^{p^{s}}$ and
$dh_{i}^{\tilde{D}}(h_{i}^{*}\rsw(\varphi))=\rsw(\chi_{i})$ for $i=0,1$,
where $dh_{i}^{\tilde{D}}\colon (h_{i}^{*}\Omega^{1}_{\tilde{X}}(\log \tilde{D})(R_{\varphi}))|_{Z_{\chi_{i}}}
\rightarrow \Omega^{1}_{X}(\log D)(R_{\chi_{i}})|_{Z_{\chi_{i}}}$ is the morphism
yielded by $h_{i}$ for $i=0,1$.
\end{lem}

\begin{proof}
We put $n_{i'}=\sw(\varphi|_{\tilde{K}_{i'}})=\sw(\chi_{0}|_{K_{i'}})=\sw(\chi_{1}|_{K_{i'}})$ for $i'\in I$.
In $\grlog_{n_{i'}}\Omega_{\tilde{K}_{i'}}^{1}$ for $i'\in I_{\mW,\varphi}=I_{\mW,\chi_{0}}=I_{\mW,\chi_{1}}$,
we have
\begin{align}
-F^{s-1}db&=-\sum_{i=0}^{s-1}(a_{i}(1-T)^{p^{s-i}})^{p^{i}-1}d(a_{i}(1-T)^{p^{s-i}})
-\sum_{i=0}^{s-1}(a_{i}'T^{p^{s-i}})^{p^{i}-1}d(a_{i}'T^{p^{s-i}}) \notag \\
&= - (1-T)^{p^{s}}\sum_{i=0}^{s-1}{a_{i}}^{p^{i}-1}da_{i}
-T^{p^{s}}\sum_{i=0}^{s-1}a_{i}'^{p^{i}-1}da'_{i}. \notag 
\end{align}
Hence the assertion holds.
\end{proof}

\begin{lem}
\label{lemhclean}
Assume $\sw(\chi_{0}|_{K_{i'}})=\sw(\chi_{1}|_{K_{i'}})$ for every $i'\in I$
and that $(X,U,\chi_{i})$ for $i=0,1$ are clean.
Then the largest open subsheme of $\tilde{X}$ where
$(\tilde{X},\tilde{U},\varphi)$ is clean contains $X_{0}\cup X_{1}$.
\end{lem}

\begin{proof}
Since the set of points on $\tilde{X}$ where $(\tilde{X},\tilde{U},\varphi)$ is clean is open in $\tilde{X}$ by Lemma \ref{lemclean} (ii),
it is sufficient to prove that $(\tilde{X},\tilde{U},\varphi)$ is clean at $x_{0}=x\times_{k}\{0\}$ and 
$x_{1}=x\times_{k}\{1\}$
for every closed point $x$ of $Z_{\chi_{0}}=Z_{\chi_{1}}$.
By Lemma \ref{lemhrswcf}, the image of $(h_{i}^{*}\rsw(\varphi))_{x}$ by the morphism
$dh_{i x}^{\tilde{D}}\colon \Omega^{1}_{\tilde{X}}(\log \tilde{D})(R_{\varphi})_{x_{i}}
\otimes_{\dvr_{\tilde{X},x_{i}}}k(x)
\rightarrow \Omega^{1}_{X}(\log D)(R_{\chi_{i}})_{x}\otimes_{\dvr_{X,x}}k(x)$
induced by $h_{i}$ is $\rsw(\chi_{i})_{x}$ for $i=0,1$.
By the assumption that $(X,U,\chi_{i})$ for $i=0,1$ are clean, we have
$\rsw(\chi_{i})_{x}\neq 0$ in $\Omega^{1}_{X}(\log D)(R_{\chi_{i}})_{x}\otimes_{\dvr_{X,x}}k(x)$ for $i=0,1$.
Hence we have $\rsw(\varphi)_{x_{i}}\neq 0$ in $\Omega^{1}_{\tilde{X}}(\log \tilde{D})(R_{\varphi})_{x_{i}}
\otimes_{\dvr_{\tilde{X},x_{i}}}k(x_{i})$ for $i=0,1$,
which implies the cleanliness of $(\tilde{X},\tilde{U},\varphi)$ at $x_{i}$ for $i=0,1$.
\end{proof}

In the rest of this section, we prove Proposition \ref{prophi}.

\begin{lem}
\label{lemsuff}
Assume $d=2$ and let $x$ be a closed point of $X$. 
Let $\mg$ be a smooth sheaf of $\Lambda$-modules
of rank $1$ on $\tilde{U}$ such that $h_{i}^{*}\mg=\mf_{i}$ for $i=0,1$ and
let $\varphi\colon \pi_{1}^{\ab}(\tilde{U})\rightarrow \Lambda^{\times}$ be the character
corresponding to $\mg$. We regard $\varphi$ 
as an element of $H^{1}_{\et}(\tilde{U},\mathbf{Q}/\mathbf{Z})$ by $\psi$.
Then (\ref{eqcc}) holds in a neighborhood of $x$ if $\mg$ has an open
subscheme $\tilde{V}$ of $\tilde{X}$
satisfying the following conditions:
\begin{enumerate}
\item $\tilde{V}$ contains $X_{0}\cup X_{1}$
and $(\tilde{V},\tilde{V}\cap \tilde{U},\varphi|_{\tilde{V}\cap \tilde{U}})$ is clean.
\item $(\tilde{V},\tilde{V}\cap \tilde{U},\varphi|_{\tilde{V}\cap \tilde{U}})$ is strongly clean outside
$\tilde{x}=\{x\}\times_{k}\mathbf{A}_{k}^{1}$.
\item $\tau_{\tilde{D}}^{-1}(SS^{\log}(\tilde{V},\tilde{V}\cap \tilde{U},\varphi|_{\tilde{V}\cap \tilde{U}}))
\cap (T^{*}\tilde{X}\times_{\tilde{X}}\tilde{x})\subset T^{*}_{\tilde{x}}\tilde{X}$.
\end{enumerate}
\end{lem}

\begin{proof}
We replace $\tilde{X}$ by $\tilde{V}$.
By the condition (ii) and Lemma \ref{lemcharksc} (i), we have 
\begin{equation}
\tau_{\tilde{D}}^{-1}(SS^{\log}(\tilde{X},\tilde{U},\varphi))
=\bigcup_{I'\subset I}T^{*}_{\tilde{D}_{I'}}\tilde{X} \notag
\end{equation}
outside $\tilde{x}$, where $\tilde{D}_{I'}=\bigcap_{i'\in I'}\tilde{D}_{i'}$ for $I'\subset I$.
Further by the condition (iii), we have
\begin{equation}
\tau_{\tilde{D}}^{-1}(SS^{\log}(\tilde{X},\tilde{U},\varphi))
\subset \bigcup_{I'\subset I}T^{*}_{\tilde{D}_{I'}}\tilde{X}\cup T^{*}_{\tilde{x}}\tilde{X}. \notag
\end{equation}
Since $(\tilde{X},\tilde{U},\varphi)$ is clean by the condition (i),
we have 
\begin{equation}
SS(\tilde{j}_{!}\mg)\subset \bigcup_{I'\subset I}T^{*}_{\tilde{D}_{I'}}\tilde{X}\cup T^{*}_{\tilde{x}}\tilde{X}
\notag
\end{equation}
by Theorem \ref{proplogct}.
Hence we have $SS(\tilde{j}_{!}\mg)\subset C\times_{k}\mathbf{A}_{k}^{1}$ for a 
closed conical subset $C\subset T^{*}X$ purely of dimension $2$.
Since $h_{i}^{*}\tilde{j}_{!}\mg=j_{!}\mf_{i}$ by the assumption $h_{i}^{*}\mg=\mf_{i}$ for $i=0,1$, 
the assertion holds by Lemma \ref{lemcako} (ii).
\end{proof}

\begin{proof}[Proof of Proposition \ref{prophi}]
As is seen in the remark after the statement of Proposition \ref{prophi},
we have $R_{\chi_{0}}=R_{\chi_{1}}$ and $R_{\chi_{0}}'=R_{\chi_{1}}'$.
Let $x$ be a closed point of $X$.
Since the assertion is local, it is sufficient to prove the equality (\ref{eqcc}) in a neighborhood of $x$.
We may assume $I=I_{x}$. 
We put $n_{i'}=\sw(\chi_{0}|_{K_{i'}})=\sw(\chi_{1}|_{K_{i'}})$ 
and $s_{i'}'=\ord_{p}(n_{i'})$ for $i'\in I_{\mW,\chi_{0}}=I_{\mW,\chi_{1}}$
and we put $I=\{1,\ldots,r\}$ and $I_{\mW,\chi_{0}}=\{1, \ldots, r'\}$, where $r'\le r\le 2$.
We may assume $s_{1}'\le s_{r'}'$.

If $(X,U,\chi_{i})$ for $i=0,1$ are non-degenerate at $x$, then we have
$CC(j_{!}\mf_{i})=\Char^{K}(X,U,\chi_{i})$ for $i=0,1$ by Proposition \ref{propndeg} (i). 
Hence the equality (\ref{eqcc}) holds by the assumption 
$\Char^{K}(X,U,\chi_{0})=\Char^{K}(X,U,\chi_{1})$.
Thus we may assume $(X,U,\chi_{0})$ is not non-degenerate at $x$.
Then we have $I_{\mW,\chi_{0}}=I_{\mI,\chi_{0}}(\neq \emptyset)$ by Corollary \ref{corintdeg}.

Since the assertion is local, we may assume that $X=\Spec A$ is affine.
Since $CC(j_{!}\mf_{i})$ is stable under the replacement
of $\chi_{i}$ by the $p$-part of $\chi_{i}$ for $i=0,1$ by \cite[Theorem 0.1]{sy},
we may assume that the order of $\chi_{i}$ is a power of $p$ for $i=0,1$.
By Lemma \ref{lemsuff}, it is sufficient to prove that the sheaf $\mg$ on $\tilde{U}$ defined by 
$(F-1)(t)=b$ (\ref{tildfeq}) has an open subscheme $\tilde{V}$
of $\tilde{X}$ satisfying the conditions (i)--(iii) in Lemma \ref{lemsuff}.

By Lemma \ref{lemhclean}, we may replace $\tilde{X}$ by the largest open
subscheme of $\tilde{X}$ where $(\tilde{X},\tilde{U},\varphi)$ is clean.
Then, by Lemma \ref{lemhrswcf}, every open subschme $\tilde{V}$ of $\tilde{X}$ 
containing $X_{0}\cup X_{1}$ satisfies the conditions (i) and (iii) in Lemma \ref{lemsuff}.
Hence it is sufficient to construct an open subscheme $\tilde{V}$ of $\tilde{X}$ 
containing $X_{0}\cup X_{1}$ and satisfying the condition (ii) in Lemma \ref{lemsuff}.

Since $(X,U,\chi_{i})$ for $i=0,1$ are assumed to be clean
and $I_{\mW,\chi_{1}}=I_{\mW,\chi_{0}}=I_{\mI,\chi_{0}}=I_{\mI,\chi_{1}}$, 
we have
\begin{equation}
\label{eqsword}
\sum_{i'\in I_{\mI,\chi_{0}}}\sw(\chi_{0}|_{K_{i'}})\ord'(\chi_{0};x,D_{i'})
=\sum_{i'\in I_{\mI,\chi_{1}}}\sw(\chi_{1}|_{K_{i'}})\ord'(\chi_{1};x,D_{i'}) 
\end{equation}
by (\ref{tx}) and the assumption $\Char^{K}(X,U,\chi_{0})=\Char^{K}(X,U,\chi_{1})$.
Since $R_{\chi_{0}}=R_{\chi_{1}}$ and $I_{\mI,\chi_{0}}=I_{\mI,\chi_{1}}$, 
we have $\ord'(\chi_{0};x,D_{1})=\ord'(\chi_{1};x,D_{1})$ if $\sharp(I_{\mI,\chi_{0}})=1$.
If $\sharp(I_{\mI,\chi_{0}})=2$, then
the images of $\xi_{1}(\chi_{i})=\res_{1}\circ \times \res(\chi_{i})\colon \dvr_{X}(R_{\chi_{i}})|_{D_{1}}\rightarrow
\dvr_{D_{1}}$ (\ref{xichi}) for $i=0,1$ in $k(x)$ are not $0$ by Corollary \ref{corwitt}
and hence we have $\ord'(\chi_{0};x,D_{1})=\ord'(\chi_{1};x,D_{1})=1$ 
by Lemma \ref{lemordp}.
Since $R_{\chi_{0}}=R_{\chi_{1}}$, we have 
$\ord'(\chi_{0};x,D_{2})=\ord'(\chi_{1};x,D_{2})$ by (\ref{eqsword}).
Hence we have $\ord'(\chi_{0};x,D_{i'})=\ord'(\chi_{1};x,D_{i'})$ for every $i'\in I_{\mW,\chi_{0}}$
even if $\sharp(I_{\mI,\chi_{0}})=2$.
We put $r_{i'}=\ord'(\chi_{0};x,D_{i'})-(\sharp(I_{\mI,\chi_{0}}\cup I_{\mT,\chi_{0}})-1)$ 
for $i'\in I_{\mW,\chi_{0}}$.
Then the image of $\xi_{i'}(\chi_{i})$ is $\dvr_{D_{i'}}(-r_{i'}\cdot x)$ in a neighborhood of $x$
for $i'\in I_{\mW,\chi_{i}}$ and $i=0,1$ by Lemma \ref{lemordp}.
Hence we may assume that the image of $\xi_{i'}(\chi_{i})$ is $\dvr_{D_{i'}}(-r_{i'}\cdot x)$
for $i'\in I_{\mW,\chi_{i}}$ and $i=0,1$ by shrinking $X$ if necessary.

We consider the image of $\xi_{i'}(\varphi)$ for $i'\in I_{\mW,\varphi}$.
Let $(t_{1},t_{2})$ be a local coordinate system at $x$ such that $t_{i'}$ is a local equation of $D_{i'}$
for $i'\in I_{\mW,\chi_{0}}=I_{\mW,\chi_{1}}$.
By Lemma \ref{lemhswdt}, we have $I_{\mW,\varphi}=I_{\mW,\chi_{0}}=I_{\mW,\chi_{1}}$.
By shrinking $X$ if necessary, we may assume $D_{i'}=(t_{i'}=0)$ and $t_{i'}\in A$ for $i'\in I$.
We put $\{i',j'\}=\{1,2\}$ for $i'\in I_{\mW,\varphi}$.
Since the image of $\xi_{i'}(\chi_{i})$ is $\dvr_{D_{i'}}(-r_{i'}\cdot x)$ for $i'\in I_{\mW,\chi_{i}}$ 
and $i=0,1$, by shrinking $X$ if necessary, we have a divisor $\tilde{E}_{i'}=(t_{j'}^{r_{i'}}\tilde{u}_{i'}=0)$
for some $\tilde{u}_{i'}\in \Gamma(\tilde{D}_{i'},\dvr_{\tilde{D}_{i'}})$ such that
the image of $\xi_{i'}(\varphi)$ is $\dvr_{\tilde{D}_{i'}}(-\tilde{E}_{i'})$ for $i'\in I_{\mW,\varphi}$
and that the pul-back $h_{i,\tilde{D}_{i'}}^{*}\tilde{u}_{i'}$ by the base change 
$h_{i,\tilde{D}_{i'}}\colon D_{i'}\rightarrow \tilde{D}_{i'}$ of $h_{i}$ by $D_{i'}\rightarrow \tilde{X}$ 
is invertible in $\Gamma(D_{i'},\dvr_{D_{i'}})$ for $i'\in I_{\mW,\chi_{i}}$ and $i=0,1$
by Lemma \ref{lemhrswcf}.
We put $\tilde{V}=\bigcap_{i'\in I_{\mW,\varphi}}(\tilde{X}-(\tilde{u}_{i'}=0))$,
where $(\tilde{u}_{i'}=0)\subset D_{i'}$ for $i'\in I_{\mW,\varphi}$.
Then we have $\tilde{V}\supset X_{0}\cup X_{1}$.
Hence the assertion holds.
\end{proof}

\section{Characteristic cycle and canonical lifting}
\label{scccl}
\subsection{Main theorem and its corollaries}
\label{ssmthm}
Let $X$ be a smooth scheme purely of dimension $d$
over a perfect field $k$ of characteristic $p>0$.
Let $D$ be a divisor on $X$ with simple normal crossings and 
$\{D_{i}\}_{i\in I}$ the irreducible components of $D$.
We put $U=X-D$ and
let $j\colon U\rightarrow X$ be the open immersion.
We note that $j$ is an affine open immersion.
Let $\mf$ be a smooth sheaf of $\Lambda$-modules of rank $1$ on $U$.
Let $\chi\colon \pi_{1}^{\ab}(U)\rightarrow \Lambda^{\times}$ be the character corresponding to $\mf$
and regard $\chi$ as an element of $H^{1}_{\et}(U,\mathbf{Q}/\mathbf{Z})$
by an inclusion $\psi\colon \Lambda^{\times}\rightarrow \mathbf{Q}/\mathbf{Z}$.

\begin{thm}
\label{thmmain}
Assume $d=2$.
Then we have
\begin{equation}
\label{cceqck}
CC(j_{!}\mathcal{F})=\Char^{K}(X,U,\chi). 
\end{equation}
\end{thm}

We prove Theorem \ref{thmmain} in Subsection \ref{sspfsurf}.
In this section, we give and prove corollaries of Theorem \ref{thmmain}
by admitting the theorem. 

\begin{cor}
\label{corconj}
Conjecture \ref{conjcc} holds in the case where $X$ is purely of dimension $2$.
\end{cor}

\begin{proof}
The assertion holds by Theorem \ref{totalpullback} (ii) and Theorem \ref{thmmain}.
\end{proof}

\begin{cor}
\label{corss}
Assume $d=2$.
Then we have $SS(j_{!}\mf)=SS^{K}(X,U,\chi)$.
\end{cor}

\begin{proof}
Since $j$ is affine, the assertion holds by Lemma \ref{suppCC} and Theorem \ref{thmmain}.
\end{proof}

\begin{cor}[Milnor formula]
\label{cormilnor}
Assume $d=2$.
Let $C\subset T^{*}X$ be a closed conical subset purely of dimension $2$
containing $SS^{K}(X,U,\chi)$.
Let $(h,f)$ be a pair of an \'{e}tale morphism $h\colon W\rightarrow X$ and 
a morphism $f\colon W\rightarrow Y$ to a smooth curve $Y$ over $k$
and let $w\in W$ be at most $C$-isolated characteristic point of $f$. 
With the notation as in Theorem \ref{thmsmil}, we have
\begin{equation}
-\dimtot \phi_{w}(h^{\ast}j_{!}\mf,f)=(h^{\ast}\Char^{K} (X,U,\chi), df)_{T^{\ast}W,w}. \notag
\end{equation}
\end{cor}

\begin{proof}
By Corollary \ref{corss}, the sheaf $j_{!}\mf$ is micro-supported on $C$.
Hence the assertion holds by Theorem \ref{thmsmil} and Theorem \ref{thmmain}.
\end{proof}

\begin{cor}
\label{corcclass}
Assume $d=2$.
Let $f\colon X\rightarrow \Spec k$ be the structure morphism.
We put $\mk_{X}=Rf^{!}\Lambda$.
Let $C(j_{!}\mf)\in H^{0}(X,\mk_{X})$ be the characteristic class of $j_{!}\mf$ (\cite[Definition 2.1.1]{as3}).
Then we have 
\begin{equation}
C(j_{!}\mf)=(CC(j_{!}\mf),T^{\ast}_{X}X)_{T^{\ast}X} \notag
\end{equation} 
in $H^{4}(X,\Lambda(2))$,
where the right hand side denotes the intersection product of $CC(j_{!}\mf)$ with the zero section $T^{*}_{X}X$
and $C(j_{!}\mf)$ is regarded as an element of $H^{4}(X,\Lambda(2))$ by the isomorphism
$\Lambda(2)[4]\rightarrow \mk_{X}$ defined by the cycle class.
\end{cor}

\begin{proof}
We prove the equality
\begin{equation}
\label{eqccla}
C(j_{!}\mf)=(\Char^{K}(X,U,\chi),T^{*}_{X}X)_{T^{*}X}
\end{equation}
in $H^{4}(X,\Lambda(2))$.
Then the assertion holds by Theorem \ref{thmmain}.

Suppose that $(X,U,\chi)$ is clean.
By \cite[Theorem 3.7]{saj}, we have 
\begin{equation}
C(j_{!}\mf)=(\Char^{\log}(X,U,\chi),T^{\ast}_{X}X(\log D))_{T^{\ast}X(\log D)} \notag
\end{equation}
in $H^{4}(X,\Lambda(2))$.
By Theorem \ref{totalpullback} (ii), the right hand side is equal to
$(\Char^{K}(X,U,\chi),T^{\ast}_{X}X)_{T^{\ast}X}$
in $H^{4}(X,\Lambda (2))$.
Hence the equality (\ref{eqccla}) holds if $(X,U,\chi)$ is clean.

In general, let $f\colon X'\rightarrow X$ be the composition (\ref{seqcl}) of blow-ups at closed points over $Z_{\chi}$
such that $(X',f^{-1}(U),f^{*}\chi)$ is clean.
Then we have
\begin{align}
(\Char^{K}(X,U,\chi),T^{*}_{X}X)_{T^{*}X}&=
f_{*}(df^{!}\Char^{K}(X',f^{-1}(U),f^{*}\chi),T^{*}_{X}X\times_{X}X')_{T^{*}X\times_{X}X'}  \notag\\
&=f_{*}(\Char^{K}(X',f^{-1}(U),f^{*}\chi),T^{*}_{X'}X')_{T^{*}X'}. \notag
\end{align}
Since $(X',f^{-1}(U),f^{*}\chi)$ is clean, we have
\begin{equation}
C(f^{*}j_{!}\mf)=(\Char^{K}(X',f^{-1}(U),f^{*}\chi),T^{*}_{X'}X')_{T^{*}X'} \notag
\end{equation}
as proved above.
Since $f^{-1}(U)$ is isomorphic to $U$ via $f$,
we have $f_{*}C(f^{*}j_{!}\mf)=C(f_{*}f^{*}j_{!}\mf)=C(j_{!}\mf)$ by \cite[Proposition 2.1.6]{as3}.
Hence the equation (\ref{eqccla}) holds.
\end{proof}

\subsection{Proof of the main theorem}
\label{sspfsurf}

In this subsection, we prove Theorem \ref{thmmain}.
We assume that $X$ is purely of dimension $2$ throughout this subsection.

Let $t_{x}$ and $u_{x}$ be the multiplicities of $[\spf]$ in $\Char^{K}(X,U,\chi)$ and $CC(j_{!}\mf)$
respectively for a closed point $x$ of $D$.
By (\ref{saitoccro}) and (\ref{canlift}), it is sufficient to prove the equality
\begin{equation}
\label{tequalu}
t_{x}=u_{x}
\end{equation}
for every closed point $x$ of $D$ to prove the equality (\ref{cceqck}).

\begin{lem}
\label{sumtxux}
Assume that $X$ is projective over $k$.
Then we have $\sum_{x\in |D|}t_{x}=\sum_{x\in |D|}u_{x}$.
\end{lem}

\begin{proof}
Since $\Char^{K}(X,U,\chi)$ and $CC(j_{!}\mf)$ are equal outside the closed points $x$ of $D$
such that $(t_{x},u_{x})\neq (0,0)$ by (\ref{saitoccro}) and (\ref{canlift}),
the assertion holds by Theorem \ref{sindex} and Corollary \ref{ndegindex}.
\end{proof}

Theorem \ref{thmmain} is reduced to the clean case:

\begin{prop}
\label{propmain}
Let $x$ be a closed point of $D$ and assume that $(X,U,\chi)$ is clean. 
Then the equality (\ref{tequalu}) holds for $x$.
\end{prop}

We prove Theorem \ref{thmmain} by admitting Proposition \ref{propmain}.
Let $x$ be a closed point of $D$ and
let $X\rightarrow \bar{X}$ be a smooth compactification over $k$. 
By blowing up $\bar{X}$ outside $X$ if necessary, 
we may assume that $U$ is the complement in $\bar{X}$ of a divisor on $\bar{X}$ with simple normal crossings. 
Since $t_{x}$ and $u_{x}$ are defined \'{e}tale locally, by replacing $X$ by $\bar{X}$ if necessary, 
we may assume that $X$ is projective over $k$ to prove the equality (\ref{tequalu}).

Let $f\colon X^{\prime} \rightarrow X$ be a composition of finitely many blow-ups
at (non-clean) closed points lying above $D-\{x\}$ such that
$(X^{\prime},f^{-1}(U),f^{\ast}\chi)$ is clean outside $f^{-1}(x)$
(\cite[Theorem 4.1]{ka2}).
Since $t_{x}$ and $u_{x}$ are stable under the replacement of $X$ by $X^{\prime}$,
we may replace $X$ by $X^{\prime}$ and we may assume that $x$ is the unique
point on $X$ where $(X,U,\chi)$ is not clean.
Hence the equality (\ref{tequalu}) holds by Lemma \ref{sumtxux} and Proposition \ref{propmain}.
\vspace{0.2cm}

In the rest of this section, we prove Proposition \ref{propmain}.
Since $CC(j_{!}\mf)$ and $\Char^{K}(X,U,\chi)$ are compatible with the pull-back by
the projection $X_{\bar{k}}\rightarrow X$,
we may assume that $k$ is algebraically closed.
Further, since $CC(j_{!}\mf)$ is stable under the replacement
of $\chi$ by the $p$-part of $\chi$ by \cite[Theorem 0.1]{sy}
and so is $\Char^{K}(X,U,\chi)$, we may assume
that $\chi$ is of order $p^{s}$ for an integer $s\ge 0$.
Since the assertion is local, we may shrink $X$ to a neighborhood of $x$
and we may assume that $X=\Spec A$ is affine and $D_{i}=(t_{i}=0)$ for some $t_{i}\in A$ for $i\in I$.
Further, we may assume $I=I_{x}$.
By Proposition \ref{propndeg}, we may assume that $(X,U,\chi)$ is not non-degenerate at $x$
and $I=I_{\mW,\chi}$.
Then we have $I= I_{\mI,\chi}$ by Corollary \ref{corintdeg}.
We put $n_{i}=\sw(\chi|_{K_{i}})$, 
$s_{i}'=\ord_{p}(n_{i})$, and $n_{i}'=p^{-s_{i}'}n_{i}$ for $i\in I$.

After shrinking $X$ if necessary, we take a global section $a=(a_{s-1},\ldots, a_{0})$ of 
$\fillog_{R_{\chi}}j_{*}W_{s}(\dvr_{U})$ whose image in 
$H^{1}_{\et}(U,\mathbf{Q}/\mathbf{Z})$ is $\chi$.
By Lemma \ref{lemtone}, we have $s_{i}'<s$ for every $i\in I$, since $I=I_{\mI,\chi}$.
We give a proof of the equation (\ref{tequalu}) by dividing it into the following $2$ cases:
\begin{enumerate}
\renewcommand{\labelenumi}{(\alph{enumi})}
\item $\sharp(I)=1$.
\item $\sharp(I)=2$.
\end{enumerate}

In the case (a), we put $I=\{1\}$ and
$a_{s'_{1}}=c/t_{1}^{n_{1}^{\prime}}$ for $c\in \Gamma(X,\dvr_{X})$. 
By Lemma \ref{lemtone} and $\cform(\chi|_{K_{1}})\neq 0$, we have $c\neq 0$ in $\dvr_{D_{1},x}$.
Let $m$ be the valuation of $c$ in $\dvr_{D_{1},x}$.
Since $(X,U,\chi)$ is not non-degenerate at $x$, we have $m>0$
by Lemma \ref{lemwitt} (iv).
We note that if $s'_{1}=0$ then $m=1$ by $m>0$ and Lemma \ref{lemwitt} (iii).
We divide the case (a) into the following $3$ subcases:
\begin{itemize}
\item[($\text{a}_{1}$)] $m=1$.
\item[($\text{a}_{2}$)] $1< m\le n_{1}'$.
\item[($\text{a}_{3}$)] $m> n_{1}'$.
\end{itemize}

In the case (b), we put $I=\{1,2\}$ and we may assume $s_{1}'\le s_{2}'$.
We put $a_{s'_{1}}=u/t_{1}^{n_{1}^{\prime}}t_{2}^{p^{s_{2}'-s_{1}'}n_{2}^{\prime}}$ 
for $u\in \Gamma(X,\dvr_{X})$.
Since $(X,U,\chi)$ is assumed to be clean,
the element $u$ is invertible in $\dvr_{X,x}$ by Lemma \ref{lemwitt} (ii).
If $s_{1}'<s_{2}'$, then
we put $a_{s'_{2}}=c/t_{1}^{m}t_{2}^{n_{2}^{\prime}}$ for $c\in \Gamma(X,\dvr_{X})$ 
and $m\in \mathbf{Z}$ such that $-p^{s'_{2}}m\ge -n_{1}$.
Then, since $\cform(\chi|_{K_{2}})\neq 0$ and $\chi|_{K_{2}}$ is of type $\mI$,
we have $c\neq 0$ in $\dvr_{D_{2},x}$ by Lemma \ref{lemtone}.
Let $m^{\prime}$ be the valuation of $c$ in $\dvr_{D_{2},x}$ and put $M=m'-m$.
We divide the case (b) into the following $3$ subcases:
\begin{itemize}
\item[($\text{b}_{1}$)] $s_{1}'=s_{2}'$.
\item[($\text{b}_{2}$)] $s_{1}'<s_{2}'$ and $M \le n_{2}'$. 
\item[($\text{b}_{3}$)] $s_{1}'<s_{2}'$ and $M > n_{2}'$.
\end{itemize}

We give a format of our proof for the equality (\ref{tequalu}) in the cases 
($\text{a}_{1}$)--($\text{b}_{3}$).
By using Proposition \ref{prophi},
we first reduce the proof to the case where $X=\mathbf{A}_{k}^{2}=\Spec k[T_{1},T_{2}]$, 
$D_{i}=(T_{i}=0)$ for $i\in I$, and where $x$ is the origin.
More precisely, we prove by using Proposition \ref{prophi} that we may replace $a$ by another global section $a'$ of 
$\fillog_{R_{\chi}}j_{*}W_{s}(\dvr_{U})$ which is locally the image of an element
of $W_{s}(k[T_{1},T_{2}, T_{i}^{-1}; i\in I])$ and is concretely constructed.
Then we do the reduction by using the fact that $CC(j_{!}\mf)$ and $\Char^{K}(X,U,\chi)$
are defined \'{e}tale locally.

We canonically regard $X=\mathbf{A}_{k}^{2}$ as the complement 
$\Spec k[S_{1}/S_{0}, S_{2}/S_{0}]$ of the divisor $(S_{0}=0)$ in 
$\mathbf{P}^{2}_{k}=\Proj k[S_{0},S_{1},S_{2}]$ by the correspondence $S_{i}/S_{0}\mapsto T_{i}$
for $i=1,2$.
Then $x=(1:0:0)\in \mathbf{P}_{k}^{2}$.
Let $E\subset (S_{0}S_{1}S_{2}=0)$ be a divisor on $\mathbf{P}_{k}^{2}$ and
put $V=\mathbf{P}_{k}^{2}-E$.
Let $j'\colon V\rightarrow \mathbf{P}_{k}^{2}$ be the open immersion.
Then we construct a smooth sheaf $\mg$ of $\Lambda$-modules of rank $1$ on $V$ 
by concretely constructing a global section $b$ of $j'_{*}W_{s}(\dvr_{V})$
and by defining $\mg$ by $(F-1)(t)=b$.
Let $\varphi$ be the character $\pi_{1}^{\ab}(V)\rightarrow \Lambda^{\times}$ 
corresponding to $\mg$ and regard $\varphi$
as an element of $H^{1}_{\et}(V,\mathbf{Q}/\mathbf{Z})$ by $\psi$.
Let $\Sigma$ be the finite set of closed points of $E$
where $(\mathbf{P}_{k}^{2},V,\varphi)$ is not non-degenerate.
We note that if $(\mathbf{P}_{k}^{2},V,\varphi)$ is clean 
then we have $CC(j'_{!}\mg)=\Char^{K}(\mathbf{P}_{k}^{2},V,\varphi)$ outside
$\Sigma$ by Proposition \ref{propndeg} (i).
Let $t_{y}'$ and $u'_{y}$ be the coefficients of the fiber $[T^{*}_{y}\mathbf{P}_{k}^{2}]$ at $y$
in $\Char^{K}(\mathbf{P}_{k}^{2},V,\varphi)$ and $CC(j'_{!}\mg)$ for a closed point $y$ of $E$.
We prove that $\mg$ satisfies the following conditions:
\begin{enumerate}
\item $(\mathbf{P}_{k}^{2},V,\varphi)$ is clean.
\item In the case ($\text{a}_{1}$), we have 
$x\in \Sigma$ and $(t'_{y},u'_{y})=(t_{x},u_{x})$ for every $y\in \Sigma$. 
In the other cases, we have $CC(j'_{!}\mg)=\Char^{K}(\mathbf{P}_{k}^{2},V,\varphi)$ outside $x$ and
$(t_{x}',u_{x}')=(t_{x},u_{x})$.
\end{enumerate}

In the case ($\text{a}_{1}$), we have
\begin{equation}
\sharp(\Sigma)\cdot t_{x}=\sum_{y\in \Sigma}t'_{y}=\sum_{y\in \Sigma}u'_{y}=\sharp(\Sigma)\cdot u_{x} \notag
\end{equation}
by the condition (ii) and Lemma \ref{sumtxux},
since the equality $CC(j'_{!}\mg)=\Char^{K}(\mathbf{P}_{k}^{2},V,\varphi)$ holds outside $\Sigma$
by the condition (i) and Proposition \ref{propndeg} (i).
Since $(x\in) \Sigma$ is not empty by the condition (ii), we obtain the equality (\ref{tequalu}).

In the other cases, we have $t_{x}=t'_{x}=u'_{x}=u_{x}$ by the condition (ii) and Lemma \ref{sumtxux}.
Hence the equality (\ref{tequalu}) holds.
\vspace{0.2cm}

We give the proofs in the order of ($\text{a}_{1}$), ($\text{a}_{2}$), ($\text{b}_{1}$), ($\text{a}_{3}$), 
($\text{b}_{2}$), and ($\text{b}_{3}$).
By the above format, it is sufficient to reduce each proof to the case where $X=\mathbf{A}_{k}^{2}$, 
$D_{i}=(T_{i}=0)$ for $i\in I$, and where $x$ is the origin,
construct a smooth sheaf $\mg$ of $\Lambda$-modules of rank $1$ on the complement $V$ of a divisor 
$E\subset (S_{0}S_{1}S_{2}=0)$ on $\mathbf{P}^{2}_{k}$, and 
prove that $\mg$ satisfies the conditions (i) and (ii) above.
Let $U_{i}$ be the complement of $(S_{i}=0)$ in $\mathbf{P}_{k}^{2}$ for $i=0,1,2$.

($\text{a}_{1}$)
Let $a'=(a_{s-1}',\ldots,a_{0}')$ be the global section of $j_{*}W_{s}(\dvr_{U})$
defined by $a_{i}'=t_{2}^{m}/t_{1}^{n_{1}'}$ if $i=s'_{1}$, $a_{i}'=t_{2}/t_{1}^{n_{1}}$ if $i=0$, 
and $a_{i}'=0$ if $i\neq 0, s'_{1}$.
Then the Swan conductor $\sw(\chi|_{K_{1}})$ and the total dimension $\dt(\chi|_{K_{1}})$
are stable under the replacement of $a$ by $a'$ by Corollary \ref{cortw}.
Especially $I=I_{\mW,\chi}=I_{\mI,\chi}$ is stable under the replacement of $a$ by $a'$.
The condition that $(X,U,\chi)$ is clean at $x$ is stable
under the replacement of $a$ by $a'$ by Lemma \ref{lemwitt} (iii),
so is $[L_{1,\chi}']$ by Lemma \ref{lemint} (ii),
and so is $\ord'(\chi;x,D_{1})$ by Lemma \ref{stabord}.
Hence $\Char^{K}(X,U,\chi)$ and $CC(j_{!}\mf)$ are stable in a neighborhood of $x$ under the replacement of $a$ by $a'$
by Corollary \ref{corckcc}. 
Thus we may replace $a$ by $a'$.
Since $\Char^{K}(X,U,\chi)$ and $CC(j_{!}\mf)$ are defined \'{e}tale locally,
we may assume $X=\mathbf{A}_{k}^{2}$, $D=(T_{1}=0)$,
and that $x$ is the origin.

Let $c_{1},\ldots ,c_{n_{1}-1}$ be $n_{1}-1$ different elements of $k^{\times}$.
We put $E=(S_{1}=0)$.
Let $\mathcal{G}$ be the smooth sheaf of $\Lambda$-modules of rank $1$ on $V$
defined by $(F-1)(t)=b$ where $b=(b_{s-1},\ldots ,b_{0})$ is the global section of 
$j'_{*}W_{s}(\dvr_{V})$ defined by
\begin{equation}
b_{i}=\begin{cases}
(S_{2}\prod_{i'=1}^{n_{1}^{\prime}-1}(S_{2}-c_{i'}S_{0}))/S_{1}^{n_{1}'} & (i=s'_{1}) \\
(S_{2}\prod_{i'=1}^{n_{1}-1}(S_{2}-c_{i'}S_{0}))/S_{1}^{n_{1}} & (i=0) \\
0 & (i\neq 0,s'_{1}).
\end{cases}
\notag
\end{equation}

We prove that $\mg$ satisfies the condition (i).
It is sufficient to prove that $(\mathbf{P}_{k}^{2},V,\varphi)$ is clean at
every closed point of $E$ by Lemma \ref{lemclean} (iii).
We put $x_{i'}=(1:0:c_{i'})$ for $i'=1, \ldots ,n_{1}^{\prime}-1$ and 
$\Sigma'=\{x,x_{1},\ldots,x_{n_{1}'-1}\}$.
We put $z=(0:0:1)$.
Then $b_{s_{1}'}S_{1}^{n_{1}'}/S_{0}^{n_{1}'}$ is invertible in $\dvr_{\mathbf{P}_{k}^{2},y}$
for every closed point $y$ of $U_{0}\cap E$ outside $\Sigma'$.
Further $(S_{0}=0)\cap E=\{z\}$ and $b_{s_{1}'}S_{1}^{n_{1}'}/S_{2}^{n_{1}'}$ is invertible in 
$\dvr_{\mathbf{P}_{k}^{2},z}$.
Hence $(\mathbf{P}_{k}^{2},V,\varphi)$ is clean outside $\Sigma'$ by Lemma \ref{lemwitt} (ii).
Since $(S_{1}/S_{0}, b_{0}S_{1}^{n_{1}}/S_{0}^{n_{1}})$ is a local coordinate system at 
every $y\in \Sigma'$,
the triple $(\mathbf{P}_{k}^{2},V,\varphi)$ is clean even at every $y\in \Sigma'$ by Lemma \ref{lemwitt} (iii). 
Hence $\mg$ satisfies the condition (i).

We prove that $\mg$ satisfies the condition (ii).
By Corollary \ref{cortw}, we have $\sw(\varphi|_{K_{1}})=n_{1}=\sw(\chi|_{K_{1}})$ 
and $\dt(\varphi|_{K_{1}})=n_{1}+1=\dt(\chi|_{K_{1}})$.
Hence the character $\varphi|_{K_{1}}$ is of type $\mI$.
By Lemma \ref{lemwitt} (iv), we have $\Sigma=\Sigma'$.
Since $\varphi|_{K_{1}}$ is of type $\mI$,
we have $\ord'(\varphi;y,E)=\ord'(\chi;x,D_{1})$ for every $y\in \Sigma=\Sigma'$
by Lemma \ref{stabord}.
Since $(\mathbf{P}_{k}^{2},V,\varphi)$ is clean, 
we have $(t_{x}',u_{x}')=(t_{x},u_{x})$ by Corollary \ref{corckcc}
and we have $t_{y}'=t_{x}$ for every $y\in \Sigma$ by (\ref{tx}).
By Lemma \ref{lemint} (ii) and (\ref{canlift}), 
there exists a family $\{(V_{y},y)\}_{y\in \Sigma}$ of pointed open subschemes 
of $\mathbf{P}_{k}^{2}$ isomorphic to each other such that
$\Char^{K}(V_{y},V_{y}\cap V,\varphi|_{V_{y}\cap V})=
\Char^{K}(V_{x},V_{x}\cap V,\varphi|_{V_{x}\cap V})$ for every $y\in \Sigma$
if we identify $(V_{y},y)$ with $(V_{x},x)$.
Hence we have $u_{y}'=u_{x}'$ for every $y\in \Sigma$ by
Proposition \ref{prophi}.
Thus $\mg$ satisfies the condition (ii).

($\text{a}_{2}$)
As in the proof of the case ($\text{a}_{1}$),
we may assume $a_{i}=t_{2}^{m}/t_{1}^{n_{1}'}$ if $i=s'_{1}$, $a_{i}=t_{2}/t_{1}^{n_{1}}$ if $i=0$, 
and $a_{i}=0$ if $i\neq 0, s'_{1}$.
Further we may assume $X=\mathbf{A}_{k}^{2}$, $D_{1}=(T_{1}=0)$,
and that $x$ is the origin.

Let $c_{1},\ldots ,c_{n_{1}-1}$ be $n_{1}-1$ different elements of $k^{\times}$.
We put $E=(S_{1}=0)$.
Let $\mathcal{G}$ be the smooth sheaf of $\Lambda$-modules of rank $1$ on $V$
defined by $(F-1)(t)=b$ where $b=(b_{s-1},\ldots ,b_{0})$ is the global section of 
$j'_{*}W_{s}(\dvr_{V})$ 
defined by
\begin{equation}
b_{i}=\begin{cases}
(S_{2}^{m}\prod_{i'=1}^{n_{1}^{\prime}-m}(S_{2}-c_{i'}S_{0}))/S_{1}^{n_{1}'} & (i=s'_{1}) \\
(S_{2}\prod_{i'=1}^{n_{1}-1}(S_{2}-c_{i'}S_{0}))/S_{1}^{n_{1}} & (i=0) \\
0 & (i\neq 0,s'_{1}).
\end{cases}
\notag
\end{equation}
Then, similarly as the proof of ($\text{a}_{1}$), the sheaf $\mg$ satisfies the condition (i).

We prove that $\mg$ satisfies the condition (ii).
As in the proof of ($\text{a}_{1}$), 
we have $(\sw(\varphi|_{K_{1}}),\dt(\varphi|_{K_{1}}))=(\sw(\chi|_{K_{1}}),\dt(\chi|_{K_{1}}))$ 
and $\Sigma=\{x,x_{1},\ldots,x_{n_{1}'-m}\}$,
where $x_{i'}=(1:0:c_{i'})$ for $i'=1, \ldots ,n_{1}^{\prime}-m$.
Hence $\varphi|_{K_{1}}$ is of type $\mI$.
By the case ($\text{a}_{1}$), we have $t'_{y}=u'_{y}$ for every $y\in \Sigma-\{x\}$.
Hence we have $CC(j'_{!}\mg)=\Char^{K}(\mathbf{P}_{k}^{2},V,\varphi)$ outside $x$.
Since $\chi|_{K_{1}}$ and $\varphi|_{K_{1}}$ are of type $\mI$, 
we have $\ord'(\varphi;x,D_{1})=\ord'(\chi;x,D_{1})$ by Lemma \ref{stabord}.
Since $I_{\mW,\chi}=I_{\mW,\varphi}=I_{\mI,\chi}=I_{\mI,\varphi}=\{1\}$, 
$\sw(\chi|_{K_{1}})=\sw(\varphi|_{K_{1}})$, and
both $(X,U,\chi)$ and $(\mathbf{P}_{k}^{2},V,\varphi)$ are clean,
we have $\Char^{K}(X,U,\chi)=\Char^{K}(X,U,\varphi|_{U})$ and
$CC(j_{!}\mf)=CC(j'_{!}\mg|_{X})$ in a neighborhood of $x$
by Corollary \ref{corckcc}.
Hence we have $(t'_{x},u'_{x})=(t_{x},u_{x})$.
Therefore $\mg$ satisfies the condition (ii).

($\text{b}_{1}$) 
Let $a'=(a_{s-1}',\ldots,a_{0}')$ be the global section of $j_{*}W_{s}(\dvr_{U})$
defined by $a_{i}'=1/t_{1}^{n_{1}'}t_{2}^{n_{2}'}$ if $i=s'_{1}$ and $a_{i}'=0$ if $i\neq s'_{1}$.
By Corollary \ref{cortw}, the Swan conductor $\sw(\chi|_{K_{i}})$ and the total dimension $\dt(\chi|_{K_{i}})$
are stable under the replacement of $a$ by $a'$ for $i=1,2$.
Especially $I=I_{\mW,\chi}=I_{\mI,\chi}$ is stable under the replacement of $a$ by $a'$.
The condition that $(X,U,\chi)$ is clean at $x$ is stable
under the replacement of $a$ by $a'$ by Lemma \ref{lemwitt} (ii),
so is $[L'_{i,\chi}]$ for $i=1,2$ by Lemma \ref{lemint} (ii),
and so is $\ord'(\chi;x,D_{i})$ for $i=1,2$ by Lemma \ref{stabord}.
Hence $\Char^{K}(X,U,\chi)$ and $CC(j_{!}\mf)$ are stable under the replacement of $a$ by $a'$
by Corollary \ref{corckcc}.
Thus we may replace $a$ by $a'$.
Since $\Char^{K}(X,U,\chi)$ and $CC(j_{!}\mf)$ are defined \'{e}tale locally,
we may assume $X=\mathbf{A}_{k}^{2}$, $D_{i}=(T_{i}=0)$ for $i=1,2$,
and that $x$ is the origin.

Let $c_{1},\ldots ,c_{n_{1}+n_{2}}$ be $n_{1}+n_{2}$ different elements of $k^{\times}$.
We put $E=(S_{1}S_{2}=0)$.
Let $\mathcal{G}$ be the smooth sheaf of $\Lambda$-modules of rank $1$ on $V$
defined by $(F-1)(t)=b$ where $b=(b_{s-1},\ldots ,b_{0})$ is the global section
of $j'_{*}W_{s}(\dvr_{V})$ defined by
\begin{equation}
b_{i}=\begin{cases}
(\prod_{i'=1}^{n_{1}^{\prime}+n_{2}^{\prime}}((S_{1}+S_{2})-c_{i'}S_{0}))/
S_{1}^{n_{1}^{\prime}}S_{2}^{n_{2}^{\prime}} & (i=s_{1}') \\
(\prod_{i'=1}^{n_{1}+n_{2}}((S_{1}+S_{2})-c_{i'}S_{0}))/S_{1}^{n_{1}}S_{2}^{n_{2}} & (i=0) \\
0 & (i\neq 0,s_{1}').
\end{cases}
\notag
\end{equation}

We prove that $\mg$ satisfies the condition (i).
It is sufficient to prove that $(\mathbf{P}_{k}^{2},V,\varphi)$ is clean at
every closed point of $E$ by Lemma \ref{lemclean} (iii).
We put $x_{i'}=(1:0:c_{i'})$ and $y_{i'}=(1:c_{i'}:0)$ for $i'=1, \ldots ,n_{1}^{\prime}+n_{2}'$
and $\Sigma_{1}=\{x_{1},\ldots,x_{n_{1}'+n_{2}'}\}$ and $\Sigma_{2}=\{y_{1},\ldots,y_{n_{1}'+n_{2}'}\}$.
Then $b_{s_{1}'}S_{2}^{n_{2}'}/S_{1}^{n_{2}'}$ is invertible in $\dvr_{\mathbf{P}_{k}^{2},y}$
for every closed point of $U_{1}\cap E$ outside $\Sigma_{2}$
and $b_{s_{1}'}S_{1}^{n_{1}'}/S_{2}^{n_{1}'}$ is invertible in $\dvr_{\mathbf{P}_{k}^{2},y}$
for every closed point of $U_{2}\cap E$ outside $\Sigma_{1}$.
Further $(S_{1}=0)\cap (S_{2}=0)=\{x\}$ and $b_{s_{1}'}S_{1}^{n_{1}'}S_{2}^{n_{2}'}/S_{0}^{n_{1}'+n_{2}'}$ is invertible in $\dvr_{\mathbf{P}_{k}^{2},x}$.
Hence $(\mathbf{P}_{k}^{2},V,\varphi)$ is clean outside $\Sigma_{1}\cup \Sigma_{2}$
by Lemma \ref{lemwitt} (ii).
Since $(S_{1}/S_{0},b_{0}S_{1}^{n_{1}}S_{2}^{n_{2}}/S_{0}^{n_{1}+n_{2}})$ is 
a local coordinate system at every $y\in \Sigma_{1}$
and $(S_{2}/S_{0}, b_{0}S_{1}^{n_{1}}S_{2}^{n_{2}}/S_{0}^{n_{1}+n_{2}})$ is 
a local coordinate systen at every $y\in \Sigma_{2}$,
the triple $(\mathbf{P}_{k}^{2},V,\varphi)$ is clean even at every $y\in \Sigma_{1}\cup \Sigma_{2}$
by Lemma \ref{lemwitt} (iii).
Hence $\mg$ satisfies the condition (i).

We prove that $\mg$ satisfies the condition (ii).
By Corollary \ref{cortw}, we have $\sw(\varphi|_{K_{i'}})=n_{i'}=\sw(\chi|_{K_{i'}})$ 
and $\dt(\varphi|_{K_{i'}})=n_{i'}+1=\dt(\chi|_{K_{i'}})$ for $i'=1,2$.
Hence the character $\varphi|_{K_{i'}}$ is of type $\mI$ for $i'=1,2$.
By Lemma \ref{lemwitt} (iv), we have $\Sigma=\Sigma_{1}\cup\Sigma_{2}\cup\{x\}$.
By the case ($\text{a}_{1}$), we have $t'_{y}=u'_{y}$ for every $y\in \Sigma_{1}\cup\Sigma_{2}$.
Hence we have $CC(j'_{!}\mg)=\Char^{K}(\mathbf{P}_{k}^{2},V,\varphi)$ outside $x$.
Since $\varphi|_{K_{i'}}$ is of type $\mI$ for $i'=1,2$,
we have $\ord'(\varphi;x,D_{i'})=\ord'(\chi;x,D_{i'})$ for $i'=1,2$ by Lemma \ref{stabord}.
Since $I_{\mW,\chi}=I_{\mW,\varphi}=I_{\mI,\chi}=I_{\mI,\varphi}=\{1,2\}$, 
$\sw(\chi|_{K_{i'}})=\sw(\varphi|_{K_{i'}})$ for $i'=1,2$,
and both $(X,U,\chi)$ and $(\mathbf{P}_{k}^{2},V,\varphi)$ are clean, 
we have $\Char^{K}(X,U,\chi)=\Char^{K}(X,U,\varphi|_{U})$ and
$CC(j_{!}\mf)=CC(j'_{!}\mg|_{X})$ in a neighborhood of $x$ by Corollary \ref{corckcc}.
Hence we have $(t'_{x},u'_{x})=(t_{x},u_{x})$.
Therefore $\mg$ satisfies the condition (ii).

($\text{a}_{3}$) 
As in the proof of the case ($\text{a}_{1}$),
we may assume $a_{i}=t_{2}^{m}/t_{1}^{n_{1}'}$ if $i=s'_{1}$, $a_{i}=t_{2}/t_{1}^{n_{1}}$ if $i=0$, 
and $a_{i}=0$ if $i\neq 0, s'_{1}$.
Further we may assume $X=\mathbf{A}_{k}^{2}$, $D_{1}=(T_{1}=0)$,
and that $x$ is the origin.

Let $M'\ge m$ be an integer prime to $p$ and put $M=p^{s_{1}'}M'$.
Let $c_{1},\ldots ,c_{n_{1}+M-1}$ be $n_{1}+M-1$ different elements of $k^{\times}$.
We put $E=(S_{0}S_{1}=0)$.
Let $\mathcal{G}$ be the smooth sheaf of $\Lambda$-modules of rank $1$ on $V$
defined by $(F-1)(t)=b$ where $b=(b_{s-1},\ldots ,b_{0})$ is the global section of 
$j'_{*}W_{s}(\dvr_{V})$ defined by
\begin{equation}
b_{i}=\begin{cases}
(S_{2}^{m}\prod_{i'=1}^{n_{1}^{\prime}+M'-m}((S_{0}+S_{1})-c_{i'}S_{2}))/S_{0}^{M'}S_{1}^{n_{1}'}
& (i=s_{1}') \\
(S_{2}\prod_{i'=1}^{n_{1}+M-1}((S_{0}+S_{1})-c_{i'}S_{2}))/S_{0}^{M}S_{1}^{n_{1}}
& (i=0) \\
0 & (i\neq 0,s_{1}').
\end{cases}
\notag
\end{equation}

We prove that $\mg$ satisfies the condition (i).
It is sufficient to prove that $(\mathbf{P}_{k}^{2},V,\varphi)$ is clean at
every closed point of $E$ by Lemma \ref{lemclean} (iii).
We put $x_{i'}=(1:0:c_{i'}^{-1})$ and $z_{i'}=(0:1:c_{i'}^{-1})$ for $i'=1,\ldots,n_{1}'+M'-m$
and $z=(0:1:0)$ and $w=(0:0:1)$.
We put $\Sigma_{1}=\{x,x_{1},\ldots,x_{n_{1}'+M'-m}\}$ and
$\Sigma_{2}=\{z,z_{1},\ldots,x_{n_{1}'+M'-m}\}$.
Then $b_{s_{1}'}S_{1}^{n_{1}'}/S_{0}^{n_{1}'}$ is invertible in $\dvr_{\mathbf{P}_{k}^{2},y}$ 
for every closed point $y$ of $U_{0}\cap E$ outside $\Sigma_{1}$ 
and $b_{s_{1}'}S_{0}^{M'}/S_{1}^{M'}$ is invertible in $\dvr_{\mathbf{P}_{k}^{2},y}$ 
for every closed point $y$ of $U_{1}\cap E$ outside $\Sigma_{2}$.
Further $(S_{0}=0)\cap (S_{1}=0)=\{w\}$ and $b_{s_{1}'}S_{0}^{M'}S_{1}^{n_{1}'}/S_{2}^{M'+n_{1}'}$
is invertible in $\dvr_{\mathbf{P}_{k}^{2},w}$.
Hence $(\mathbf{P}_{k}^{2},V,\varphi)$ is clean outside $\Sigma_{1}\cup\Sigma_{2}$ by Lemma \ref{lemwitt} (ii).
Since $(S_{1}/S_{0},b_{0}S_{1}^{n_{1}}/S_{0}^{n_{1}})$ is a local coordinate system at 
every $y\in \Sigma_{1}$
and $(S_{0}/S_{1}, b_{0}S_{0}^{M}/S_{1}^{M})$ is a local coordinate systen at
every $y\in \Sigma_{2}$,
the triple $(\mathbf{P}_{k}^{2},V,\varphi)$ is clean even at every $y\in \Sigma_{1}\cup \Sigma_{2}$ 
by Lemma \ref{lemwitt} (iii).
Hence $\mg$ satisfies the condition (i).

We prove that $\mg$ satisfies the condition (ii).
Let $K_{0}$ be the local field at the generic point of $(S_{0}=0)$.
By Corollary \ref{cortw}, we have $\sw(\varphi|_{K_{1}})=n_{1}=\sw(\chi|_{K_{1}})$ 
and $\dt(\varphi|_{K_{1}})=n_{1}+1=\dt(\chi|_{K_{1}})$.
Further, we have $\sw(\varphi|_{K_{0}})=M$ and $\dt(\varphi|_{K_{0}})=M+1$.
Hence the character $\varphi|_{K_{i'}}$ is of type $\mI$ for $i'=0,1$.
By Lemma \ref{lemwitt} (iv), we have $\Sigma=\Sigma_{1}\cup\Sigma_{2}\cup\{w\}$.
By the case ($\text{a}_{1}$), we have $t'_{y}=u'_{y}$ for every $y\in \Sigma_{1}\cup\Sigma_{2}-\{x,z\}$.
We have $t'_{z}=u'_{z}$ by the case ($\text{a}_{2}$) and
we have $t'_{w}=u'_{w}$ by the case ($\text{b}_{1}$).
Hence we have $CC(j'_{!}\mg)=\Char^{K}(\mathbf{P}_{k}^{2},V,\varphi)$ outside $x$.
Since $\varphi|_{K_{1}}$ is of type $\mI$,
we have $\ord'(\varphi;x,D_{1})=\ord'(\chi;x,D_{1})$ by Lemma \ref{stabord}.
Since both $(X,U,\chi)$ and $(\mathbf{P}_{k}^{2},V,\varphi)$ are clean, 
we have $\Char^{K}(X,U,\chi)=\Char^{K}(X,U,\varphi|_{U})$ and
$CC(j_{!}\mf)=CC(j'_{!}\mg|_{X})$ in a neighborhood of $x$ by Corollary \ref{corckcc}.
Hence we have $(t'_{x},u'_{x})=(t_{x},u_{x})$.
Thus $\mg$ satisfies the condition (ii).

($\text{b}_{2}$)
As in the proof of ($\text{b}_{1}$), we may assume 
$a'_{i}=t_{1}^{M}/t_{2}^{n_{2}'}$ if $i=s_{2}'$,
$a_{i}=1/t_{1}^{n_{1}'}t_{2}^{p^{s_{2}'-s_{1}'}n_{2}'}$ if $i=s'_{1}$, 
and $a_{i}'=0$ if $i\neq s'_{1}, s_{2}'$.
Further we may assume $X=\mathbf{A}_{k}^{2}$, $D_{i}=(T_{i}=0)$ for $i\in I$,
and that $x$ is the origin.

Let $c_{1},\ldots,c_{n_{1}+n_{2}}$ be $n_{1}+n_{2}$ different elements of $k^{\times}$.
We put $n_{2}''=p^{-s_{1}'}n_{2}$ and $E=(S_{1}S_{2}=0)$.
Let $\mathcal{G}$ be the smooth sheaf of $\Lambda$-modules of rank $1$ on $V$
defined by $(F-1)(t)=b$ where $b=(b_{s-1},\ldots ,b_{0})$ is the global section of 
$j'_{*}W_{s}(\dvr_{V})$ defined by
\begin{equation}
b_{i}=\begin{cases}
(S_{1}^{M}\prod_{i'=1}^{n_{2}^{\prime}-M}((S_{1}+S_{2})-c_{i'}S_{0}))/S_{2}^{n_{2}^{\prime}} & (i=s_{2}') \\
(\prod_{i'=1}^{n_{1}^{\prime}+n_{2}''}((S_{1}+S_{2})-c_{i'}S_{0}))/S_{1}^{n_{1}^{\prime}}S_{2}^{n_{2}''} & (i=s_{1}') \\
(\prod_{i'=1}^{n_{1}+n_{2}}((S_{1}+S_{2})-c_{i'}S_{0}))/S_{1}^{n_{1}}S_{2}^{n_{2}} & (i=0) \\
0 & (i\neq 0, s_{1}',s_{2}').
\end{cases}
\notag
\end{equation}

We prove that $\mg$ satisfies the condition (i).
It is sufficient to prove that $(\mathbf{P}_{k}^{2},V,\varphi)$ is clean at
every closed point of $E$ by Lemma \ref{lemclean} (iii).
We put $x_{i'}=(1:0:c_{i'})$ for $i'=1,\ldots,n_{1}'+n_{2}''$,
$y_{i'}=(1:c_{i'}:0)$ for $i'=1, \ldots ,n_{2}'-M$,
$z=(0:1:0)$, and $w=(0:0:1)$.
We put $\Sigma_{1}=\{x_{1},\ldots,x_{n_{1}'+n_{2}''}\}$ and 
$\Sigma_{2}=\{y_{1},\ldots,y_{n_{2}'-M}\}$.
Then $b_{s_{2}'}S_{2}^{n_{2}'}/S_{1}^{n_{2}'}$ is
invertible in $\dvr_{\mathbf{P}_{k}^{2},y}$ for every closed point $y$ 
of $U_{1}\cap E$ outside $\Sigma_{2}$.
Further $b_{s_{1}'}S_{1}^{n_{1}'}S_{2}^{n_{2}''}/S_{0}^{n_{1}'+n_{2}''}$ is invertible in $\dvr_{\mathbf{P}_{k}^{2},y}$
for every closed point $y$ of $U_{0}\cap (S_{1}=0)$ outside $\Sigma_{1}$.
Since $(S_{0}=0)\cap (S_{1}=0)=\{w\}$ and 
$b_{s_{1}'}S_{1}^{n_{1}'}/S_{2}^{n_{1}'}$ is invertible in $\dvr_{\mathbf{P}_{k}^{2},w}$,
the triple $(\mathbf{P}_{k}^{2},V,\varphi)$ is clean outside $\Sigma_{1}\cup\Sigma_{2}$
by Lemma \ref{lemwitt} (ii).
Since $(S_{i''}/S_{0},b_{0}S_{1}^{n_{1}}S_{2}^{n_{2}}/S_{0}^{n_{1}+n_{2}})$ is 
a local coordinate system at every $y\in \Sigma_{i''}$ for $i''=1,2$,
the triple $(\mathbf{P}_{k}^{2},V,\varphi)$ is clean even at every $y\in \Sigma_{1}\cup \Sigma_{2}$
 by Lemma \ref{lemwitt} (iii).
Hence $\mg$ satisfies the condition (i).

We prove that $\mg$ satisfies the condition (ii).
By Corollary \ref{cortw}, we have $\sw(\varphi|_{K_{i'}})=n_{i'}=\sw(\chi|_{K_{i'}})$ 
and $\dt(\varphi|_{K_{i'}})=n_{i'}+1=\dt(\chi|_{K_{i'}})$ for $i'=1,2$.
Hence the character $\varphi|_{K_{i'}}$ is of type $\mI$ for $i'=1,2$.
Since we have $\sharp(I_{y})=1$ for every closed point $y$ of $E$ except $x$,
we have $CC(j'_{!}\mg)=\Char^{K}(\mathbf{P}_{k}^{2},V,\varphi)$ outside $x$ by the case (a).
Since $\varphi|_{K_{i'}}$ is of type $\mI$ for $i'=1,2$,
we have $\ord'(\varphi;x,D_{i'})=\ord'(\chi;x,D_{i'})$ for $i'=1,2$ by Lemma \ref{stabord}.
Since $I_{\mW,\chi}=I_{\mW,\varphi}=I_{\mI,\chi}=I_{\mI,\varphi}=\{1,2\}$,
$\sw(\chi|_{K_{i'}})=\sw(\varphi|_{K_{i'}})$ for $i'=1,2$,
and both $(X,U,\chi)$ and $(\mathbf{P}_{k}^{2},V,\varphi)$ are clean, 
we have $\Char^{K}(X,U,\chi)=\Char^{K}(X,U,\varphi|_{U})$ and
$CC(j_{!}\mf)=CC(j'_{!}\mg|_{X})$ in a neighborhood of $x$ by Corollary \ref{corckcc}.
Hence we have $(t'_{x},u'_{x})=(t_{x},u_{x})$.
Therefore $\mg$ satisfies the condition (ii).

($\text{b}_{3}$)
As in the proof of the case ($\text{b}_{2}$), we may assume
$a_{i}=t_{1}^{M}/t_{2}^{n_{2}'}$ if $i=s_{2}'$,
$a_{i}=1/t_{1}^{n_{1}'}t_{2}^{p^{s_{2}'-s_{1}'}n_{2}'}$ if $i=s'_{1}$, 
and $a_{i}=0$ if $i\neq s'_{1}, s_{2}'$.
Further we may assume $X=\mathbf{A}_{k}^{2}$, $D_{i}=(T_{i}=0)$ for $i\in I$,
and that $x$ is the origin.

Let $N^{\prime}\ge M$ be an integer prime to $p$ and put $N=p^{s'_{2}}N^{\prime}$.
Let $c_{1}, \ldots ,c_{n_{1}+n_{2}+N}, d_{1}, \ldots ,d_{n_{1}+n_{2}+N}$ be 
$2(n_{1}+n_{2}+N)$ different elements of $k^{\times}$
such that $c_{1}/d_{1},\ldots,c_{n_{1}+n_{2}+N}/d_{n_{1}+n_{2}+N}$ are different.
We put $n_{2}''=p^{-s_{1}'}n_{2}$, $N''=p^{-s_{1}'}N$, and $E=(S_{0}S_{1}S_{2}=0)$.
Let $\mathcal{G}$ be the smooth sheaf of $\Lambda$-modules of rank $1$ on $V$
defined by $(F-1)(t)=b$ where $b=(b_{s-1},\ldots ,b_{0})$ is the global section of 
$j'_{*}W_{s}(\dvr_{V})$ 
defined by
\begin{equation}
b_{i}=\begin{cases}
(S_{1}^{M}\prod_{i'=1}^{n_{2}^{\prime}+N^{\prime}-M}(S_{1}-c_{i'}S_{2}-d_{i'}S_{0}))/S_{0}^{N^{\prime}}S_{2}^{n_{2}^{\prime}} & (i=s_{2}') \\
(\prod_{i'=1}^{n_{1}^{\prime}+n_{2}''+N''}
(S_{1}-c_{i'}S_{2}-d_{i'}S_{0}))/S_{0}^{N''}
S_{1}^{n_{1}^{\prime}}S_{2}^{n_{2}''}
& (i=s_{1}') \\
(\prod_{i'=1}^{n_{1}+n_{2}+N}(S_{1}-c_{i'}S_{2}-d_{i'}S_{0}))/S_{0}^{N}S_{1}^{n_{1}}S_{2}^{n_{2}} & (i=0) \\
0 & (i\neq 0, s_{1}',s_{2}').
\end{cases}
\notag
\end{equation}

We prove that $\mg$ satisfies the condition (i).
It is sufficient to prove that $(\mathbf{P}_{k}^{2},V,\varphi)$ is clean at
every closed point of $E$ by Lemma \ref{lemclean} (iii).
We put $x_{i'}=(1:0:-d_{i'}/c_{i'})$ for $i'=1,\ldots,n_{1}'+n_{2}''+N''$, 
$y_{i'}=(1:d_{i'}:0)$ and $z_{i'}=(0:1:c_{i'}^{-1})$ for $i'=1, \ldots ,n_{2}'+N'-M$,
$z=(0:1:0)$, and $w=(0:0:1)$.
We put $\Sigma_{1}=\{x_{1},\ldots,x_{n_{1}'+n_{2}''+N''}\}$,
$\Sigma_{2}=\{y_{1},\ldots,y_{n_{2}'+N'-M}\}$, and
$\Sigma_{3}=\{z_{1},\ldots,z_{n_{2}'+N'-M}\}$.
Then $b_{s_{2}'}S_{0}^{N'}S_{2}^{n_{2}'}/S_{1}^{n_{2}'+N'}$ is 
invertible in $\dvr_{\mathbf{P}_{k}^{2},y}$ for every closed point $y$ 
of $U_{1}\cap E$ outside $\Sigma_{2}\cup \Sigma_{3}$.
Further $b_{s_{1}'}S_{1}^{n_{1}'}S_{2}^{n_{2}''}/S_{0}^{n_{1}'+n_{2}''}$ is invertible 
in $\dvr_{\mathbf{P}_{k}^{2},y}$ for every closed point $y$ of $U_{0}\cap (S_{1}=0)$ outside $\Sigma_{1}$.
Since $(S_{0}=0)\cap (S_{1}=0)=\{w\}$ and $b_{s_{1}'}S_{0}^{N''}S_{1}^{n_{1}'}/S_{2}^{N''+n_{1}'}$
is invertible in $\dvr_{\mathbf{P}_{k}^{2},w}$, 
the triple $(\mathbf{P}_{k}^{2},V,\varphi)$ is clean outside $\Sigma_{1}\cup\Sigma_{2}\cup\Sigma_{3}$
by Lemma \ref{lemwitt} (ii).
Since $(S_{i''}/S_{0},b_{0}S_{1}^{n_{1}}S_{2}^{n_{2}}/S_{0}^{n_{1}+n_{2}})$ is 
a local coordinate system at every $y\in \Sigma_{i''}$ for $i''=1,2$ and
$(S_{0}/S_{1}, b_{0}S_{0}^{N}S_{2}^{n_{2}}/S_{1}^{n_{2}+N})$ is a local coordinate system at
every $y\in \Sigma_{3}$,
the triple $(\mathbf{P}_{k}^{2},V,\varphi)$ is clean even at
every $y\in \Sigma_{1}\cup\Sigma_{2}\cup\Sigma_{3}$ by Lemma \ref{lemwitt} (iii).
Hence $\mg$ satisfies the condition (i).

We prove that $\mg$ satisfies the condition (ii).
Let $K_{0}$ be the local field at the generic point of $(S_{0}=0)$.
Since we have $\sw(\varphi|_{K_{i'}})=n_{i'}=\sw(\chi|_{K_{i'}})$ 
and $\dt(\varphi|_{K_{i'}})=n_{i'}+1=\dt(\chi|_{K_{i'}})$ for $i'=1,2$ by Corollary \ref{cortw},
the character $\varphi|_{K_{i'}}$ is of type $\mI$ for $i'=1,2$.
Further, by Corollary \ref{cortw}, we have $\sw(\varphi|_{K_{0}})=N$ and $\dt(\varphi|_{K_{0}})=N+1$.
Hence the character $\varphi|_{K_{0}}$ is of type $\mI$.
Since $\sharp(I_{y})=1$ for every closed point of $E$ except $x,z,w$,
we have $CC(j'_{!}\mg)=\Char^{K}(\mathbf{P}_{k}^{2},V,\varphi)$ outside $\{x,z,w\}$ 
by the case (a).
By the case ($\text{b}_{1}$), we have 
$CC(j'_{!}\mg)=\Char^{K}(\mathbf{P}_{k}^{2},V,\varphi)$ in a neighborhood of $z$.
Further, by the case ($\text{b}_{2}$), we have
$CC(j'_{!}\mg)=\Char^{K}(\mathbf{P}_{k}^{2},V,\varphi)$ in a neighborhood of $w$.
Hence we have $CC(j'_{!}\mg)=\Char^{K}(\mathbf{P}^{2}_{k},V,\varphi)$ outside $x$.
Since $\varphi|_{K_{i'}}$ is of type $\mI$ for $i'=1,2$,
we have $\ord'(\varphi;x,D_{i'})=\ord'(\chi;x,D_{i'})$ for $i'=1,2$ by Lemma \ref{stabord}.
Since both $(X,U,\chi)$ and $(\mathbf{P}_{k}^{2},V,\varphi)$ are clean, 
we have $\Char^{K}(X,U,\chi)=\Char^{K}(X,U,\varphi|_{U})$ and $CC(j_{!}\mf)=CC(j'_{!}\mg|_{X})$ 
in a neighborhood of $x$ Corollary \ref{corckcc}.
Hence we have $(t_{x}',u_{x}')=(t_{x},u_{x})$.
Thus $\mg$ satisfies the condition (ii).


\end{document}